\documentclass{aart}

\usepackage{titlesec}
\titleformat{\subsection}[runin]{\normalfont\bfseries}{\thesubsection.}{.5em}{}[.]\titlespacing{\subsection}{0pt}{2ex plus .1ex minus .2ex}{.8em}
\titleformat{\subsubsection}[runin]{\normalfont\itshape}{\thesubsubsection.}{.3em}{}[.]\titlespacing{\subsubsection}{0pt}{1ex plus .1ex minus .2ex}{.5em}

\usepackage[labelfont=sc,font=small,labelsep=period]{caption}
\usepackage[letterpaper, hmargin=1in, top=1.5in, bottom=1.9in, footskip=0.7in]{geometry}

\usepackage{amsmath} %[intlimits]
\usepackage{amssymb}
\usepackage{amsfonts}
\usepackage{latexsym}
\usepackage{amsthm}
\usepackage{amsxtra}
\usepackage{amscd}
\usepackage{bbm}
\usepackage{mathrsfs}
\usepackage{bm}

\usepackage{graphicx, color}
\usepackage[nottoc,notlof,notlot]{tocbibind} 
\usepackage{cite} % Enabling `[1-4]` - style citing
\usepackage{xspace} %command to insert a space after a latex command

\numberwithin{equation}{section}
\numberwithin{figure}{section}
\theoremstyle{plain} %plain, definition, remark
\newtheorem{theorem}{Theorem}[section]
\newtheorem*{theorem*}{Theorem}
\newtheorem{lemma}[theorem]{Lemma}
\newtheorem*{lemma*}{Lemma}
\newtheorem{corollary}[theorem]{Corollary}
\newtheorem*{corollary*}{Corollary}
\newtheorem{proposition}[theorem]{Proposition}
\newtheorem*{proposition*}{Proposition}
\newtheorem{definition}[theorem]{Definition}
\newtheorem*{definition*}{Definition}

\theoremstyle{definition} %plain, definition, remark

\newtheorem*{example*}{Example}
\newtheorem{remark}[theorem]{Remark}

\newtheorem*{remark*}{Remark}
\newtheorem*{remarks*}{Remarks}

\newcommand{\f}[1]{\boldsymbol{\mathrm{#1}}} %bold
\newcommand{\bb}{\mathbb} %blackboard bold
\renewcommand{\cal}{\mathcal} 
 
\newcommand{\fra}{\mathfrak} 
\newcommand{\ol}[1]{\overline{#1} \!\,} %overline
\newcommand{\wh}{\widehat}
\newcommand{\wt}{\widetilde}

\renewcommand{\P}{\mathbb{P}}
\newcommand{\E}{\mathbb{E}}
\newcommand{\R}{\mathbb{R}}
\newcommand{\C}{\mathbb{C}}
\newcommand{\N}{\mathbb{N}}
\newcommand{\Z}{\mathbb{Z}}

\newcommand{\me}{\mathrm{e}}
\newcommand{\ii}{\mathrm{i}}
\newcommand{\dd}{\mathrm{d}}
\newcommand{\col}{\mathrel{\vcenter{\baselineskip0.75ex \lineskiplimit0pt \hbox{.}\hbox{.}}}}
\newcommand*{\deq}{\mathrel{\vcenter{\baselineskip0.65ex \lineskiplimit0pt \hbox{.}\hbox{.}}}=}
\newcommand*{\eqd}{=\mathrel{\vcenter{\baselineskip0.65ex \lineskiplimit0pt \hbox{.}\hbox{.}}}}

\renewcommand{\leq}{\leqslant}
\renewcommand{\geq}{\geqslant}
\renewcommand{\epsilon}{\varepsilon}

\newcommand{\ceil}[1]  {\lceil  {#1} \rceil}

\newcommand{\ind}[1]{\f 1 (#1)}
\newcommand{\indb}[1]{\f 1 \pb{#1}}
\newcommand{\indB}[1]{\f 1 \pB{#1}}
\newcommand{\indbb}[1]{\f 1 \pbb{#1}}

\newcommand{\p}[1]{({#1})}
\newcommand{\pb}[1]{\bigl({#1}\bigr)}
\newcommand{\pB}[1]{\Bigl({#1}\Bigr)}
\newcommand{\pbb}[1]{\biggl({#1}\biggr)}
\newcommand{\pBB}[1]{\Biggl({#1}\Biggr)}

\newcommand{\qB}[1]{\Bigl[{#1}\Bigr]}
\newcommand{\qbb}[1]{\biggl[{#1}\biggr]}

\newcommand{\h}[1]{\{{#1}\}}
\newcommand{\hb}[1]{\bigl\{{#1}\bigr\}}

\newcommand{\hbb}[1]{\biggl\{{#1}\biggr\}}
\newcommand{\hBB}[1]{\Biggl\{{#1}\Biggr\}}

\newcommand{\abs}[1]{\lvert #1 \rvert}
\newcommand{\absb}[1]{\bigl\lvert #1 \bigr\rvert}
\newcommand{\absB}[1]{\Bigl\lvert #1 \Bigr\rvert}
\newcommand{\absbb}[1]{\biggl\lvert #1 \biggr\rvert}
\newcommand{\absBB}[1]{\Biggl\lvert #1 \Biggr\rvert}

\newcommand{\norm}[1]{\lVert #1 \rVert}

\newcommand{\normB}[1]{\Bigl\lVert #1 \Bigr\rVert}
\newcommand{\normbb}[1]{\biggl\lVert #1 \biggr\rVert}

\DeclareMathOperator{\tr}{Tr}

\DeclareMathOperator{\re}{Re}
\DeclareMathOperator{\im}{Im}

\DeclareMathOperator{\dist}{dist}

\newcommand{\e}{\epsilon}

\newcommand{\be}{\begin{equation}}
\newcommand{\ee}{\end{equation}}

\renewcommand{\le}{\leq}
\renewcommand{\ge}{\geq}

\title{The Local Semicircle Law for a General Class of Random Matrices}

\author{
L\'aszl\'o Erd\H os${}^1$\thanks{Partially supported
by SFB-TR 12 Grant of the German Research Council. On leave from Institute of Mathematics, University of Munich, Germany.} \quad
Antti Knowles${}^2$\thanks{Partially supported by NSF grant DMS-0757425.} \quad
Horng-Tzer Yau${}^3$\thanks{Partially supported
by NSF grants  DMS-0804279 and Simons  Investigator Award.}  \quad
Jun Yin${}^4$\thanks{Partially supported by NSF grants DMS-1001655 and DMS-1207961.} \\\\\\
\normalsize{Institute of Science and Technology Austria, Am Campus 1, A-3400 Klosterneuburg, Austria} \\
\normalsize{lerdos@ist.ac.at ${}^1$} \\ \\
\normalsize{Courant Institute, New York University, 251 Mercer Street, New York, NY 10012, USA} \\
\normalsize{knowles@cims.nyu.edu ${}^2$} \\ \\
\normalsize{Department of Mathematics, Harvard University, Cambridge MA 02138, USA} \\
\normalsize{htyau@math.harvard.edu ${}^3$}  \\  \\
\normalsize{Department of Mathematics, University of Wisconsin, Madison, WI 53706, USA} \\
\normalsize{jyin@math.uwisc.edu ${}^4$} \\ \\ 
}

\date{May 24, 2013}

\begin{document}

\maketitle

\begin{abstract}
We consider a general class of $N\times N$ random matrices
whose entries $h_{ij}$ are independent up to a symmetry
constraint, but not necessarily identically distributed.
Our main result is a local semicircle law which improves previous results 
\cite{EYY} both in the bulk and at the edge.  The error bounds are given in terms of the basic small
parameter of the model, $\max_{i,j} \E \abs{h_{ij}}^2$. 
As a consequence, we prove the universality of the local $n$-point correlation functions
in the bulk spectrum for a class of matrices whose entries do not have comparable variances,
including random band matrices with band width  $W\gg N^{1-\e_n}$ with some $\e_n>0$
 and with a negligible
 mean-field component.
In addition, we provide a coherent and pedagogical proof of the local semicircle
law, streamlining and strengthening previous arguments from \cite{EYY, EYYrigi, EKYY1}.
\end{abstract}

\medskip
\medskip

% {\bf AMS Subject Classification (2010):} 15B52, 82B44
% 
% \medskip


\newpage

\section{Introduction} \label{sec:intro}

Since the pioneering work \cite{W} of Wigner in the fifties,
random matrices have played a fundamental role in modelling
complex systems. The basic example is the Wigner matrix
ensemble, consisting of $N\times N$ symmetric or Hermitian 
matrices $H = (h_{ij})$ whose matrix entries are identically
distributed random variables that are independent up to the symmetry constraint
 $H=H^*$. From a physical point of view,
these matrices represent Hamilton operators
of disordered mean-field quantum systems, where the
quantum transition rate from state $i$ to state $j$ is given
by the entry $h_{ij}$. 

A central problem in the theory or random matrices
is to establish the \emph{local universality of the spectrum}.
Wigner observed that the distribution of the distances 
between consecutive eigenvalues (the gap distribution) in complex physical
systems follows a universal pattern. The Wigner-Dyson-Gaudin-Mehta conjecture,
formalized in \cite{M}, states that this gap distribution is universal in the sense that it depends only
on the symmetry class of the matrix, but is otherwise
independent of the details of the distribution of the matrix entries.
This conjecture has recently been established for all symmetry classes in a series of works \cite{ESY4, EYYrigi,
EKYY2}; an alternative approach
 was  given in \cite{TV} for the special Wigner Hermitian case. 
The general approach of \cite{ESY4, EYYrigi,
EKYY2} to prove universality
consists of three steps: (i) establish a local semicircle law for the density of eigenvalues; 
(ii) prove universality of Wigner matrices with a small Gaussian
component by analysing the convergence of Dyson Brownian motion to local equilibrium; (iii) remove the small Gaussian
component by comparing Green functions of Wigner ensembles 
with a few matching moments.
For an overview of recent results and this three-step
strategy, see \cite{EYBull}.

Wigner's vision was not restricted to Wigner matrices.
 In fact, he predicted that universality should hold
for any quantum system, described by a large Hamiltonian $H$, of sufficient complexity. 
In order to make such complexity mathematically tractable, one typically replaces
the detailed structure of $H$ with a statistical description. In this phenomenological model, $H$ is drawn from a random ensemble whose distribution mimics the true complexity.
One prominent example
 where random matrix statistics are expected
to hold is the random Schr\"odinger
operator in the delocalized regime. The random Schr\"odinger operator differs greatly from Wigner matrices in that most of its entries vanish. It describes a model with \emph{spatial structure}, in contrast to the mean-field Wigner matrices where all matrix entries are of comparable size. In order to address the question of universality of general disordered quantum systems, and in particular to probe Wigner's vision, one therefore has to break the mean-field permutational symmetry
of Wigner's original model, and hence to allow the distribution of $h_{ij}$ to depend
on $i$ and $j$ in a nontrivial fashion.  For example, if the matrix entries are labelled by a
discrete torus $\bb T \subset\Z^d$ on the $d$-dimensional lattice,
then the distribution of $h_{ij}$ may depend on the Euclidean distance $\abs{i - j}$
between sites $i$ and $j$, thus introducing a nontrivial spatial
structure into the model. If $h_{ij} = 0$ for $\abs{i - j} > 1$ we essentially obtain the random Schr\"odinger operator. A random Schr\"odinger operator models a physical system with a \emph{short-range} interaction, in contrast to the infinite-range, mean-field interaction described by Wigner matrices. More generally, we may consider a \emph{band matrix}, characterized
by the property that $h_{ij}$ becomes negligible if $\abs{i - j}$ exceeds
a certain parameter, $W$, called the \emph{band width}, describing the range of the interaction. 
Hence, by varying the band width $W$, band matrices naturally interpolate between mean-field
Wigner matrices and random Schr\"odinger operators; see \cite{Spe} for
an overview. 

For definiteness, let us focus on the case of a one-dimensional band matrix $H$. A fundamental conjecture, supported by nonrigorous supersymmetric arguments as well as numerics \cite{Fy}, is that the local spectral statistics of $H$ are governed by random matrix statistics for large $W$ and by Poisson statistics for small $W$. This transition is in the spirit of the Anderson metal-insulator
transition \cite{Fy, Spe}, and is conjectured to be sharp around the critical value $W = \sqrt{N}$.
In other words, if $W \gg \sqrt{N}$, we expect the universality results of \cite{EYY, EYY2, EYYrigi} to hold. In addition to a transition in the local spectral statistics, an accompanying transition is conjectured to occur in the behaviour \emph{localization length} of the eigenvectors of $H$, whereby in the large-$W$ regime they are expected to be completely delocalized and in the small-$W$ regime exponentially localized. 
The localization length for band matrices was recently investigated in great detail in \cite{EKYY3}.

Although the Wigner-Dyson-Gaudin-Mehta conjecture was originally stated for 
Wigner matrices, the methods of \cite{ESY4, EYYrigi,
EKYY2} also apply to certain ensembles 
with independent but not identically distributed entries, which however retain the mean-field character of Wigner matrices.
More precisely, they yield universality provided the variances 
$$
  s_{ij} \;\deq\; \E \abs{h_{ij}}^2
$$
of the matrix entries are only required to be of comparable size (but not necessarily equal):
\be\label{genwig}
\frac{c}{N} \;\leq\; s_{ij} \;\leq\; \frac{C}{N}
\ee
for some positive constants $c$ and $C$. 
(Such matrices were called \emph{generalized Wigner matrices} in \cite{EYYrigi}.)
This condition admits a  departure from spatial homogeneity, but still
 imposes a mean-field behaviour and hence excludes genuinely inhomogeneous models such as band matrices.

In the three-step approach to universality outlined above, the first step is to establish
the semicircle law on very short scales. In the scaling of $H$ where its spectrum is asymptotically given by the interval $[-2,2]$, the typical distance between neighbouring
eigenvalues is of order $1/N$. The number of eigenvalues
in  an interval of length $\eta$ is typically of order  $N\eta$.
Thus, the smallest possible scale on which the empirical
density may be close to a deterministic density (in our case the semicircle
law) is $\eta\gg 1/N$. If we characterize the
empirical spectral density around an energy $E$ on scale $\eta$
 by its Stieltjes transform, $m_N(z) = N^{-1}\tr (H-z)^{-1}$ for
$z=E+i\eta$, then the local semicircle law around the energy $E$ and in a spectral window of size $\eta$ is essentially equivalent to  
\begin{equation}\label{mNm}
|m_N(z)-m(z)| \;=\; o(1)
\end{equation}
as $N \to \infty$, where $m(z)$ is the Stieltjes transform of the semicircle law.
For any $\eta\gg 1/N$ (up to logarithmic corrections) the asymptotics \eqref{mNm} in the bulk spectrum was first proved
in \cite{ESY2} for Wigner matrices. The optimal
error bound of the form $O((N\eta)^{-1})$ (with an $N^\e$ correction)
was first proved in \cite{EYY2} in the bulk.
(Prior to this work, the best results 
were restricted to regime $\eta \ge N^{-1/2}$; see Bai et al.\ \cite{BMT}
as well as related concentration bounds in \cite{GZ}.)  
This result
was then extended to the spectral edges in \cite{EYYrigi}.  
 (Some  improvements over the estimates from \cite{ESY2}  
at the edges, for  a  special class of ensembles, were obtained  in  \cite{TV2}.)
In  \cite{EYYrigi},  the identical distribution of the entries of $H$ was not required,
 but the upper bound in  \eqref{genwig} on the variances was necessary. 
Band matrices in $d$ dimensions with band width $W$ satisfy the 
weaker bound $s_{ij}\leq C/W^d$. (Note that the band width $W$ is typically much smaller
than the linear size $L$ of the configuration space $\bb T$, i.e.\ the bound $W^{-d}$ is much larger than the inverse number of lattice sites, $L^{-d}= \abs{\bb T}^{-1}=N^{-1}$.) This motivates us to consider even more general matrices, with the sole condition
\be
s_{ij} \;\leq\; C/M
\label{sM}
\ee
on the variances (instead of \eqref{genwig}).
Here $M$ is a
new parameter that typically satisfies $M \ll N$.
(From now on, the relation $A\ll B$ for two $N$-dependent quantities $A$ and $B$ means that $A \leq N^{-\e}B$ 
for some positive $\e>0$.)
The question of the validity of the local semicircle law under the assumption \eqref{sM} was initiated in \cite{EYY}, where \eqref{mNm} was proved with an error term of order $(M\eta)^{-1/2}$
away from the spectral edges.

The purpose of this paper is twofold. First, we prove a local semicircle law \eqref{mNm}, under the variance condition \eqref{sM}, with a stronger error bound of order $(M\eta)^{-1}$, including energies $E$ near
the spectral edge. Away from the spectral edge (and from the origin $E=0$ if the matrix does not have a band structure), the result holds for any $\eta\gg 1/M$. 
Near the edge there is a restriction
on how small $\eta$ can be. This restriction depends explicitly
on a norm of the resolvent of the matrix of variances, $S=(s_{ij})$; we give
explicit bounds on this norm for various special cases of interest.

As a corollary, we derive bounds on the eigenvalue counting function
and rigidity estimates on the locations of the eigenvalues for a general class of matrices.
Combined with an analysis of Dyson Brownian motion and the Green function comparison method,
this yields bulk universality of the local eigenvalue statistics in a certain range of parameters, 
which depends on the matrix $S$.
In particular, we extend bulk universality, proved for generalized Wigner matrices
in \cite{EYY}, to a large class of matrix ensembles where the upper and lower bounds on 
the variances \eqref{genwig} are relaxed.

The main motivation for the generalizations in this paper is the Anderson
 transition for band matrices outlined above.
While not optimal, our results nevertheless imply that band matrices with a 
sufficiently broad band  plus a negligible mean-field component exhibit 
bulk universality: their local spectral statistics are governed by random matrix statistics.
For example, the local two-point correlation functions coincide if $W\gg N^{33/34}$. 
Although eigenvector delocalization and random matrix statistics are conjectured to occur
in tandem, delocalization was actually proved in \cite{EKYY3} under
more general conditions than those under which we establish random matrix statistics. 
In fact, the delocalization results of \cite{EKYY3} hold for a mean-field component as small as
$(N/W^2)^{2/3}$, and, provided that $W\gg N^{4/5}$, the mean-field component may even vanish 
(resulting in a genuine band matrix).

The second purpose of this paper is to 
provide a coherent, pedagogical, and
self-contained proof of the local semicircle law.
In recent years, a series of papers \cite{ESY1, ESY2,  EYY, EYY2, EYYrigi, EKYY1} 
with gradually weaker assumptions, was published on this topic. 
These papers often cited and relied on the previous ones.
This made it difficult for the interested reader to follow all the details of the argument.
The basic strategy of our proof (that is, using resolvents and large deviation bounds) was already used in \cite{ESY1, ESY2,  EYY, EYY2, EYYrigi, EKYY1}. In this paper we not only streamline the argument
for generalized Wigner matrices (satisfying \eqref{genwig}), but we also obtain sharper bounds for random matrices satisfying the much weaker condition \eqref{sM}.
This allows us to establish universality results for a class of ensembles beyond generalized Wigner matrices.

Our proof is self-contained and simpler than those of \cite{EYY, EYY2, EYYrigi, EKYY1}.
In particular, we give a proof of the \emph{Fluctuation Averaging Theorem}, Theorems \ref{thm: averaging} and \ref{thm: averaging with Lambdao} below, which is considerably simpler than that of its predecessors
in \cite{EYY2, EYYrigi, EKYY1}. In addition, we consistently use fluctuation averaging at several key steps of the main argument, which allows us to shorten the proof and relax previous assumptions on the variances $s_{ij}$.
The reader who is mainly interested in the pedagogical presentation should 
focus on the simplest choice of $S$, $s_{ij}=1/N$, 
which corresponds to the standard Wigner matrix (for which $M = N$), and focus on Sections \ref{sec: setup}, \ref{sec: tools}, \ref{sec: simple proof}, and \ref{sec:withgap}, as well as Appendix \ref{app: fluct averaging}.

We conclude this section with an outline of the paper. In Section \ref{sec: setup} we define the model, introduce basic definitions, and state the local semicircle law in full generality (Theorem \ref{thm: with gap}). Section \ref{sec: example} is devoted to some examples of random matrix models that satisfy our assumptions; for each example we give explicit bounds on the spectral domain on which the local semicircle law holds. Sections \ref{sec: tools}, \ref{sec: simple proof}, and \ref{sec:withgap} are devoted to the proof of the local semicircle law. Section \ref{sec: tools} collects the basic tools that will be used throughout the proof. The purpose of Section \ref{sec: simple proof} is mainly pedagogical; in it, we state and prove a weaker form of the local semicircle law, Theorem \ref{thm: no gap}. The error bounds in Theorem \ref{thm: no gap} are identical to those of Theorem \ref{thm: with gap}, but the spectral domain on which they hold is smaller. Provided one stays away from  the spectral edge, Theorems \ref{thm: no gap} and \ref{thm: with gap} are equivalent; near the edge, Theorem \ref{thm: with gap} is stronger. The proof of Theorem \ref{thm: no gap} is very short and contains several key ideas from the proof of Theorem \ref{thm: with gap}. The expert reader may therefore want to skip Section \ref{sec: simple proof}, but for the reader looking for a pedagogical presentation we recommend first focusing on Sections \ref{sec: tools} and \ref{sec: simple proof} (along with Appendix \ref{app: fluct averaging}). The full proof of our main result, Theorem \ref{thm: with gap}, is given in Section \ref{sec:withgap}. In Sections \ref{sec: dos} and \ref{sec:bulk} we draw consequences from Theorem \ref{thm: with gap}. In Section \ref{sec: dos} we derive estimates on the density of states and the rigidity of the eigenvalue locations. In Section \ref{sec:bulk} we state and prove the universality of the local spectral statistics in the bulk, and give applications to some concrete matrix models. In Appendix \ref{app: bounds on varrho} we derive explicit bounds on relevant norms of the resolvent of $S$ (denoted by
the abstract control parameters $\wt \Gamma$ and $\Gamma$), which are used to define the domains of applicability of Theorems \ref{thm: with gap} and \ref{thm: no gap}. Finally, Appendix \ref{app: fluct averaging} is devoted to the proof of the fluctuation averaging estimates, Theorems \ref{thm: averaging} and \ref{thm: averaging with Lambdao}.

We use $C$ to denote a generic large positive constant, which may depend on some fixed parameters and whose value may change from one expression to the next. Similarly, we use $c$ to denote a generic small positive constant.

\section{Definitions and the main result} \label{sec: setup}

Let $(h_{ij} \col i \leq j)$ be a family of independent, complex-valued 
random variables $h_{ij} \equiv h_{ij}^{(N)}$ satisfying
$\E h_{ij} = 0$  and $h_{ii} \in \R$ for all $i$.
For $i > j$ we define $h_{ij} \deq \bar h_{ji}$, and denote by $H \equiv H_N = (h_{ij})_{i,j = 1}^N$ the $N \times N$ matrix
 with entries $h_{ij}$. By definition, $H$ is Hermitian: $H = H^*$.
We stress that all our results hold not only for complex Hermitian matrices but also for real symmetric matrices. In fact, 
the symmetry class of $H$ plays no role, and our results apply for instance in the case where some off-diagonal entries 
of $H$ are real and some complex-valued.  (In contrast to some other papers in the literature, in our terminology the 
concept of Hermitian simply refers to the fact that $H = H^*$.)

We define
\begin{equation} \label{variance of h}
s_{ij} \;\deq\; \E \abs{h_{ij}}^2\,, \qquad M \;\equiv\; M_N \;\deq\; \frac{1}{\max_{i, j} s_{ij}}\,.
\end{equation}
In particular, we have the bound
\begin{equation} \label{s leq W}
s_{ij} \;\leq\; M^{-1}
\end{equation}
for all $i$ and $j$. We regard $N$ as the fundamental parameter of our model, and $M$ as a function of $N$. We introduce the $N \times N$ symmetric matrix $S \equiv S_N = (s_{ij})_{i,j = 1}^N$. We assume that $S$ is (doubly) stochastic:
\begin{equation} \label{S is stochastic}
\sum_j s_{ij} \;=\; 1
\end{equation}
for all $i$.  For simplicity, we assume that $S$ is irreducible, so that 1 is a simple eigenvalue. 
(The case of non-irreducible $S$ may be trivially dealt with by considering its irreducible components separately.) 
We shall always assume the bounds
\begin{equation} \label{lower bound on W}
N^\delta \;\leq\; M \;\leq\; N
\end{equation}
for some fixed $\delta > 0$.

It is sometimes convenient to use the normalized entries
\begin{equation}\label{def:zeta}
\zeta_{ij} \;\deq\; (s_{ij})^{-1/2} h_{ij}\,,
\end{equation}
which satisfy $\E \zeta_{ij} = 0$ and $\E \abs{\zeta_{ij}}^2 = 1$. (If $s_{ij} = 0$ we set for convenience $\zeta_{ij}$ to be a normalized Gaussian, so that these relations continue hold. Of course in this case the law of $\zeta_{ij}$ is immaterial.) We assume that the random variables $\zeta_{ij}$ have finite moments, uniformly in $N$, $i$, and $j$, in the sense that for all $p \in \N$ there is a constant $\mu_p$ such that
\begin{equation} \label{finite moments}
\E \abs{\zeta_{ij}}^p \;\leq\; \mu_p
\end{equation}
for all $N$, $i$, and $j$. We make this assumption to streamline notation in the statements of results such 
as Theorem \ref{thm: with gap} and the proofs. In fact, our results (and our proof) also cover the case where \eqref{finite moments}
 holds for some finite large $p$; see Remark \ref{rem: finite p}.

Throughout the following we use a spectral parameter $z \in \C$ satisfying $\im z > 0$. We use the notation
\begin{equation*}
z \;=\; E + \ii \eta\
\end{equation*}
without further comment, and always assume that $\eta>0$.
 Wigner semicircle law $\varrho$ and its Stieltjes transform $m$ are defined by
\begin{equation} \label{definition of msc}
  \varrho(x) \;\deq\; \frac{1}{2\pi}\sqrt{(4-x^2)_+}\,, \qquad m(z) \;\deq\; \frac{1}{2 \pi} \int_{-2}^2 \frac{\sqrt{4 - x^2}}{x - z} \, \dd x\,.
\end{equation}
To avoid confusion, we remark that $m$ was denoted by $m_{sc}$ in the papers
 \cite{ESY1, ESY2,  ESY4, ESYY, EYY, EYY2, EYYrigi, EKYY1, EKYY2}, in which $m$ had a different
 meaning from \eqref{definition of msc}. It is well known that the Stieltjes transform $m$ is the unique solution of
\begin{equation} \label{identity for msc}
m(z) + \frac{1}{m(z)} + z \;=\; 0
\end{equation}
satisfying $\im m(z) > 0$ for $\im z > 0$. Thus we have
\begin{equation} \label{explicit m}
m(z) \;=\; \frac{-z + \sqrt{z^2 - 4}}{2}.
\end{equation}
Some basic estimates on $m$ are collected in Lemma~\ref{lemma: msc} below.

An important parameter of the model is\footnote{Here we use the notation
$ \norm{A}_{\ell^\infty \to \ell^\infty} =\max_{i} \sum_j |A_{ij}|$ for the operator norm on $\ell^\infty(\C^N)$.}
\begin{equation}\label{def:rho}
\Gamma_N(z) \;\equiv\; \Gamma(z) \;\deq\; \normB{\pb{1-m(z)^2 S}^{-1} }_{\ell^\infty\to\ell^\infty}\,.
\end{equation}
A related quantity is obtained by restricting the operator $\pb{1-m(z)^2 S}^{-1}$ to the subspace $\f e^\perp$ orthogonal to the constant vector $\f e \deq N^{-1/2} (1,1, \dots, 1)^*$. Since $S$ is stochastic, we have the estimate $-1\leq S\leq 1$ and $1$ is a simple eigenvalue of $S$ with eigenvector $\f e$. Set
\begin{equation}\label{def:rhohat}
\wt \Gamma_N(z) \;\equiv\; \wt \Gamma(z) \;\deq\; \normB{\pb{1-m(z)^2 S}^{-1} \Big |_{\f e^\perp}}_{\ell^\infty\to\ell^\infty}\,,
\end{equation}
the norm of $(1-m(z) ^2S)^{-1}$ restricted to the subspace orthogonal to the constants.
Clearly, $\wt \Gamma(z) \leq \Gamma(z)$.
Basic estimates on $\Gamma$ and $\wt \Gamma$ are collected in Proposition \ref{prop: bounds on Gamma} below.
Many estimates in this paper depend critically on $\Gamma$ and $\wt \Gamma$. Indeed, these parameters quantify the stability of certain self-consistent equations that underlie our proof.
However, $\Gamma$ and $\wt \Gamma$ remain bounded (up to a factor $\log N$) provided $E=\re z$ is separated from the set $\{-2,0,2\}$; for band matrices (see Example \ref{def: band matrix}) it suffices that $E$ be separated from the spectral edges $\{-2,2\}$; see Appendix \ref{app: bounds on varrho}. At a first reading, we recommend that the reader neglect $\Gamma$ and $\wt \Gamma$ (i.e.\ replace them with a constant). For band matrices, this amounts to focusing on the local semicircle law in the bulk of the spectrum.

We define the \emph{resolvent} or \emph{Green function} of $H$ through
\begin{equation*}
G(z) \;\deq\; (H - z)^{-1}\,,
\end{equation*}
and denote its entries by $G_{ij}(z)$.
The Stieltjes transform of the empirical spectral measure of $H$ is
\begin{equation}\label{mNdef}
m_N(z) \;\deq\; \frac{1}{N} \tr G(z)\,.
\end{equation}

The following definition introduces a notion of a high-probability bound that is suited for our purposes. It was introduced (in a slightly different form) in \cite{EKYfluc}.

\begin{definition}[Stochastic domination]\label{def:stocdom}
Let
\begin{equation*}
X = \pb{X^{(N)}(u) \col N \in \N, u \in U^{(N)}} \,, \qquad
Y = \pb{Y^{(N)}(u) \col N \in \N, u \in U^{(N)}}
\end{equation*}
be two families of nonnegative random variables, where $U^{(N)}$ is a possibly $N$-dependent parameter set. 
We say that $X$ is \emph{stochastically dominated by $Y$, uniformly in $u$,} if for all (small) $\epsilon > 0$ and (large) $D > 0$ we have
\begin{equation*}
\sup_{u \in U^{(N)}} \P \qB{X^{(N)}(u) > N^\epsilon Y^{(N)}(u)} \;\leq\; N^{-D}
\end{equation*}
for large enough $N\geq N_0(\e, D)$. Unless stated otherwise, 
throughout this paper the stochastic 
domination will always be uniform in all parameters apart from the parameter $\delta$ in \eqref{lower bound on W} and the sequence of constants $\mu_p$ in \eqref{finite moments}; thus, $N_0(\e, D)$ also depends on $\delta$ and $\mu_p$.
If $X$ is stochastically dominated by $Y$, uniformly in $u$, we use the notation $X \prec Y$. Moreover, if for some complex family $X$ we have $\abs{X} \prec Y$ we also write $X = O_\prec(Y)$.
\end{definition}

For example, using Chebyshev's inequality and \eqref{finite moments} one easily finds that 
\begin{equation}\label{hsmallerW}
\abs{h_{ij}} \;\prec\; (s_{ij})^{1/2} \;\prec\; M^{-1/2}\,,
\end{equation}
so that we may also write $h_{ij} = O_\prec((s_{ij})^{1/2})$. Another simple, but useful, example is a family of events $\Xi \equiv \Xi^{(N)}$ with asymptotically very high probability:
If $\P (\Xi^c) \leq N^{-D}$ for any $D>0$ and $N \geq N_0(D)$, then the indicator function $\ind{\Xi}$ of $\Xi$ satisfies $1 - \ind{\Xi} \prec 0$.

The relation $\prec$ is a partial ordering, i.e.\ it is transitive and it satisfies the familiar arithmetic rules of order relations. For instance if $X_1 \prec Y_1$ and $X_2 \prec Y_2$ then $X_1 + X_2 \prec Y_1 + Y_2$ and $X_1 X_2 \prec Y_1 Y_2$. More general statements in this spirit are given in Lemma \ref{lemma: basic properties of prec} below.

\begin{definition}[Spectral domain] \label{def: domain}
We call an $N$-dependent family
\begin{equation*}
\f D \;\equiv\; \f D^{(N)} \;\subset\; \hb{z \col \abs{E} \leq 10 \,,\, M^{-1} \leq \eta \leq 10}
\end{equation*}
a \emph{spectral domain}. (Recall that $M\equiv M_N$ depends on $N$.)
\end{definition}
In this paper we always consider families $X^{(N)}(u) = X^{(N)}_i(z)$ indexed by $u = (z,i)$, where $z$ takes on values in some spectral domain $\f D$, and $i$ takes on values in some finite (possibly $N$-dependent or empty) index set. The stochastic domination $X \prec Y$ of such families will always be uniform in $z$ and $i$, and we usually do not state this explicitly. Usually, which spectral domain $\f D$ is meant will be clear from the context, in which case we shall not mention it explicitly.

In this paper we shall make use of two spectral domains, $\f S$ defined in \eqref{def S varrho} and $\wt {\f S}$ defined in \eqref{def S varrho 2}. Our main result is formulated on the larger of these domains, $\wt {\f S}$. In order to define it, we introduce an $E$-dependent lower boundary $\wt \eta_E$ on the spectral domain. We choose a (small) positive constant $\gamma$, and define for each $E \in [-10,10]$
\begin{equation} \label{def wt eta E}
\wt \eta_E \;\deq\; \min \hBB{\eta \;\col\; \frac{1}{M \eta} \leq \min \hbb{\frac{M^{-\gamma}}{\wt \Gamma(z)^3} \,,\, \frac{M^{-2 \gamma}}{\wt \Gamma(z)^4 \im m(z)}} \text{ for all }
 z \in [E + \ii \eta, E + 10 \ii] }\,.
\end{equation}
Note that $\wt \eta_E$ depends on $\gamma$, but we do not explicitly indicate this dependence since we regard $\gamma$ as fixed. At a first reading we advise the
reader to think of $\gamma$ as being zero. Note also that the lower bound in \eqref{bound for Gamma tilde} below implies that
$\wt \eta_E \geq M^{-1}$.
We also define  the distance to the spectral edge,
\begin{equation} \label{def: kappa}
\kappa \;\equiv\; \kappa_E \;\deq\; \absb{\abs{E} - 2}\,.
\end{equation}
Finally, we introduce the fundamental control parameter
\begin{equation}
\Pi(z) \;\deq\; \sqrt{\frac{\im m(z)}{M \eta}} + \frac{1}{M \eta}\,,
\end{equation}
which will be used throughout this paper as a sharp, deterministic upper bound on the entries of $G$. Note that the condition in the definition of $\wt \eta_E$ states that the first term of $\Pi$ is bounded by $M^{-\gamma} \wt \Gamma^{-2}$ and the second term by $M^{-\gamma} \wt \Gamma^{-3}$.
We may now state our main result.

\begin{theorem}[Local semicircle law] \label{thm: with gap}
Fix $\gamma \in (0,1/2)$ and define the spectral domain
\begin{equation} \label{def S varrho 2}
\wt{\f S} \;\equiv\; \wt{\f S} \!\,^{(N)}(\gamma) \;\deq\; \hb{E + \ii \eta \col \abs{E} \leq 10 \,,\, \wt \eta_E \leq \eta \leq 10}\,.
\end{equation}
We have the bounds
\begin{equation}\label{Gijest 2}
\max_{i,j} \absb{G_{ij}(z) - \delta_{ij} m(z)} \;\prec\; \Pi(z)
\end{equation}
uniformly in  $z \in \wt {\f S}$, as well as
\begin{equation}\label{m-mest 2}
\absb{m_N(z) - m(z)} \;\prec\; \frac{1}{M \eta}
\end{equation}
uniformly in $z \in \wt {\f S}$.
Moreover, outside of the spectrum we have the stronger estimate
\begin{equation}\label{m-mestout}
\absb{m_N(z) - m(z)} \;\prec\; \frac{1}{M (\kappa+\eta)}  + \frac{1}{(M\eta)^2\sqrt{\kappa+\eta}}
\end{equation}
uniformly in $z \in \wt {\f S} \cap \{z \col \abs{E} \geq 2 \,,\, M\eta\sqrt{\kappa+\eta}\geq M^\gamma\}$.
\end{theorem}

We remark that the main estimate for the Stieltjes transform $m_N$ is \eqref{m-mest 2}.
 The other estimate \eqref{m-mestout} is mainly useful for controlling the norm of $H$, 
which we do in Section \ref{sec: dos}. We also recall that  uniformity for
the spectral parameter $z$ means that the threshold $N_0(\e, D)$
in the definition of $\prec$ is independent of the choice of $z$ within the indicated spectral domain.  As stated
 in Definition \ref{def:stocdom}, this uniformity holds for all statements containing $\prec$, and 
is not explicitly mentioned in the following; all of our arguments are trivially uniform in $z$ 
and any matrix indices.

\begin{remark}\label{rem: finite p}
Theorem \ref{thm: with gap} has the following variant for matrix entries where the condition \eqref{finite moments} 
is only imposed for some large but  fixed $p$. More precisely,  for any $\epsilon > 0$ and $D > 0$ there exists a 
constant $p(\epsilon, D)$ such that if \eqref{finite moments} holds for $p = p(\epsilon, D)$ then
\begin{equation*}
\P \pb{\abs{m_N(z) - m(z)} > N^\epsilon (M \eta)^{-1}} \;\leq\; N^{-D}
\end{equation*}
for all $z \in \wt {\f S}$ and $N \geq N_0(\epsilon, D)$. An analogous estimate replaces
\eqref{Gijest 2} and \eqref{m-mestout}.
The proof of this variant is the same as that of Theorem \ref{thm: with gap}.
\end{remark}

\begin{remark}\label{rem:dec}
Most of the previous works \cite{ESY1, ESY2, EYY, EYY2, EYYrigi, EKYY1} 
assumed a stronger, 
subexponential decay condition on $\zeta_{ij}$ instead of \eqref{finite moments}.
Under the subexponential decay condition, certain probability estimates in the results
were somewhat stronger and precise tolerance thresholds were sharper. Roughly, 
this corresponds to operating with a modified definition of $\prec$, where
the factors $N^\e$ are replaced by high powers of $\log N$ and the polynomial probability
bound $N^{-D}$ is replaced with a subexponential one. The proofs of the current paper
can be easily adjusted to such a setup, but we shall not pursue this further.
\end{remark}

A local semicircle law for  Wigner matrices on the optimal scale $\eta \gtrsim 1/N$ was first obtained in \cite{ESY2}.  The optimal error estimates in the bulk were proved in \cite{EYY2}, and extended to the edges in \cite{EYYrigi}.
These estimates underlie the derivation of rigidity estimates for individual eigenvalues, which in turn were used in \cite{EYYrigi} to prove
Dyson's conjecture on the optimal local relaxation time for the Dyson Brownian motion. 

Apart from the somewhat different assumption on the tails of the entries of $H$ (see Remark~\ref{rem:dec}), Theorem~\ref{thm: with gap}, when restricted to generalized Wigner matrices, 
subsumes all previous local semicircle laws obtained in \cite{ESY1, ESY2,  EYY2, EYYrigi}. 
For band matrices, a local semicircle law was proved in \cite{EYY}.
(In fact, in \cite{EYY} the band structure was not required; only 
the conditions \eqref{s leq W}, \eqref{S is stochastic}, and the subexponential decay condition for the matrix entries (instead of \eqref{finite moments}) were used.)
Theorem~\ref{thm: with gap} improves this result in several ways.
First, the error bounds in \eqref{Gijest 2} and \eqref{m-mest 2}
are uniform in $E$, even for $E$ near the spectral edge; the corresponding bounds 
in Theorem 2.1 of \cite{EYY} diverged as $\kappa^{-1}$. 
Second, the bound
\eqref{m-mest 2} on the Stieltjes transform is better than (2.16) in \cite{EYY}
by a factor $(M\eta)^{-1/2}$. This improvement
is due to exploiting the fluctuation averaging mechanism of Theorem~\ref{thm: averaging}.
Third, the domain of $\eta$ for which  Theorem~\ref{thm: with gap} 
applies is essentially $\eta\gg \kappa^{-7/2}M^{-1}$,
which is somewhat larger than the domain $\eta\gg \kappa^{-4}M^{-1}$ of \cite{EYY}.

While Theorem~\ref{thm: with gap} subsumes several previous local semicircle laws, two previous results are not covered.
The local semicircle law for sparse matrices proved in \cite{EKYY1} does not follow from
Theorem~\ref{thm: with gap}. However, the argument of this paper may be modified so as to include sparse matrices as well; we do not pursue this issue further. 
The local semicircle law for one-dimensional band matrices given in Theorem 2.2 of \cite{EKYY3} is, however, 
of a very different nature, and may not be recovered using the methods of the current paper. Under the conditions $W\gg N^{4/5}$ and $\eta \gg N^2/W^3$,
Theorem 2.2 of \cite{EKYY3} shows that (focusing for simplicity on the one-dimensional case)
\be\label{lscbbb}
\absb{G_{ij}(z) - \delta_{ij} m(z)}  \;\prec\; \frac{1}{(N\eta)^{1/2}} + \frac{1}{ (W\sqrt{\eta})^{1/2}}
\ee
in the bulk spectrum, which is stronger than the bound  of order $(W\eta)^{-1/2}$
in \eqref{Gijest 2}.  The proof of \eqref{lscbbb} relies on a very general fluctuation averaging result from \cite{EKYfluc}, which is considerably stronger than Theorems \ref{thm: averaging} and \ref{thm: averaging with Lambdao}; see Remark~\ref{FAhist} below. 
The key open problem for band matrices is to establish a local semicircle law
on a scale $\eta$ below $W^{-1}$. The estimate \eqref{lscbbb} suggests that
the resolvent entries should remain bounded throughout the range $\eta\gtrsim \max \{ N^{-1}, W^{-2}\}$.

The local semicircle law, Theorem~\ref{thm: with gap}, has numerous consequences,
 several of which are formulated in Sections \ref{sec: dos} and
\ref{sec:bulk}. Here we only sketch them. Theorem~\ref{thm:count} states that the empirical
counting function converges to the counting function of the semicircle law. The precision is of order $M^{-1}$ provided that we have the lower bound $s_{ij}\ge c/N$ for some constant $c>0$. As a consequence, Theorem~\ref{thm:evl} states that
the bulk eigenvalues are rigid on scales of order $M^{-1}$. Under the same condition, 
in Theorem~\ref{thm: bulk univ} we prove the universality of the local two-point correlation
functions 
in the bulk provided that  $M\gg N^{33/34}$; 
 we obtain similar results for higher order
correlation functions, assuming a stronger restriction on $M$.
These results generalize the earlier theorems from \cite{EYYrigi, EKYY1, EKYY2}, which were
 valid for generalized Wigner matrices satisfying the condition \eqref{genwig}, under which $M$ is comparable to $N$.  
We obtain similar results  if the condition $s_{ij} \ge c/N$ in \eqref{genwig} is relaxed to  $s_{ij} \ge N^{-1-\xi}$ with some small $\xi$. The exponent $\xi$ can be chosen near 1
for band matrices with a broad band $W\asymp N$. In particular, we prove universality for such band matrices with a 
rapidly vanishing mean-field component.
These applications of the general Theorem~\ref{thm: bulk univ}  are listed
in Corollary~\ref{cor:bulk}.

\section{Examples} \label{sec: example}

In this section we give some important example of random matrix models $H$. In each of the examples, we give the deterministic matrix $S = (s_{ij})$ of the variances of the entries of $H$. The matrix $H$ is then obtained from $h_{ij} = s_{ij} \zeta_{ij}$. Here $(\zeta_{ij})$ is a Hermitian matrix whose upper-triangular entries are independent and whose diagonal entries are real; moreover, we have $\E \zeta_{ij} = 0$, $\E \abs{\zeta_{ij}}^2 = 1$, and the condition \eqref{finite moments} for all $p$, uniformly in $N$, $i$, and $j$.

\begin{definition}[Full and flat Wigner matrices] \label{def: Wigner}
Let $a \equiv a_N$ and $b \equiv b_N$ be possibly $N$-dependent positive quantities. We call $H$ an \emph{$a$-full Wigner matrix} if $S$ satisfies \eqref{S is stochastic} and
\begin{equation} \label{def full Wigner}
s_{ij} \;\geq\; \frac{a}{N}\,.
\end{equation}
Similarly, we call $H$ a \emph{$b$-flat Wigner matrix} if $S$ satisfies \eqref{S is stochastic} and
\begin{equation*}
s_{ij} \;\leq\; \frac{b}{N}\,.
\end{equation*}
(Note that in this case we have $M \geq N / b$.)

If $a$ and $b$ are independent of $N$ we call an $a$-full Wigner matrix simply \emph{full} and a $b$-flat Wigner 
matrix simply \emph{flat}. In particular, generalized Wigner matrices, satisfying \eqref{genwig}, are full and flat Wigner matrices.
\end{definition}

\begin{definition}[Band matrix] \label{def: band matrix}
Fix $d \in \N$. Let $f$ be a bounded and symmetric (i.e.\ $f(x) = f(-x)$) probability density on $\R^d$.
Let $L$ and $W$ be integers satisfying
\begin{equation*}
L^{\delta'} \;\leq\; W \;\leq\; L
\end{equation*}
for some fixed $\delta' > 0$. Define the $d$-dimensional discrete torus
\begin{equation*}
\bb T^d_L \;=\; [-L/2, L/2)^d \cap \Z^d\,.
\end{equation*}
Thus, $\bb T^d_L$ has $N = L^d$ lattice points; and we may identify $\bb T_L^d$ with $\{1, \dots, N\}$. We define the canonical representative of $i \in \Z^d$ through
\begin{equation*}
[i]_L \;\deq\; (i + L \Z^d) \cap \bb T^d_L\,.
\end{equation*}
Then $H$ is a \emph{$d$-dimensional band matrix} with band width $W$ and profile function $f$ if
\begin{equation*}
s_{ij} \;=\; \frac{1}{Z_{L}} \, f \pbb{\frac{[i - j]_L}{W}}\,,
\end{equation*}
where $Z_L$ is a normalization chosen so that \eqref{S is stochastic} holds.
\end{definition}

\begin{definition}[Band matrix with a mean-field component] \label{def: bandwigner}
Let $H_B$ a $d$-dimensional band matrix from Definition~\ref{def: band matrix}. Let $H_W$ be
an independent $a$-full Wigner matrix indexed by the set $\bb T^d_L$.
 The matrix $H \deq \sqrt{1-\nu} H_B + \sqrt{\nu}H_W$,
with some $\nu \in [0,1]$, is called a \emph{band matrix with a mean-field component}.
\end{definition}

The example of Definition \ref{def: bandwigner} is a mixture of the previous two. We are especially interested in the
case $\nu\ll1$, when most of the variance comes from the band matrix, i.e.\ the profile of $S$ is very close to a sharp band.

We conclude with some explicit bounds for these examples. 
The behaviour of $\Gamma$ and $\wt \Gamma$ near the spectral edge is governed by the parameter
\begin{equation} \label{def theta}
\theta \;\equiv\; \theta(z) \;\deq\;
\begin{cases}
\kappa + \frac{\eta}{\sqrt{\kappa + \eta}} & \text{if } \abs{E} \leq 2
\\
\sqrt{\kappa + \eta} & \text{if }\abs{E} > 2\,,
\end{cases}
\end{equation}
where we set, as usual, $\kappa \equiv \kappa_E$ and $z=E+i\eta$. Note that the parameter $\theta$ may be bounded from below by $(\im m)^2$.
The following results follow immediately from Propositions \ref{prop: bounds on Gamma} and \ref{prop: spectrum of S} in Appendix \ref{app: bounds on varrho}. They hold for an arbitrary spectral domain $\f D$.
\begin{enumerate}
\item
For general $H$ and any constant $c > 0$, there is a constant $C > 0$ such that
\begin{equation*}
C^{-1} \;\leq\; \wt \Gamma \;\leq\; \Gamma \;\leq\; C \log N
\end{equation*}
provided $\dist (E, \{-2,0,2\}) \geq c$.
\item
For a full Wigner matrix we have
\begin{equation*}
c \;\leq\; \wt \Gamma \;\leq\; C \log N \,, \qquad \frac{c}{\sqrt{\kappa + \eta}} \;\leq\; \Gamma \;\leq\; \frac{C \log N}{\theta}\,,
\end{equation*}
where $C$ depends on the constant $a$ in Definition \ref{def: Wigner} but $c$ does not.
\item
For a band matrix with a mean-field component, as in Definition \ref{def: bandwigner}, we have
\begin{equation*}
c \;\leq\; \wt \Gamma \;\leq\; \frac{C \log N}{(W/L)^2 +\nu a + \theta}\,.
\end{equation*}
The case $\nu=0$ corresponds to a band matrix from Definition \ref{def: band matrix}.
\end{enumerate}

\section{Tools} \label{sec: tools}
In this subsection we collect some basic facts that will be used throughout the paper.
% We use $C$ to denote a generic, positive, large constant, and $c$ to denote a generic, positive, small constant; they may depend on some fixed parameters and their value may change from one expression to the next.
For two positive quantities $A_N$ and $B_N$ we use the notation $A_N \asymp B_N$ to mean $c A_N \leq B_N \leq C A_N$.
Throughout the following we shall frequently drop the arguments $z$ and $N$, bearing in mind that we are dealing with a function on some spectral domain $\f D$.

\begin{definition}[Minors] \label{def: minors}
For $ {\bb T} \subset \{1, \dots, N\}$ we define $H^{( {\bb T})}$ by
\begin{equation*}
(H^{( {\bb T})})_{ij} \;\deq\; \ind{i \notin  {\bb T}} \ind{j \notin  {\bb T}} h_{ij}\,.
\end{equation*}
Moreover, we define the resolvent of $H^{( {\bb T})}$ through
\begin{equation*}
G^{( {\bb T})}_{ij}(z) \;\deq\;  \pb{H^{( {\bb T})} - z}^{-1}_{ij}\,.
\end{equation*}
We also set
\begin{equation*}
\sum_i^{( {\bb T})} \;\deq\; \sum_{i \col i \notin  {\bb T}}\,.
\end{equation*}
When $ {\bb T} = \{a\}$, we abbreviate $(\{a\})$ by $(a)$ in the above definitions; similarly, we write $(ab)$ instead of $(\{a,b\})$.
\end{definition}

\begin{definition}[Partial expectation and independence] \label{definition: P Q}
Let $X \equiv X(H)$ be a random variable. For $i \in \{1, \dots, N\}$ define the operations $P_i$ and $Q_i$ through
\begin{equation*}
P_i X \;\deq\; \E(X | H^{(i)}) \,, \qquad Q_i X \;\deq\; X - P_i X\,.
\end{equation*}
We call $P_i$ \emph{partial expectation} in the index $i$. Moreover, we say that $X$ is \emph{independent of ${\bb T} \subset \{1, \dots, N\}$} if $X = P_i X$ for all $i \in {\bb T}$.
\end{definition}

We introduce the random $z$-dependent control parameters
\begin{equation}\label{def:Lambda}
\Lambda_o \;\deq\; \max_{i \neq j} \abs{G_{ij}}\,, \qquad
\Lambda_d \;\deq\; \max_i \abs{G_{ii} - m}\,,
\qquad \Lambda \;\deq\; \max \{\Lambda_o, \Lambda_d\}\,,
\qquad \Theta \;\deq\; \abs{m_N - m}\,.
\end{equation}
We remark that the letter $\Lambda$ had a different meaning in several earlier papers, such as \cite{EYYrigi}.
The following lemma collects basic bounds on $m$.

\begin{lemma} \label{lemma: msc}
There is a constant $c > 0$ such that for $E \in [-10, 10]$ and $\eta \in (0, 10]$ we have
\begin{equation} \label{m is bounded}
c \;\leq\; \abs{m(z)} \;\leq\; 1 - c\eta\,,
\end{equation}
\begin{equation} \label{1-msquare}
\abs{1-m^2(z)} \;\asymp\; \sqrt{\kappa + \eta} \,,
\end{equation}
as well as
\begin{equation} \label{lower bound on im msc}
\im m(z) \;\asymp\;
\begin{cases}
\sqrt{\kappa + \eta} & \text{if $\abs{E} \leq 2$}
\\
\frac{\eta}{\sqrt{\kappa + \eta}} & \text{if $\abs{E} \geq 2$}\,.
\end{cases}
\end{equation}
\end{lemma}
\begin{proof}
The proof is an elementary exercise using \eqref{explicit m}.
\end{proof}

In particular, recalling that $-1 \leq S \leq 1$ and using the upper bound $\abs{m} \leq C$ from \eqref{m is bounded}, we find that there is a constant $c > 0$ such that
\begin{equation} \label{varrho geq c}
c \;\leq\; \wt \Gamma \;\leq\; \Gamma\,.
\end{equation}

The following lemma collects basic algebraic properties of stochastic domination $\prec$. Roughly, it states that $\prec$ satisfies the usual arithmetic
 properties of order relations. We shall use it tacitly throughout the following.

\begin{lemma} \label{lemma: basic properties of prec}
\begin{enumerate}
\item
Suppose that $X(u,v) \prec Y(u,v)$ uniformly in $u \in U$ and $v \in V$. If $\abs{V} \leq N^C$ for some constant $C$ then
\begin{equation*}
\sum_{v \in V} X(u,v) \;\prec\; \sum_{v \in V} Y(u,v)
\end{equation*}
uniformly in $u$.
\item
Suppose that $X_{1}(u) \prec Y_{1}(u)$ uniformly in $u$ and $X_{2}(u) \prec Y_{2}(u)$ uniformly in $u$. Then
$X_{1}(u) X_{2}(u) \prec Y_{1}(u) Y_{2}(u)$
uniformly in $u$.
\item
If $X \prec Y + N^{-\epsilon} X $ for some $\epsilon > 0$ then $X \prec Y$.
\end{enumerate}
\end{lemma}
\begin{proof}
The claims (i) and (ii) follow from a simple union bound. The claim (iii) is an immediate consequence of the definition of $\prec$. 
\end{proof}

The following resolvent identities form the backbone of all of our calculations. The idea behind them is that a resolvent  matrix element $G_{ij}$
depends strongly on the $i$-th and $j$-th columns of $H$, but weakly on all other columns. The first identity determines how to make
a resolvent matrix element $G_{ij}$ independent of an additional index $k \neq i,j$.
The second identity expresses the dependence of a resolvent matrix element $G_{ij}$ on the matrix elements in the $i$-th or in the $j$-th column of $H$.

\begin{lemma}[Resolvent identities] \label{lemma: res id}
For any Hermitian matrix $H$ and ${\bb T} \subset \{1, \dots, N\}$ the following identities hold.
If $i,j,k \notin {\bb T}$ and $i,j \neq k$ then
\begin{equation} \label{resolvent expansion type 1}
G_{ij}^{({\bb T})} \;=\; G_{ij}^{({\bb T}k)} + \frac{G_{ik}^{({\bb T})} G_{kj}^{({\bb T})}}{G_{kk}^{({\bb T})}}\,, \qquad
\frac{1}{G_{ii}^{(\bb T)}} \;=\; \frac{1}{G_{ii}^{(\bb Tk)}} - \frac{G_{ik}^{(\bb T)} G_{ki}^{(\bb T)}}{G_{ii}^{(\bb T)} G_{ii}^{(\bb Tk)} G_{kk}^{(\bb T)}}\,.
\end{equation}
If $i,j \notin {\bb T}$ satisfy $i \neq j$ then
\begin{equation} \label{resolvent expansion type 2}
G_{ij}^{({\bb T})} \;=\; - G_{ii}^{({\bb T})} \sum_{k}^{({\bb T}i)} h_{ik} G_{kj}^{({\bb T}i)} \;=\; - G_{jj}^{({\bb T})} \sum_k^{({\bb T}j)} G_{ik}^{({\bb T}j)} h_{k j}\,.
\end{equation}
\end{lemma}
\begin{proof}
This is an exercise in linear algebra. The first identity \eqref{resolvent expansion type 1} was proved in
 Lemma 4.2 of \cite{EYY} and the second is an immediate consequence of the first. 
The identity \eqref{resolvent expansion type 2} is proved in Lemma 6.10 of \cite{EKYY2}.
\end{proof}

Our final tool consists of the following results on \emph{fluctuation averaging}.
 They exploit cancellations in sums of fluctuating quantities involving resolvent
matrix entries. A very general result was obtained in \cite{EKYfluc}; in this paper
we state a special case sufficient for our purposes here, and give a relatively simple proof in
 Appendix \ref{app: fluct averaging}.
We consider weighted averages of diagonal resolvent matrix entries $G_{kk}$. They
are weakly dependent, but the correlation between $G_{kk}$ and
$G_{mm}$ for $m\ne k$ is not sufficiently small to apply the general 
theory of sums of weakly dependent random variables; instead, we need to exploit
the precise form of the dependence using the resolvent structure.

It turns out that the key quantity that controls the magnitude of the fluctuations is $\Lambda$.
However, being a random variable, $\Lambda$ itself is unsuitable as an upper bound. For technical reasons (our proof relies on a high-moment estimate combined with Chebyshev's inequality), it is essential that $\Lambda$ be estimated by
 a \emph{deterministic} control parameter, which we call $\Psi$. The error terms
are then estimated in terms of powers of $\Psi$.
We shall always assume that $\Psi$ satisfies
\begin{equation} \label{admissible Psi}
M^{-1/2} \;\leq\; \Psi \;\leq\; M^{-c}
\end{equation}
in the spectral domain $\f D$, where $c > 0$ is some constant. 
 We shall perform the averaging with respect to a family of complex 
 weights $T = (t_{ik})$ satisfying
\begin{equation} \label{condition on weights}
0 \;\leq\;  \abs{t_{ik}}  \;\leq\; M^{-1} \,, \qquad \sum_{k}  \abs{t_{ik}}
  \;\leq\; 1\,.
\end{equation}
Typical example  weights are $t_{ik} = s_{ik}$ and $t_{ik} = N^{-1}$. Note that in both of these cases $T$ commutes with $S$.
We introduce the \emph{average} of a vector $(a_i)_{i = 1}^N$ through
\begin{equation} \label{def average}
[a] \;\deq\; \frac{1}{N} \sum_i a_i\,.
\end{equation}

\begin{theorem}[Fluctuation averaging] \label{thm: averaging}
Fix a spectral domain $\f D$ and a deterministic control parameter $\Psi$ satisfying \eqref{admissible Psi}.
Suppose that $\Lambda \prec \Psi$ and the weight $T = (t_{ik})$ satisfies \eqref{condition on weights}.
Then we have
\begin{equation} \label{averaging with Q}
\sum_{k} t_{ik} Q_k \frac{1}{G_{kk}} \;=\; O_\prec(\Psi^2)\,, \qquad \sum_{k} t_{ik} Q_k G_{kk} \;=\; O_\prec(\Psi^2)\,.
\end{equation}
 If $T$ commutes with $S$ then
\begin{equation} \label{averaging without Q}
\sum_{k} t_{ik} v_k \;=\;  O_\prec(\Gamma \Psi^2)\,.
\end{equation}
Finally, if $T$ commutes with $S$ and
\begin{equation} \label{sum_t_1}
\sum_{k} t_{ik} \;=\; 1
\end{equation}
for all $i$ then
\begin{equation} \label{averaging without Q avg}
\sum_{k} t_{ik} (v_k - [v]) \;=\;  O_\prec(\wt \Gamma \Psi^2)\,,
\end{equation}
where we defined $v_i \deq G_{ii} - m$.
The estimates \eqref{averaging with Q}, \eqref{averaging without Q}, and \eqref{averaging without Q avg} are uniform in the index $i$. 
\end{theorem}

In fact, the first 	bound of \eqref{averaging with Q} can be improved as follows.

\begin{theorem} \label{thm: averaging with Lambdao}
Fix a spectral domain $\f D$ deterministic control parameters $\Psi$ and $\Psi_o$, both satisfying \eqref{admissible Psi}.
Suppose that $\Lambda \prec \Psi$, $\Lambda_o \prec \Psi_o$, and that the weight $T = (t_{ik})$ satisfies \eqref{condition on weights}.
Then
\begin{equation} \label{averaging with Lambdao}
\sum_{k} t_{ik} Q_k \frac{1}{G_{kk}} \;=\;  O_\prec(\Psi_o^2)\,.
\end{equation}
\end{theorem}

\begin{remark}\label{FAhist}
The first instance of the fluctuation averaging mechanism appeared in \cite{EYY2} for the Wigner case,
where $[Z]= N^{-1}\sum_k Z_k$ was proved to be bounded by $\Lambda_o^2$. Since $Q_k[G_{kk}]^{-1}$ is essentially
$Z_k$ (see \eqref{schur} below), this corresponds to the first bound in \eqref{averaging with Q}. 
A different proof (with a better bound on the constants) was given in \cite{EYYrigi}.
A conceptually streamlined version of the original proof was extended to
sparse matrices \cite{EKYY1} and to sample covariance matrices \cite{PY1}. 
Finally, an extensive analysis in \cite{EKYfluc} treated the fluctuation averaging of general polynomials
of resolvent entries and identified the order of cancellations depending on the
algebraic structure of the polynomial.
 Moreover, in \cite{EKYfluc} an  additional cancellation effect was found 
for the quantity $Q_i|G_{ij}|^2$. These improvements played a key role in obtaining 
the diffusion profile for the resolvent of band matrices and the estimate~\eqref{lscbbb} in \cite{EKYY3}.

All proofs of the fluctuation averaging  theorems rely on computing expectations of high moments
of the averages, and carefully estimating the resulting terms. 
In \cite{EKYfluc}, a diagrammatic representation was developed for bookkeeping such terms,
but this is necessary only for the case of general polynomials. For the special cases given in Theorem
\ref{thm: averaging}, the proof is relatively simple and it is presented in Appendix~\ref{app: fluct averaging}.
Compared with \cite{EYY2,EYYrigi,EKYY1}, the algebra of the decoupling of the randomness is greatly 
simplified in the current paper. Moreover, unlike their counterparts from \cite{EYY2,EYYrigi,EKYY1}, 
the fluctuation averaging results of Theorems \ref{thm: averaging} and \ref{thm: averaging with Lambdao} do not require conditioning on the complement of some ``bad'' low-probability event, because such events are automatically accounted for by the definition of $\prec$\,; this leads to further simplifications in the proofs of Theorems \ref{thm: averaging} and \ref{thm: averaging with Lambdao}.
\end{remark}

\section{A simpler proof using $\Gamma$ instead of $\wt \Gamma$} \label{sec: simple proof}

In this section we prove the following weaker version of Theorem \ref{thm: with gap}. 
In analogy to \eqref{def wt eta E}, we introduce the lower boundary
\begin{equation} \label{def eta E}
\eta_E \;\deq\; \min \hBB{\eta \;\col\; \frac{1}{M \eta} \leq
 \min \hbb{\frac{M^{-\gamma}}{\Gamma(z)^3} \,,\, \frac{M^{-2 \gamma}}{\Gamma(z)^4 \im m(z)}} \text{ for all }
 z \in [E + \ii \eta, E + 10 \ii] }\,.
\end{equation}

\begin{theorem} \label{thm: no gap}
Fix $\gamma \in (0,1/2)$ and define the spectral domain
\begin{equation} \label{def S varrho}
\f S \;\equiv\; \f S^{(N)}(\gamma) \;\deq\; \hb{E + \ii \eta \col \abs{E} \leq 10 \,,\, \eta_E \leq \eta \leq 10}\,.
\end{equation}
We have the bounds
\begin{equation}\label{Gijest}
\absb{G_{ij}(z) - \delta_{ij} m(z)} \;\prec\; \Pi(z)
\end{equation}
uniformly in $i,j$ and $z \in \f S$, as well as
\begin{equation}\label{m-mest}
\absb{m_N(z) - m(z)} \;\prec\; \frac{1}{M \eta}
\end{equation}
uniformly in $z \in \f S$.
\end{theorem}

Note that the only difference between Theorems \ref{thm: with gap} and \ref{thm: no gap}
 is that $\wt \Gamma$ was replaced with the larger quantity $\Gamma$ in the definition of
the threshold $\eta_E$ and the spectral domain, so that
\begin{equation} \label{orders}
    \frac{1}{M} \;\le\; \wt \eta_E \;\le\; \eta_E\,, \qquad \f S \;\subset\; \wt{\f S}\,.
\end{equation}
Hence Theorem \ref{thm: no gap} is indeed weaker than Theorem \ref{thm: with gap}, since it holds  on a smaller spectral domain.
As outlined after \eqref{def:rhohat} and discussed in detail in Appendix \ref{app: bounds on varrho}, Theorems \ref{thm: no gap} and \ref{thm: with gap} are equivalent provided $E$ is separated from  the set $\{-2, 0, 2\}$ (for band matrices they are equivalent provided $E$ is separated from the spectral edges $\pm 2$).

The rest of this section is devoted to the proof of Theorem \ref{thm: no gap}. We give the full proof of Theorem \ref{thm: no gap} for pedagogical reasons, since it is simpler than that of Theorem \ref{thm: with gap} but already contains several of its key ideas. Theorem \ref{thm: with gap} will be proved in Section~\ref{sec:withgap}. 
One big difference between the two proofs is that in Theorem \ref{thm: no gap} the main control parameter is $\Lambda$, while
in  Theorem \ref{thm: with gap} we have to keep track of two control parameters, $\Lambda$ and the smaller $\Theta$.

\subsection{The self-consistent equation}

The key tool behind the proof is a self-consistent equation for the diagonal entries of $G$. The starting point is Schur's complement formula, which we write as
\begin{equation} \label{schur}
\frac{1}{G_{ii}} \;=\; h_{ii} - z - \sum_{k,l}^{(i)} h_{ik} G_{kl}^{(i)} h_{li}\,.
\end{equation}
The partial expectation with respect to the index $i$ (see Definition \ref{definition: P Q}) of the last term on the right-hand side reads
\begin{equation*}
   P_i  \sum_{k,l}^{(i)} h_{ik} G_{kl}^{(i)} h_{li} \;=\; \sum_{k}^{(i)} s_{ik}   G_{kk}^{(i)} 
  \;=\; \sum_{k}^{(i)} s_{ik} G_{kk} -  \sum_{k}^{(i)} s_{ik}  \frac{G_{ik} G_{ki}}{G_{ii}}
\;=\; \sum_{k} s_{ik} G_{kk} -  \sum_{k} s_{ik}  \frac{G_{ik} G_{ki}}{G_{ii}}\,,
\end{equation*}
where in the first step we used \eqref{variance of h} and in the second \eqref{resolvent expansion type 1}.
Introducing  the notation
\begin{equation*}
v_i \;\deq\; G_{ii} - m
\end{equation*}
and recalling \eqref{S is stochastic}, we therefore get from \eqref{schur} that
\begin{equation}\label{1g}
\frac{1}{G_{ii}} \;=\; -z - m + \Upsilon_i - \sum_k s_{ik} v_k\,,
\end{equation}
where we introduced the fluctuating error term
\begin{align}\label{def:Ups}
\Upsilon_i \;\deq\; A_i + h_{ii} - Z_i\,, \qquad
A_i \;\deq\; \sum_{k}s_{ik} \frac{G_{ik} G_{ki}}{G_{ii}}\,, \qquad 
Z_i \;\deq\; Q_i \sum_{k,l}^{(i)} h_{ik} G_{kl}^{(i)} h_{li}\,.
\end{align}
Using \eqref{identity for msc}, we therefore get the \emph{self-consistent equation}
\begin{equation} \label{vself for exp}
-\sum_k s_{ik} v_k + \Upsilon_i \;=\; \frac{1}{m + v_i} - \frac{1}{m}\,.
\end{equation}
Notice that this is an equation for the family $(v_i)_{i = 1}^N$, with random error terms $\Upsilon_i$.

 Self-consistent equations play a crucial role in analysing resolvents of
random matrices. The simplest one is the \emph{scalar (or first level) self-consistent equation}
for $m_N(z)$, the Stieltjes transform of the empirical density \eqref{mNdef}. 
By averaging the inverse of \eqref{1g} and neglecting the error terms, one
obtains that $m_N$ approximately satisfies the equation $m = -(m+z)^{-1}$, which
is the defining relation for the Stieltjes transform of the semicircle law \eqref{identity for msc}.

The \emph{vector (or second level) self-consistent equation}, as given in \eqref{vself for exp}, allows one to control not only fluctuations of $m_N-m$ but also those of $G_{ii}-m$. 
The equation \eqref{vself for exp} first appeared in \cite{EYY}, where
a systematic study of resolvent entries of random matrices was initiated.

For completeness, we mention that a \emph{matrix (or third level) self-consistent equation} 
for local averages of $|G_{ij}|^2$, was introduced in \cite{EKYY3}. This equation constitutes the backbone of the study of the diffusion profile of the resolvent entries of random band matrices.

\subsection{Estimate of the error $\Upsilon_i$ in terms of $\Lambda$}

\begin{lemma} \label{lemma: Lambdao}
The following statements hold for any spectral domain $\f D$. Let $\phi$ be the indicator function of some (possibly $z$-dependent) event. If $\phi \Lambda \prec M^{-c}$ for some $c > 0$ then
\begin{equation} \label{Lambdao 1}
\phi \pb{\Lambda_o + \abs{Z_i} + \abs{\Upsilon_i}} \;\prec\; \sqrt{\frac{\im m + \Lambda}{M \eta}}
\end{equation}
uniformly in $z \in \f D$.
Moreover, for any fixed ($N$-independent) $\eta > 0$ 
 we have
\begin{equation} \label{Lambdao 2}
\Lambda_o + \abs{Z_i} + \abs{\Upsilon_i} \;\prec\; M^{-1/2}
\end{equation}
uniformly in $z \in \h{w \in \f D \col \im w = \eta}$.
\end{lemma}

\begin{proof}
We begin with the first statement. We shall often use the fact that, by the lower bound of \eqref{m is bounded} and the assumption $\phi \Lambda \prec M^{-c}$, we have
\begin{equation} \label{1/G is bounded}
\phi / \abs{G_{ii}} \;\prec\; 1\,.
\end{equation}
First we estimate $Z_i$, which we split as
\begin{equation} \label{Zi split}
\phi \abs{Z_i} \;\leq\; \phi \absBB{\sum_{k}^{(i)} \pb{\abs{h_{ik}}^2 - s_{ik}} G_{kk}^{(i)}} + \phi \absBB{\sum_{k \neq l}^{(i)} h_{ik} G_{kl}^{(i)} h_{li}}\,.
\end{equation}
We estimate each term using the large deviation estimates from Theorem \ref{thm: LDE}, by conditioning on $G^{(i)}$ and using the fact that the family $(h_{ik})_{k = 1}^N$ is independent of $G^{(i)}$. By \eqref{lde 1}, the first term of \eqref{Zi split} is stochastically dominated by $\phi \pb{\sum_k^{(i)} s_{ik}^2 \absb{G_{kk}^{(i)}}^2}^{1/2} \prec M^{-1/2}$, where we used the estimate \eqref{s leq W} and $\phi \absb{G_{kk}^{(i)}} \prec 1$, as follows from \eqref{resolvent expansion type 1}, \eqref{1/G is bounded}, and the assumption $\phi \Lambda \prec M^{-c}$. For the second term of \eqref{Zi split} we apply \eqref{lde 3} with $a_{kl} = s_{ik}^{1/2} G_{kl}^{(i)} s_{li}^{1/2}$ and $X_k = \zeta_{ik}$ (see \eqref{def:zeta}). We find
\begin{equation} \label{estimate for sum Gkl}
\phi \sum_{k,l}^{(i)} s_{ik} \absb{G_{kl}^{(i)}}^2 s_{li} \;\leq\; \phi \frac{1}{M} \sum_{k,l}^{(i) } s_{ik} \absb{G_{kl}^{(i)}}^2 \;=\; \phi \frac{1}{M \eta} \sum_k^{(i)} s_{ik} \im G_{kk}^{(i)} \;\prec\; \frac{\im m + \Lambda}{M \eta}\,,
\end{equation}
where  the second step follows by spectral decomposition of $G^{(i)}$, and in the last step we used \eqref{resolvent expansion type 1} and \eqref{1/G is bounded}.
Thus we get
\begin{equation} \label{Zi estimate 1}
\phi \abs{Z_i} \;\prec\; \sqrt{\frac{\im m + \Lambda}{M \eta}}\,,
\end{equation}
where we absorbed the bound $M^{-1/2}$ on the first term of \eqref{Zi split} into the right-hand side of \eqref{Zi estimate 1}, using $\im m \geq \eta$ as follows from \eqref{lower bound on im msc}.

Next, we estimate $\Lambda_o$. We can iterate \eqref{resolvent expansion type 2} once to get, for $i \neq j$,
\begin{equation} \label{iterated identity}
G_{ij} \;=\; - G_{ii}\sum_k^{(i)} h_{ik} G_{kj}^{(i)} \;=\; - G_{ii} G_{jj}^{(i)} \pBB{h_{ij} - \sum_{k,l}^{(ij)} h_{ik} G_{kl}^{(ij)} h_{lj}}\,.
\end{equation}
The term $h_{ij}$ is trivially $O_\prec(M^{-1/2})$. In order to estimate the other term, we invoke \eqref{lde 2} with $a_{kl} = s_{ik}^{1/2} G_{kl}^{(ij)} s_{lj}^{1/2}$, $X_k = \zeta_{ik}$, and $Y_l = \zeta_{lj}$. As in \eqref{estimate for sum Gkl}, we find
\begin{equation*}
\phi \sum_{k,l}^{ (ij) } s_{ik} \absb{G_{kl}^{(ij)}}^2 s_{lj} \;\prec\; \frac{\im m + \Lambda}{M \eta}\,.
\end{equation*}
Thus we find
\begin{equation} \label{Lambdao estimate 1}
\phi \Lambda_o \;\prec\; \sqrt{\frac{\im m + \Lambda}{M \eta}}\,,
\end{equation}
where we again absorbed the term $h_{ij} \prec M^{-1/2}$ into the right-hand side.

In order to estimate $A_i$ and $h_{ii}$ in the definition of $\Upsilon_i$, we use \eqref{1/G is bounded} to estimate 
\begin{equation*}
\phi \pb{\abs{A_i} + \abs{h_{ii}}} \;\prec\; \phi \Lambda_o^2 + M^{-1/2} \;\leq\; \phi \Lambda_o + C \sqrt{\frac{\im m}{M \eta}} \;\prec\; \sqrt{\frac{\im m + \Lambda}{M \eta}}\,,
\end{equation*}
where the second step follows from $\im m \geq \eta$ (recall \eqref{lower bound on im msc}). This completes the proof of \eqref{Lambdao 1}.

The proof of \eqref{Lambdao 2} is almost identical to that of \eqref{Lambdao 1}. The quantities $\absb{G^{(i)}_{kk}}$ and $\absb{G^{(ij)}_{kk}}$ are estimated by the trivial deterministic bound $\eta^{-1}$. We omit the details.
\end{proof}

\subsection{A rough bound on $\Lambda$}
The next step in the proof of Theorem \ref{thm: no gap} is to establish the following rough bound on $\Lambda$.

\begin{proposition} \label{prop: rough bound 1}
We have $\Lambda \prec M^{-\gamma / 3} \Gamma^{-1}$ uniformly in $\f S$.
\end{proposition}

The rest of this subsection is devoted to the proof of Proposition \ref{prop: rough bound 1}.
The core of the proof is a {\it continuity argument.}
 Its basic idea is to establish a \emph{gap} in the range of 
$\Lambda$ of the form $\ind{\Lambda \leq M^{-\gamma/4} \Gamma^{-1}} \Lambda  \prec M^{-\gamma/2} \Gamma^{-1}$ (Lemma \ref{lem: bootstrap 1} below). In other words, for all $z \in \f S$, with high probability either $\Lambda \leq M^{-\gamma/2} \Gamma^{-1}$ or $\Lambda \geq M^{-\gamma/4} \Gamma^{-1}$. For $z$ with a large imaginary part $\eta$, the estimate $\Lambda \leq M^{-\gamma / 2} \Gamma^{-1}$ is  easy to prove using a simple expansion (Lemma \ref{lem: initial estimate} below). Thus, for large $\eta$ the parameter $\Lambda$ is below the gap. Using the fact that $\Lambda$ is continuous in $z$ and hence cannot jump from one side of the gap to the other, we then conclude that with high probability $\Lambda$ is below the gap for all $z \in \f S$. See Figure \ref{fig: gap lambda} for an illustration of this argument.

\begin{lemma} \label{lem: bootstrap 1}
We have the bound
\begin{equation*}
\indb{\Lambda \leq M^{-\gamma/4} \Gamma^{-1}} \Lambda \;\prec\; M^{-\gamma /2} \Gamma^{-1}
\end{equation*}
uniformly in $\f S$.
\end{lemma}

\begin{proof}
Set 
$$
\phi \;\deq\; \indb{\Lambda \leq M^{-\gamma/4} \Gamma^{-1}}.
$$
 Then by definition we have $\phi \Lambda \leq M^{-\gamma / 4} \Gamma^{-1} \leq C M^{-\gamma/4}$, where in the last step we used \eqref{varrho geq c}. Hence we may invoke \eqref{Lambdao 1} to estimate $\Lambda_o$ and $\Upsilon_i$. In order to estimate $\Lambda_d$, we expand the right-hand side of \eqref{vself for exp} in $v_i$ to get
\begin{equation*}
\phi \pbb{-\sum_k s_{ik} v_k + \Upsilon_i} \;=\; \phi \pb{-m^{-2} v_i + O(\Lambda^2)}\,,
\end{equation*}
where we used \eqref{m is bounded} and that $\abs{v_i} \leq C M^{-\gamma/4}$ on the event $\{\phi = 1\}$.
Using \eqref{Lambdao 1} we therefore have
\begin{equation*}
\phi \pbb{v_i - m^2 \sum_{k} s_{ik} v_k} \;=\; O_\prec \pbb{\Lambda^2 + \sqrt{\frac{\im m + \Lambda}{M \eta}}}\,.
\end{equation*}
We write the left-hand side as $\phi [(1-m^2S)\f v]_i$ with the vector $\f v = (v_i)_{i = 1}^N$. Inverting
the operator $1-m^2 S$, we
therefore conclude that
\begin{equation*}
\phi \Lambda_d \;=\; \phi \max_i \abs{v_i} \;\prec\; \Gamma \pbb{\Lambda^2 + \sqrt{\frac{\im m + \Lambda}{M \eta}}}\,.
\end{equation*}
Recalling \eqref{varrho geq c} and \eqref{Lambdao 1}, we therefore get
\begin{equation} \label{phi Lambda}
\phi \Lambda \;\prec\; \phi \Gamma \pbb{\Lambda^2 + \sqrt{\frac{\im m + \Lambda}{M \eta}}}\,.
\end{equation}

Next, by definition of $\phi$ we may estimate
\begin{equation*}
\phi \Gamma \Lambda^2 \;\leq\; M^{-\gamma/2} \Gamma^{-1}\,.
\end{equation*}
Moreover, by definitions of $\f S$ and $\phi$ we have
\begin{equation*}
\phi \Gamma \sqrt{\frac{\im m + \Lambda}{M \eta}} \;\leq\; \Gamma \sqrt{\frac{\im m}{M \eta}} + \Gamma \sqrt{\frac{\Gamma^{-1}}{M \eta}} \;\leq\; M^{-\gamma} \Gamma^{-1} + M^{-\gamma/2} \Gamma^{-1} \;\leq\; 2 M^{-\gamma /2} \Gamma^{-1}\,.
\end{equation*}
Plugging this into \eqref{phi Lambda} yields $\phi \Lambda \prec M^{-\gamma/2} \Gamma^{-1}$, which is the claim.
\end{proof}

In order to start the continuity argument underlying the proof of Proposition \ref{prop: rough bound 1}, we need the following bound on $\Lambda$ for large $\eta$.

\begin{lemma} \label{lem: initial estimate}
We have $\Lambda \prec M^{-1/2}$ uniformly in $z \in [-10,10] + 2 \ii$.
\end{lemma}
\begin{proof}
We shall make use of the trivial bounds
\begin{equation} \label{trivial bounds for large eta}
\absb{G_{ij}^{(\bb T)}} \;\leq\; \frac{1}{\eta} \;=\; \frac{1}{2} \,, \qquad \abs{m} \;\leq\; \frac{1}{\eta}\;=\; \frac{1}{2}\,.
\end{equation}
From \eqref{Lambdao 2} we get
\begin{equation} \label{trivial bound for Lambdao}
\Lambda_o + \abs{Z_i} \;\prec\; M^{-1/2}\,.
\end{equation}
Moreover, we use \eqref{resolvent expansion type 1} and \eqref{iterated identity} to estimate
\begin{equation*}
\abs{A_i} \;\leq\; \sum_j s_{ij} \absbb{\frac{G_{ij} G_{ji}}{G_{ii}}} \;\leq\; M^{-1} + \sum_j^{(i)} s_{ij} \absb{G_{ji} G_{jj}^{(i)}} \, \absBB{h_{ij} - \sum_{k,l}^{(ij)} h_{ik} G_{kl}^{(ij)} h_{lj}} \;\prec\; M^{-1/2}\,,
\end{equation*}
where the last step follows using \eqref{lde 2}, exactly as the estimate of the right-hand side of \eqref{iterated identity} in the proof of Lemma \ref{lemma: Lambdao}. We conclude that $\abs{\Upsilon_i} \prec M^{-1/2}$.

Next, we write \eqref{vself for exp} as
\begin{equation*}
v_i \;=\; \frac{m \pb{\sum_k s_{ik} v_k - \Upsilon_i}}{\pb{m^{-1} - \sum_k s_{ik} v_k + \Upsilon_i}}\,.
\end{equation*}
Using $\abs{m^{-1}} \geq 2$ and $\abs{v_k} \leq 1$ as follows from \eqref{trivial bounds for large eta}, we find 
\begin{equation*}
\absbb{m^{-1} + \sum_k s_{ik} v_k - \Upsilon_i} \;\geq\; 1 + O_\prec(M^{-1/2})\,.
\end{equation*}
Using $\abs{m} \leq 1/2$ we therefore conclude that
\begin{equation*}
\Lambda_d \;\leq\; \frac{\Lambda_d + O_\prec(M^{-1/2})}{2 + O_\prec(M^{-1/2})} 
 \;=\;  \frac{\Lambda_d}{2} + O_\prec(M^{-1/2})\,,
\end{equation*}
from which the claim follows together with the estimate
on $\Lambda_o$ from \eqref{trivial bound for Lambdao}.
\end{proof}

We may now conclude the proof of Proposition \ref{prop: rough bound 1} by a continuity argument in $\eta=\im z$. The gist of the continuity argument is depicted in Figure \ref{fig: gap lambda}.

\begin{figure}[ht!]
\begin{center}
\includegraphics{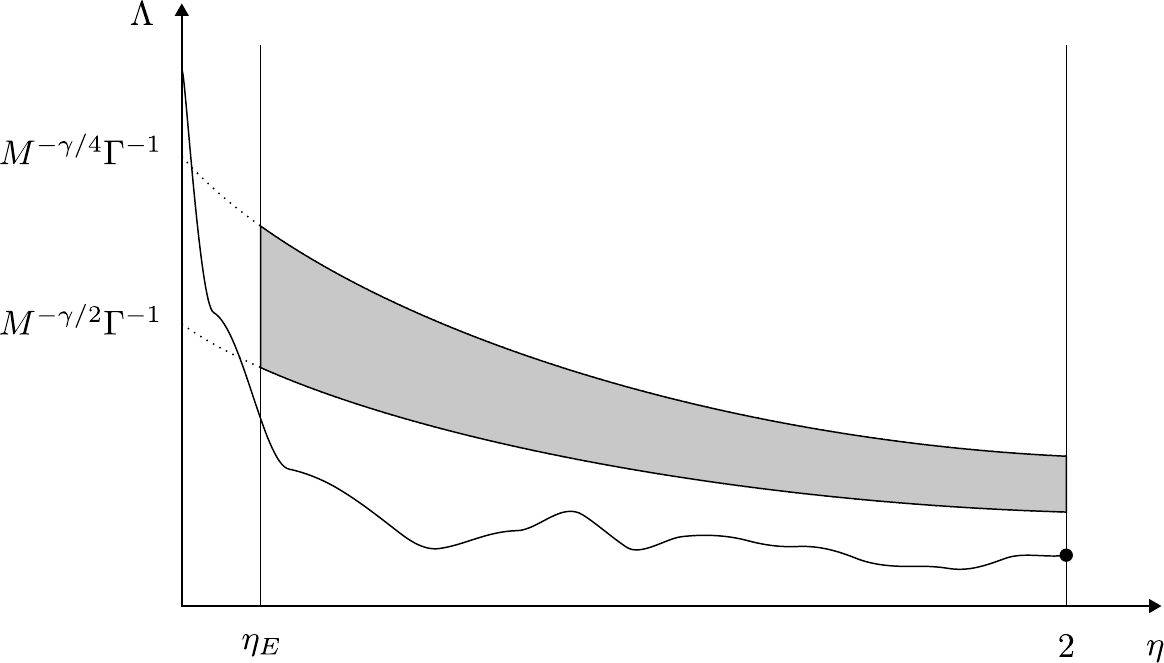}
\end{center}
\caption{The $(\eta, \Lambda)$-plane for a fixed $E$. The shaded region is forbidden with high probability by Lemma \ref{lem: bootstrap 1}. The initial estimate, given by Lemma \ref{lem: initial estimate}, is marked with a black dot.
The graph of $\Lambda = \Lambda(E + \ii \eta)$ is continuous and lies beneath the shaded region.
Note that this method does not control $\Lambda(E+i\eta)$
in the regime  $\eta\leq \eta_E$.  \label{fig: gap lambda}
}
\end{figure}

\begin{proof}[Proof of Proposition \ref{prop: rough bound 1}]
Fix $D > 10$. Lemma \ref{lem: bootstrap 1} implies that for each $z \in \f S$ we have
\begin{equation} \label{gap for one z}
\P \pB{M^{-\gamma/3} \Gamma(z)^{-1} \leq \Lambda(z) \leq M^{-\gamma/4} \Gamma(z)^{-1}} \;\leq\; N^{-D}
\end{equation}
for $N \geq N_0$, where $N_0 \equiv N_0(\gamma, D)$ does not depend on $z$.

Next, take a lattice $\Delta \subset \f S$ such that $\abs{\Delta} \leq N^{10}$ and for each $z \in \f S$ there exists a $w \in \Delta$ such that $\abs{z - w} \leq N^{-4}$. Then \eqref{gap for one z} combined with a union bounds gives
\begin{equation} \label{gap for lattice}
\P \pB{\exists w \in \Delta \,\col\, M^{-\gamma/3} \Gamma(w)^{-1} \leq \Lambda(w) \leq M^{-\gamma/4} \Gamma(w)^{-1}} \;\leq\; N^{-D + 10}
\end{equation}
for $N \geq N_0$.
From the definitions of $\Lambda(z)$, $\Gamma(z)$, and $\f S$ (recall \eqref{varrho geq c}),
 we immediately find that $\Lambda$ and $\Gamma$ are Lipschitz continuous on $\f S$, with Lipschitz constant at most $M^2$. Hence \eqref{gap for lattice} implies
\begin{equation*}
\P \pB{\exists z \in \f S \,\col\, 2 M^{-\gamma/3} \Gamma(z)^{-1} \leq \Lambda(z) \leq 2^{-1} M^{-\gamma/4} \Gamma(z)^{-1}} \;\leq\; N^{-D + 10}
\end{equation*}
for $N \geq N_0$.
We conclude that there is an event $\Xi$ satisfying $\P(\Xi) \geq 1 - N^{-D + 10}$ such
that, for each $z \in \f S$, either $\ind{\Xi} \Lambda(z) \leq 2 M^{-\gamma/3} \Gamma(z)^{-1}$ or $\ind{\Xi} \Lambda(z) \geq 2^{-1} M^{-\gamma/4} \Gamma(z)^{-1}$. Since $\Lambda$ is  continuous 
and $\f S$ is by definition connected, we conclude that either
\begin{equation} \label{good case}
\forall z \in \f S \,\col\, \ind{\Xi} \Lambda(z) \;\leq\; 2 M^{-\gamma/3} \Gamma(z)^{-1}
\end{equation}
or
\begin{equation} \label{bad case}
\forall z \in \f S \,\col\, \ind{\Xi} \Lambda(z) \;\geq\; 2^{-1} M^{-\gamma/4} \Gamma(z)^{-1}\,.
\end{equation}
(Here the bounds \eqref{good case} and \eqref{bad case} each hold surely, i.e.\ for every realization of $\Lambda(z)$.)

It remains to show that \eqref{bad case} is impossible. In order to do so, it suffices to show that there exists a $z \in \f S$ such that $\Lambda(z) < 2^{-1} M^{-\gamma/4} \Gamma(z)^{-1}$ with probability greater than $1/2$. But this holds for any $z$ with $\im z = 2$, as follows from Lemma \ref{lem: initial estimate} and the bound $\Gamma \leq C \eta^{-1}$, which itself follows easily by a simple expansion of $(1 - m^2 S)^{-1}$ combined with the bounds $\norm{S}_{\ell^\infty \to \ell^\infty} \leq 1$ and  \eqref{m is bounded}. This concludes the proof.
\end{proof}

\subsection{Iteration step and conclusion of the proof of Theorem \ref{thm: no gap}}

In the following a key role will be played by \emph{deterministic} control parameters $\Psi$ satisfying
\begin{equation} \label{condition on Psi 1}
c M^{-1/2} \;\leq\; \Psi \;\leq\; M^{-\gamma/3} \Gamma^{-1}\,.
\end{equation}
(Using the definition of $\f S$ and \eqref{lower bound on im msc} it is not hard to check that the upper bound in \eqref{condition on Psi 1} is always larger than the lower bound.)
Suppose that $\Lambda \prec \Psi$ in $\f S$ for some deterministic parameter $\Psi$ satisfying \eqref{condition on Psi 1}. For example, by Proposition \ref{prop: rough bound 1} we may choose $\Psi = M^{-\gamma/3} \Gamma^{-1}$.

We now improve the estimate $\Lambda \prec \Psi$ iteratively. The iteration step is the content of the following proposition.

\begin{proposition} \label{prop: optimal simple}
Let $\Psi$ be a control parameter
satisfying \eqref{condition on Psi 1} and
 fix $\epsilon \in (0,\gamma/3)$. Then 
\begin{equation} \label{iteration step}
\Lambda \;\prec\; \Psi \qquad \Longrightarrow \qquad \Lambda \;\prec\; F(\Psi)\,,
\end{equation}
where we defined
\begin{equation*}
F(\Psi) \;\deq\; M^{-\epsilon} \Psi + \sqrt{\frac{\im m}{M \eta}} + \frac{M^{\epsilon}}{M \eta}\,.
\end{equation*}
\end{proposition}

For the proof of Proposition \ref{prop: optimal simple} we need the following averaging result, which is a simple corollary of Theorem \ref{thm: averaging}. 

\begin{lemma} \label{lem: avg Upsilon}
Suppose that $\Lambda \prec \Psi$ for some deterministic control parameter $\Psi$ satisfying \eqref{admissible Psi}. Then $[\Upsilon] = O_\prec(\Psi^2)$ (recall the definition of the average $[\cdot]$ from \eqref{def average}).
\end{lemma}
\begin{proof}
The claim easily follows from Schur's complement formula \eqref{schur} written in the form
\begin{equation*}
\Upsilon_i \;=\; A_i + Q_i \frac{1}{G_{ii}}\,.
\end{equation*}
We may therefore estimate  $[\Upsilon]$ using the trivial bound $\abs{A_i} \prec \Psi^2$ as well as the fluctuation averaging bound from the first estimate of \eqref{averaging with Q} with $t_{ik} = 1/N$.
\end{proof}

\begin{proof}[Proof of Proposition \ref{prop: optimal simple}]
Suppose that $\Lambda \prec \Psi$ for some deterministic control parameter $\Psi$ satisfying \eqref{condition on Psi 1}.
We invoke Lemma \ref{lemma: Lambdao} with $\phi = 1$ (recall the bound \eqref{varrho geq c}) to get
\begin{equation} \label{Lambdao Z estimate in varrho proof}
\Lambda_o + \abs{Z_i} + \abs{\Upsilon_i} \;\prec\; \sqrt{\frac{\im m + \Lambda}{M \eta}} \;\prec\; \sqrt{\frac{\im m + \Psi}{M \eta}}\,.
\end{equation}
Next, we estimate $\Lambda_d$. Define the $z$-dependent indicator function 
\begin{equation*}
\psi \;\deq\; \ind{\Lambda \leq M^{-\gamma/4}}\,.
\end{equation*}
By \eqref{condition on Psi 1}, \eqref{varrho geq c}, and the assumption $\Lambda \prec \Psi$, we have $1 - \psi \prec 0$. On the event $\{\psi = 1\}$, we expand the right-hand side of \eqref{vself for exp} to get the bound
\begin{equation*}
\psi \abs{v_i} \;\leq\; C \psi \absbb{\sum_k s_{ik} v_k - \Upsilon_i} + C \psi \Lambda^2\,.
\end{equation*}
Using the fluctuation averaging estimate \eqref{averaging without Q} as well as \eqref{Lambdao Z estimate in varrho proof}, we find
\begin{equation} \label{use of FA in short proof}
\psi \abs{v_i} \;\prec\; \Gamma \Psi^2 + \sqrt{\frac{\im m + \Psi}{M \eta}}\,,
\end{equation}
where we again used the lower bound from \eqref{varrho geq c}. 
 Using $1 - \psi \prec 0$ we conclude
\begin{equation} \label{bound for Lambdad 1}
\Lambda_d \;\prec\; \Gamma \Psi^2 + \sqrt{\frac{\im m + \Psi}{M \eta}}\,,
\end{equation}
which, combined with \eqref{Lambdao Z estimate in varrho proof}, yields
\begin{equation} \label{bound Lambda 1}
\Lambda \;\prec\; \Gamma \Psi^2 + \sqrt{\frac{\im m + \Psi}{M \eta}}\,.
\end{equation}
Using Young's inequality and the assumption $\Psi \leq M^{-\gamma/3} \Gamma^{-1}$ we conclude the proof.
\end{proof}

For the remainder of the proof of Theorem \ref{thm: no gap} we work on the spectral domain $\f S$. We claim that if $\Psi$ satisfies \eqref{condition on Psi 1} then so does $F(\Psi)$. The lower bound $F(\Psi) \geq c M^{-1/2}$ is a consequence of the estimate $\im m / \eta \geq c$, which follows from \eqref{lower bound on im msc}. The upper bound $M^{-\gamma/3 - \epsilon} \Gamma^{-1}$ on the first term of $F(\Psi)$ is trivial by assumption on $\Psi$. Moreover, the second term of $F(\Psi)$ satisfies $\sqrt{\im m / (M \eta)} \leq M^{-\gamma} \Gamma^{-2} \leq C M^{-\gamma} \Gamma^{-1} \leq M^{-\gamma/3 - \epsilon} \Gamma^{-1}$ by definition of $\f S$ and the lower bound \eqref{varrho geq c}. Similarly, the last term of $F(\Psi)$ satisfies $M^\epsilon / (M \eta) \leq C M^{\epsilon -\gamma} \Gamma^{-1} \leq M^{-\gamma/3 - \epsilon} \Gamma^{-1}$ by definition of $\f S$.

We may therefore iterate \eqref{iteration step}. This yields a bound on $\Lambda$ that is essentially the fixed point of the map $\Psi \mapsto F(\Psi)$, which is $\Pi$ (up to the factor $M^\epsilon$). More precisely,
the iteration is started with $\Psi_0 \deq M^{-\gamma/3} \Gamma^{-1}$; the initial hypothesis $\Lambda \prec \Psi_0$ is provided by the rough bound from Proposition \ref{prop: rough bound 1}. For $k \geq 1$ we set $\Psi_{k+1} \deq F(\Psi_k)$. Hence from \eqref{iteration step} we conclude that $\Lambda \prec \Psi_k$ for all $k$. Choosing $k \deq \ceil{\epsilon^{-1}}$ yields
\begin{equation*}
\Lambda \;\prec\; \sqrt{\frac{\im m}{M \eta}} + \frac{M^{\epsilon}}{M \eta}\,.
\end{equation*}
Since $\epsilon$ was arbitrary, we have proved that
\begin{equation} \label{Gijest in proof}
\Lambda \;\prec\; \Pi\,,
\end{equation}
which is \eqref{Gijest}.

What remains is to prove \eqref{m-mest}, i.e.\ to estimate $\Theta$. We expand \eqref{vself for exp} on $\{\psi = 1\}$ to get
\begin{equation} \label{expanded vself 2}
\psi m^2 \pbb{-\sum_k s_{ik} v_k + \Upsilon_i} \;=\; - \psi v_i + O(\psi \Lambda^2)\,.
\end{equation}
Averaging in \eqref{expanded vself 2} yields
\begin{equation*}
\psi m^2 \pb{-[v] + [\Upsilon]} \;=\; - \psi [v] + O(\psi \Lambda^2)\,.
\end{equation*}
By \eqref{Gijest in proof} and \eqref{Lambdao Z estimate in varrho proof} with $\Psi=\Pi$, we have $\Lambda + \abs{\Upsilon_i} \prec \Pi$. Moreover, by Lemma \ref{lem: avg Upsilon} we have $\abs{[\Upsilon]} \prec \Pi^2$. Thus we get
\begin{equation*}
\psi [v] \;=\; m^2 \psi [v] + O_\prec(\Pi^2)\,.
\end{equation*}
Since $1 - \psi \prec 0$, we conclude that $[v] = m^2 [v] + O_\prec(\Pi^2)$. Therefore
\begin{equation*}
\abs{[v]} \;\prec\; \frac{\Pi^2}{\abs{1 - m^2}} \;\leq\; \pbb{\frac{\im m}{\abs{1 - m^2}} + \frac{1}{\abs{1 - m^2} M \eta}} \frac{2}{M \eta} \;\leq\; \pbb{C + \frac{\Gamma}{M \eta}} \frac{2}{M \eta} \;\leq\; \frac{C}{M \eta}\,.
\end{equation*}
Here in the third step we used \eqref{1-msquare}, \eqref{lower bound on im msc}, and the bound $\Gamma \geq \abs{1 - m^2}^{-1}$ which follows from the definition of $\Gamma$ by applying the matrix $(1 - m^2 S)^{-1}$ to the vector $\f e = N^{-1/2} (1, 1, \dots, 1)^*$. The last step follows from the definition of $\f S$. Since $\Theta = \abs{[v]}$, this concludes the proof of \eqref{m-mest}, and hence of Theorem \ref{thm: no gap}.

\section{Proof of Theorem \ref{thm: with gap}}\label{sec:withgap}

The key novelty in this proof is that we solve
the self-consistent equation \eqref{vself for exp} separately on the subspace 
of constants (the span of the vector $\f e$) and on its orthogonal complement $\f e^\perp$.
On the space of constant vectors, it becomes a scalar equation for
the average $[v]$, which can be expanded up to second order. 
Near the spectral edges $\pm 2$, the resulting quadratic self-consistent scalar equation
(given in \eqref{vself for avg v} below) is
more effective than its linearized version.
On the space orthogonal to the constants, we still solve
a self-consistent vector equation, but the stability will
now be quantified using $\wt\Gamma$ instead of the larger quantity $\Gamma$.

Accordingly, the main control parameter in this proof is $\Theta = \abs{[v]}$, and the key iterative
scheme (Lemma \ref{lem: iteration for optimal bound} below) is formulated in terms of $\Theta$. However, many intermediate estimates still involve $\Lambda$. In particular, 
the self-consistent equation \eqref{vself for exp} is effective only 
in the regime where $v_i$ is already small. Hence we need two preparatory
steps. In Section \ref{sec: rough bound with gap} we will prove an apriori
bound on $\Lambda$, essentially showing that $\Lambda \ll 1$. This
proof itself is a continuity argument (see Figure \ref{fig: gap} for a graphical illustration) similar to the proof of 
Proposition \ref{prop: rough bound 1}; now, however, we have to follow $\Lambda$ and 
$\Theta$ in tandem. The main reason why $\Theta$ is already involved in
this part is that we work in larger spectral domain $\wt {\f S}$ defined using
$\wt\Gamma$. Thus, already in this preparatory step,
 the self-consistent equation has to be solved separately 
on the subspace of constants and its orthogonal complement.

In Section \ref{sec:lambdatheta}, we control $\Lambda$
in terms of $\Theta$, which allows us to obtain a self-consistent equation involving only $\Theta$.
In this step we use the Fluctuation Averaging Theorem to obtain a quadratic estimate which, very roughly, states that $\Lambda \lesssim \Theta + \Lambda^2$ (see \eqref{quadr} below for the
precise statement). This implies $\Lambda \lesssim \Theta$ in the regime $\Lambda \ll 1$.

Finally, in Section \ref{sec:itertheta}, we solve the quadratic iteration
for $\Theta$. Since the corresponding quadratic equation has a dichotomy
and for large $\eta=\im z$ we know that $\Theta$ is small by direct expansion,
a continuity argument similar to the proof of  Proposition \ref{prop: rough bound 1} will complete the
proof.

\subsection{A rough bound on $\Lambda$} \label{sec: rough bound with gap}

In this section
we prove the following apriori bounds
 on both control parameters, $\Lambda$ and $\Theta$.

\begin{proposition} \label{prop: rough bound with gap}
In $\wt {\f S}$ we have the bounds
\begin{equation*}
\Lambda \;\prec\; M^{-\gamma/4} \wt \Gamma^{-1}\,, \qquad \Theta \;\prec\; (M \eta)^{-1/3}\,.
\end{equation*}
\end{proposition}

Before embarking on the proof of Proposition \ref{prop: rough bound with gap}, we
state some preparatory lemmas. First, we
 derive the key equation for $[v] = N^{-1} \sum_i v_i$, the average of $v_i$.

\begin{lemma} \label{lem: vself with gap}
Define the $z$-dependent indicator function
\begin{equation} \label{def phi with gap}
\phi \;\deq\; \ind{\Lambda \leq M^{-\gamma/4} \wt \Gamma^{-1}}
\end{equation}
and the random control parameter
\begin{equation*}
q(\Theta) \;\deq\; \sqrt{\frac{\im m + \Theta}{M \eta}} + \frac{\wt \Gamma}{M \eta}\,.
\end{equation*}
Then we have
\begin{equation} \label{vself for avg v}
\phi \pB{(1 - m^2) [v] - m^{-1} [v]^2} \;=\; \phi \, O_\prec \pb{q(\Theta) + M^{-\gamma/4} \Theta^2}
\end{equation}
and
\begin{equation} \label{Lambda estimated in terms of Theta}
\phi \Lambda \;\prec\; \Theta + \wt \Gamma \, q(\Theta)\,.
\end{equation}
\end{lemma}

\begin{proof}
For the whole proof we work on the event $\{\phi = 1\}$, i.e.\ every quantity is multiplied by $\phi$. We consistently drop these factors $\phi$ from our notation in order to avoid cluttered expressions. In particular, $\Lambda\leq CM^{-\gamma/4}$ throughout the proof.

We begin by estimating $\Lambda_o$ and $\Lambda_d$ in terms of $\Theta$. 
Recalling \eqref{varrho geq c}, we find that $\phi$ satisfies the hypotheses of Lemma \ref{lemma: Lambdao}, from which we get
\begin{equation}\label{Lambdao}
\Lambda_o + \abs{\Upsilon_i} \;\prec\; r(\Lambda)\,, \qquad r(\Lambda) \;\deq\; \sqrt{\frac{\im m + \Lambda}{M \eta}}\,.
\end{equation}
In order to estimate $\Lambda_d$, we expand the self-consistent equation \eqref{vself for exp} (on the event $\{\phi = 1\}$) to get
\begin{equation}\label{exp1}
v_i- m^2 \sum_k s_{ik} v_k \;=\; O_\prec\pb{\Lambda^2 + r(\Lambda)}\,;
\end{equation}
here we used the bound \eqref{Lambdao} on $|\Upsilon_i|$.
Next, we subtract the average $N^{-1} \sum_i$ from each side to get
\begin{equation*}
(v_i - [v])- m^2 \sum_k s_{ik} (v_k - [v]) \;=\; O_\prec\pb{\Lambda^2 + r(\Lambda)}\,.
\end{equation*}
Note that the average of the left-hand side vanishes, so that the average of the right-hand side also vanishes. Hence the right-hand side is perpendicular to  $\f e$. Inverting the operator $1 - m^2 S$ on the subspace $\f e^\perp$ therefore yields
\begin{equation} \label{vi - avg v bound}
\absb{v_i - [v] } \;\prec\; \wt \Gamma \pb{\Lambda^2 + r(\Lambda)}\,.
\end{equation}
Combining with the bound $\Lambda_o\prec r(\Lambda)$ 
from \eqref{Lambdao}, we therefore get 
\begin{equation} \label{Lambda step 1}
\Lambda \;\prec\; \Theta + \wt \Gamma \Lambda^2 + \wt \Gamma r(\Lambda)\,.
\end{equation}
By definition of $\phi$ we have $\wt \Gamma \Lambda^2 \leq M^{-\gamma/4} \Lambda$, so that  by Lemma \ref{lemma: basic properties of prec} (iii) 
 the second term on the right-hand side of \eqref{Lambda step 1} may be absorbed into the left-hand side:
\begin{equation} \label{Lambda step 2}
\Lambda \;\prec\; \Theta + \wt \Gamma r(\Lambda)\,.
\end{equation}
Now we claim that
\begin{equation} \label{p Lambda p Theta}
r(\Lambda) \;\prec\; q(\Theta)\,.
\end{equation}
If \eqref{p Lambda p Theta} is proved, clearly \eqref{Lambda estimated in terms of Theta} follows from \eqref{Lambda step 2}. In order to prove \eqref{p Lambda p Theta}, we use
 \eqref{Lambda step 2} and the Cauchy-Schwarz inequality to get
\begin{equation*}
r(\Lambda) \;\leq\; \sqrt{\frac{\im m}{M \eta}} + \sqrt{\frac{\Lambda}{M \eta}} \;\prec\; \sqrt{\frac{\im m}{M \eta}} + \sqrt{\frac{\Theta}{M \eta}} + \sqrt{\frac{\wt \Gamma \, r(\Lambda)}{M \eta}} \;\leq\; \sqrt{\frac{\im m}{M \eta}} + \sqrt{\frac{\Theta}{M \eta}} + M^{-\epsilon} r(\Lambda) + M^\epsilon \frac{\wt \Gamma}{M \eta}
\end{equation*}
for any $\epsilon > 0$. We conclude that
\begin{equation*}
r(\Lambda)  \;\prec\; \sqrt{\frac{\im m}{M \eta}} + \sqrt{\frac{\Theta}{M \eta}} +  M^\epsilon \frac{\wt \Gamma}{M \eta}\,.
\end{equation*}
Since $\epsilon > 0$ was arbitrary, \eqref{p Lambda p Theta} follows.

Next, we estimate $\Theta$. We expand \eqref{vself for exp} to second order:
\begin{equation} \label{selfc expanded deg 2}
-\sum_k s_{ik} v_k + \Upsilon_i \;=\; - \frac{1}{m^2} v_i + \frac{1}{m^3} v_i^2 + O(\Lambda^3)\,.
\end{equation}
In order to take the average and get a closed equation for $[v]$, we write, using \eqref{vi - avg v bound},
\begin{equation*}
v_i^2 \;=\; \pb{[v] + v_i - [v]}^2 \;=\; [v]^2 + 2 [v] (v_i - [v]) + O_\prec \pB{\wt \Gamma^2 \pb{\Lambda^2 + r(\Lambda)}^2}\,.
\end{equation*}
Plugging this back into \eqref{selfc expanded deg 2} and taking the average over $i$ gives
\begin{equation*}
-m^2 [v] + m^2 [\Upsilon] \;=\; -[v] + m^{-1} [v]^2  + O_\prec \pB{\Lambda^3 + \wt \Gamma^2 \Lambda^4 + \wt \Gamma^2 r(\Lambda)^2}\,.
\end{equation*}
Estimating  $[\Upsilon]$ by $\max \abs{\Upsilon_i} \prec r(\Lambda)$ (recall \eqref{Lambdao}) yields
\begin{equation*}
(1 - m^2) [v] - m^{-1} [v]^2 \;=\; O_\prec \pB{r(\Lambda) + \Lambda^3 + \wt \Gamma^2 \Lambda^4 + \wt \Gamma^2 r(\Lambda)^2}\,.
\end{equation*}
By definitions of $\wt {\f S}$ and $\phi$, we have $\wt \Gamma^2 r(\Lambda) \leq 1$. Therefore we may absorb the last error term into the first.
For the second and third error terms we use \eqref{Lambda step 2} to get
\begin{equation*}
(1 - m^2) [v] - m^{-1} [v]^2 \;=\; O_\prec \pB{r(\Lambda) + \Theta^3 + \wt \Gamma^3 r(\Lambda)^3 + \wt \Gamma^2 \Theta^4 + \wt \Gamma^6 r(\Lambda)^4}\,.
\end{equation*}
In order to conclude the proof of \eqref{vself for avg v}, we observe that, by the estimates $\Theta \leq \Lambda \leq C M^{-\gamma/4}$, $\wt \Gamma^2 r(\Lambda) \leq 1$, and $\Lambda \leq M^{-\gamma/4} \wt \Gamma^{-1}$, we have 
\begin{equation*}
\Theta^3 \;\leq\; C M^{-\gamma/4} \Theta^2 \,, \qquad \wt \Gamma^3 r(\Lambda)^3 \;\leq\;
 r(\Lambda)\,, \qquad 
\wt \Gamma^2 \Theta^4 \;\leq\; \wt \Gamma^2 \Lambda^2 \Theta^2 \;\leq\; M^{-\gamma/2} \Theta^2\,, \qquad \wt \Gamma^6 r(\Lambda)^4 \;\leq\; r(\Lambda)\,.
\end{equation*}
Putting everything together, we have
\begin{equation*}
(1 - m^2) [v] - m^{-1} [v]^2 \;=\; O_\prec \pb{r(\Lambda) + M^{-\gamma/4} \Theta^2}\,.
\end{equation*}
Hence \eqref{vself for avg v} follows from \eqref{p Lambda p Theta}.
\end{proof}

Next, we establish a bound analogous to Lemma \ref{lem: bootstrap 1}, establishing gaps in the ranges of $\Lambda$ and $\Theta$. To that end, we need to partition $\wt {\f S}$ in two. For the following we fix $\epsilon \in (0,\gamma/12)$ and partition $\wt {\f S} = \wt {\f S}_> \cup \wt {\f S}_\leq$, where
\begin{equation*}
\wt {\f S}_> \;\deq\; \hb{z \in \wt {\f S} \col \sqrt{\kappa + \eta} > M^{\epsilon} (M \eta)^{-1/3}}\,,
 \qquad \wt {\f S}_\leq \;\deq\; \hb{z \in \wt {\f S} \col \sqrt{\kappa + \eta} \leq M^{\epsilon} (M \eta)^{-1/3}}.
\end{equation*}
The bound relies on \eqref{vself for avg v}, whereby one of the two terms on the left-hand side of \eqref{vself for avg v} is estimated in terms of all the other terms, which are regarded as an error. In $\wt {\f S}_>$ we shall estimate the first term on the left-hand side of \eqref{vself for avg v}, and in $\wt {\f S}_\leq$ the second. Figure \ref{fig: gap} summarizes the estimates on $\Theta$ of Lemma \ref{lem: gap in S >} and \ref{lem: gap leq}.
\begin{figure}[ht!]
\begin{center}
\includegraphics{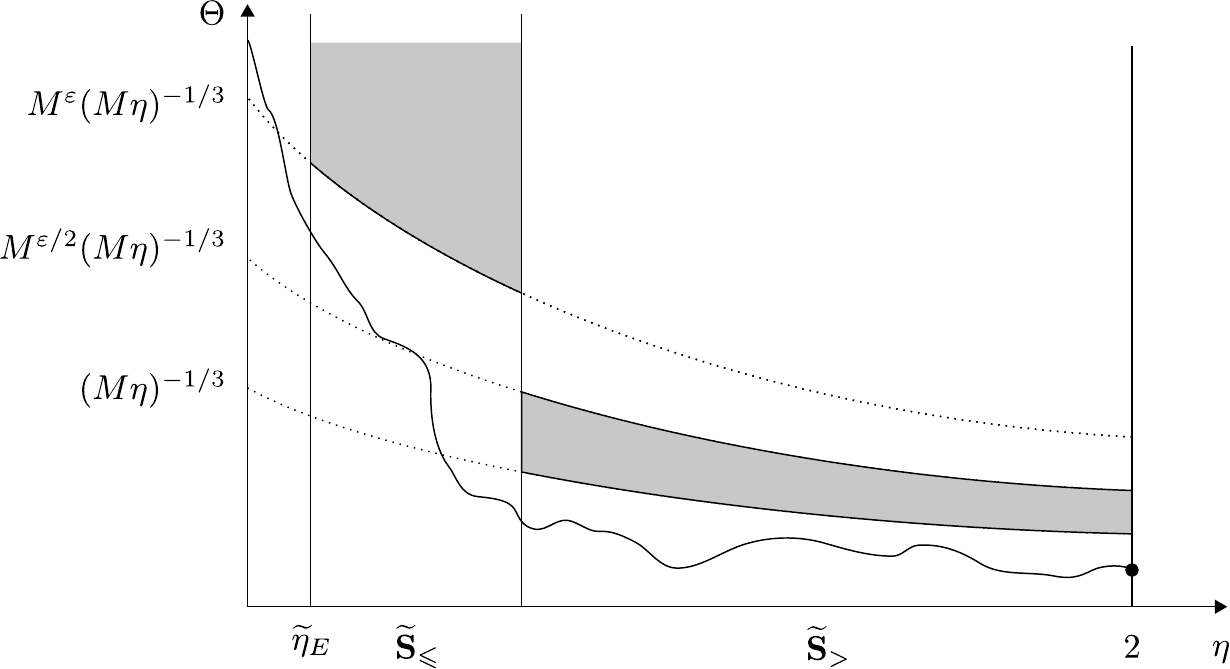}
\end{center}
\caption{The $(\eta, \Theta)$-plane for a fixed $E$ near the edge (i.e.\ with small $\kappa$). The shaded regions are forbidden with high probability by Lemmas \ref{lem: gap in S >} and \ref{lem: gap leq}. The initial estimate, given by Lemma \ref{lem: initial estimate}, is marked with a black dot. The graph of $\Theta = \Theta(E + \ii \eta)$ is continuous, and hence lies beneath the shaded regions. \label{fig: gap}}
\end{figure}

We begin with the domain $\wt {\f S}_>$. In this domain, the following lemma roughly says that if $\Theta \leq 
 M^{\epsilon/2} (M \eta)^{-1/3}$ 
and $\Lambda \leq M^{-\gamma/4} \wt\Gamma^{-1}$ then we get the improved bounds
$\Theta \prec  (M \eta)^{-1/3}$,  $\Lambda\prec M^{-\gamma/3} \wt\Gamma^{-1}$,
i.e.\ we gain a small power of $M$.
These improvements will be fed into the continuity argument
as before.

\begin{lemma} \label{lem: gap in S >}
Let $\epsilon \in (0,\gamma/12)$. Define the $z$-dependent indicator function
\begin{equation*}
\chi \;\deq\; \indB{\Theta \leq M^{\epsilon/2} (M \eta)^{-1/3}}
\end{equation*}
and recall the indicator function $\phi$ from \eqref{def phi with gap}.
In $\wt {\f S}_>$ we have the bounds
\begin{equation} \label{gap in S >}
\phi \chi \, \Theta \;\prec\; (M \eta)^{-1/3}\,, \qquad \phi \chi \, \Lambda \;\prec\; M^{-\gamma/3} \wt \Gamma^{-1}\,.
\end{equation}
\end{lemma}
\begin{proof}
From the definition of $\wt {\f S}_>$ and \eqref{1-msquare} 
we  get
\begin{equation*}
\phi \chi \, \abs{[v]} \;=\; \phi \chi \, \Theta \;\leq\; M^{\epsilon/2} (M \eta)^{-1/3} \;\leq\; M^{-\epsilon/2} \sqrt{\kappa + \eta} \;\leq\; C M^{-\epsilon/2} \abs{1 - m^2}\,.
\end{equation*}
Therefore, on the event $\{\phi \chi = 1\}$, in \eqref{vself for avg v} we may absorb the second term on the left-hand side and the second term on the right-hand side into the first term on the left-hand side:
\begin{equation*}
\phi \chi \, (1 - m^2) [v] \;=\; \phi O_\prec \pb{q(\Theta)}\,.
\end{equation*}
Recalling $|1-m^2|\asymp \sqrt{\kappa+\eta}$ (see \eqref{1-msquare}), $\im m \leq C\sqrt{\kappa+\eta}$
(see \eqref{lower bound on im msc}),
\eqref{p Lambda p Theta}, $\abs{[v]} =\Theta$, 
and the definition of $\wt {\f S}_>$, we get
\begin{align*}
\phi \chi \, \Theta &\;\prec\; \phi \chi \, (\kappa + \eta)^{-1/2} \pBB{\sqrt{\frac{\im m}{M \eta}} + \sqrt{\frac{\Theta}{M \eta}} + \frac{\wt \Gamma}{M \eta}}
\\
&\;\leq\; (\kappa + \eta)^{-1/4} (M \eta)^{-1/2} + (\kappa + \eta)^{-1/2} M^{\epsilon / 2} (M \eta)^{-2/3} + (\kappa + \eta)^{-1/2} \wt \Gamma (M \eta)^{-1}
\\
&\;\leq\; (M \eta)^{-1/3}\,.
\end{align*}
What remains is to estimate $\Lambda$. From \eqref{Lambda estimated in terms of Theta}, the bound $\wt \Gamma^2 \sqrt{\im m (M \eta)^{-1}} \leq M^{-\gamma}$ from the definition of $\wt {\f S}$, and the estimate $ \phi \wt\Gamma\Theta \leq \phi \wt\Gamma \Lambda \leq 1$ 
we get
\begin{align*}
\phi \chi \, \Lambda &\;\prec\; \phi \chi \, \Theta + M^{-\gamma} \wt \Gamma^{-1}+ \wt \Gamma \sqrt{\wt \Gamma^{-1} (M \eta)^{-1}} + \wt \Gamma^2 (M \eta)^{-1} 
\\
&\;\prec\; (M \eta)^{-1/3} + M^{-\gamma/2} \wt \Gamma^{-1} + M^{-\gamma} \wt \Gamma^{-1}
\\
&\;\leq\; 2 M^{-\gamma/3} \wt \Gamma^{-1}\,.
\end{align*}
This concludes the proof.
\end{proof}
Next, we establish a gap in the range of $\Lambda$, in the domain $\wt {\f S}_\leq$. 
To that end, we improve the estimate on $\Lambda$ from $\Lambda \leq M^{-\gamma/4} \wt\Gamma^{-1}$ to
 $\Lambda\prec M^{\e-\gamma/3} \wt\Gamma^{-1}$ as before.
In this regime there is no need for a gap in $\Theta$, i.e.\ the continuity argument will be performed on the value of $\Lambda$ only.

\begin{lemma} \label{lem: gap leq}
In $\wt {\f S}_\leq$ we have the bounds
\begin{equation} \label{gap in S leq}
\phi \Theta \;\prec\; M^\epsilon (M \eta)^{-1/3}\,, \qquad \phi \Lambda \;\prec\; M^{\epsilon - \gamma/3} \wt \Gamma^{-1}\,.
\end{equation}
\end{lemma}
\begin{proof}
We write \eqref{vself for avg v} as
\begin{equation*}
\phi [v] (1 - m^2 - m^{-1} [v]) \;=\; \phi \, O_\prec \pb{q(\Theta) + M^{-\gamma/4} \Theta^2}.
\end{equation*}
Solving this quadratic relation for $[v]$, we get
\begin{equation}\label{phith}
\phi \Theta \;\prec\; \abs{1 - m^2} + \phi \sqrt{q(\Theta) + M^{-\gamma/4} \Theta^2}\,.
\end{equation}
Using \eqref{lower bound on im msc},
the bound $\wt\Gamma \leq M^{-\gamma/3} (M\eta)^{1/3} \leq (M\eta)^{1/3}$ from 
 the definition of $\wt {\f S}$, and Young's inequality, we estimate
\begin{align*}
\sqrt{q(\Theta) + M^{-\gamma/4} \Theta^2} &\;\leq\; (\im m)^{1/4} (M \eta)^{-1/4} + \Theta^{1/4} (M \eta)^{-1/4} + \wt \Gamma^{\,1/2} (M \eta)^{-1/2} + M^{-\gamma/8} \Theta
\\
&\;\leq\; C \sqrt{\kappa + \eta} + C M^\epsilon (M \eta)^{-1/3} + C M^{-\epsilon} \Theta\,.
\end{align*}
Plugging this bound into \eqref{phith}, 
together with \eqref{1-msquare} and
the definition of $\wt {\f S}_\leq$, we find
\begin{equation*}
\phi \Theta \;\prec\; \sqrt{\kappa + \eta} + M^\epsilon (M \eta)^{-1/3} \;\leq\; 2 M^\epsilon (M \eta)^{-1/3}\,.
\end{equation*}
This proves the first bound of \eqref{gap in S leq}.

What remains is the estimate of $\Lambda$. From \eqref{Lambda estimated in terms of Theta} and the bounds $\wt\Gamma \leq M^{-\gamma/3} (M\eta)^{1/3}$ and $\wt \Gamma^2 \sqrt{\im m (M \eta)^{-1}} \leq M^{-\gamma}$ from the definition of $\wt {\f S}$, we get
\begin{align*}
\phi \Lambda &\;\prec\; \phi  \Theta + M^{-\gamma} \wt \Gamma^{-1} + \wt \Gamma \sqrt{\wt \Gamma^{-1} (M \eta)^{-1}} + \wt \Gamma^2 (M \eta)^{-1} 
\\
&\;\prec\; M^\epsilon (M \eta)^{-1/3} + M^{-\gamma/2} \wt \Gamma^{-1} + M^{-\gamma} \wt \Gamma^{-1}
\\
&\;\leq\; 2 M^{\epsilon -\gamma/3} \wt \Gamma^{-1}\,.
\end{align*}
This concludes the proof.
\end{proof}

We now have all of the ingredients to complete the proof of Proposition \ref{prop: rough bound with gap}.

\begin{proof}[Proof of Proposition \ref{prop: rough bound with gap}]
The proof is a continuity argument similar to the proof of Proposition \ref{prop: rough bound 1}. In a first step, we prove that
\begin{equation} \label{result in S >}
\Lambda \;\prec\; M^{-\gamma/3} \wt \Gamma^{-1}\,, \qquad \Theta \;\prec\; (M \eta)^{-1/3}\,.
\end{equation}
in $\wt {\f S}_>$. The continuity argument is almost identical to that following \eqref{gap for one z}; the only difference is that we keep track of the two parameters $\Lambda$ and $\Theta$. The required gaps in the ranges of $\Lambda$ and $\Theta$ are provided by \eqref{gap in S >}, and the argument is closed using the large-$\eta$ estimate from Lemma \ref{lem: initial estimate}, which yields $\Theta \leq \Lambda \prec M^{-1/2}$ for $\eta = 2$.

In a second step, we prove that
\begin{equation*}
\Lambda \;\prec\; M^{\epsilon -\gamma/4} \wt \Gamma^{-1}\,, \qquad \Theta \;\prec\; M^\epsilon (M \eta)^{-1/3}
\end{equation*}
in $\wt {\f S}_\leq$. This is again a continuity argument almost identical to that following \eqref{gap for one z}. Now we establish a gap only in the range of  $\Lambda$. The gap is provided by \eqref{gap in S leq} (recall that by definition of $\epsilon$ we have $\epsilon - \gamma/3 < - \gamma/4$), and the argument is closed using the bound \eqref{result in S >} at the boundary of the domains $\wt {\f S}_>$ and $\wt {\f S}_\leq$.

The claim now follows since we may choose $\epsilon \in (0, \gamma/12)$ to be arbitrarily small. This concludes the proof of Proposition \ref{prop: rough bound with gap}.
\end{proof}

\subsection{An improved bound on $\Lambda$ in terms of $\Theta$}\label{sec:lambdatheta}

In  \eqref{Lambda estimated in terms of Theta} we already estimated
$\Lambda$ in terms of $\Theta$; the goal of this section
is to improve this bound by removing the factor $\wt\Gamma$
from that estimate. We do this using the Fluctuation Averaging Theorem, but we stress that
the removal of a factor $\wt\Gamma$ is not the main rationale
for using the fluctuation averaging mechanism. Its fundamental use
will take place in Lemma \ref{lem:psquare} below.
A technical consequence of invoking fluctuation averaging
is that we have to use deterministic control parameters instead of random ones.
Thus, we introduce a deterministic control parameter $\Phi$ that captures the size of the random control parameter $\Theta$ through the relation $\Theta \prec \Phi$.
Throughout the following we shall make use of the control parameter
\begin{equation*}
p(\Phi) \;\deq\; \sqrt{\frac{\im m + \Phi}{M \eta}} + \frac{1}{M \eta}\,,
\end{equation*}
which differs from $q(\Phi)$ only by a factor $\wt\Gamma$
in the second term.

\begin{lemma} \label{lem: estimate of Lambdao 2} 
Suppose that   $\Lambda \prec \Psi$ and $\Theta \prec \Phi$ in $\wt {\f S}$
for some deterministic control parameters $\Psi$ and $\Phi$ satisfying
\begin{equation} \label{condition on Psi 2}
c M^{-1/2} \;\leq\; \Psi \;\leq\; C M^{-\gamma/4} \wt \Gamma^{-1}\,, \qquad \Phi \;\leq\; C M^{-\gamma/4} \wt \Gamma^{-1}\,.
\end{equation}
Then
\begin{equation}\label{LL}
\Lambda_o + \abs{Z_i} \;\prec\; p(\Phi)\,, \qquad \Lambda \;\prec\; p(\Phi) + \Phi\,.
\end{equation}
\end{lemma}

We remark that, by Proposition \ref{prop: rough bound with gap},
the choice  $\Psi = M^{-\gamma/4} \wt \Gamma^{-1}$ and 
$\Phi = (M \eta)^{-1/3} \leq M^{-\gamma/4} \wt \Gamma^{-1}$ satisfies the assumptions of Lemma \ref{lem: estimate of Lambdao 2}.

\begin{proof}[Proof of Lemma \ref{lem: estimate of Lambdao 2}]
Choosing $\phi = 1$ in Lemma \ref{lemma: Lambdao} and recalling \eqref{varrho geq c}, we get
\begin{equation} \label{Lambdao estimate 3}
\Lambda_o + \abs{\Upsilon_i} \;\prec\; r(\Psi)\,, \qquad r(\Psi) \;\deq\; \sqrt{\frac{\im m + \Psi}{M \eta}}\,.
\end{equation}
In order to estimate $\Lambda_d$, as in \eqref{expanded vself 2}, we expand \eqref{vself for exp} to get
\begin{equation}\label{exp2}
-\sum_k s_{ik} v_k + \Upsilon_i \;=\; -m^{-2} v_i + O_\prec(\Psi^2)\,.
\end{equation}
As in the proof of \eqref{expanded vself 2} and \eqref{exp1}, 
the expansion of
\eqref{vself for exp} is only possible on the event $\{\Lambda\leq M^{-\delta}\}$
for some $\delta>0$. By $\Lambda\prec \Psi$ and \eqref{condition on Psi 2}, the indicator function of this event is $1 + O_\prec(0)$; the contribution $O_\prec(0)$
of the complementary event can be absorbed in the error term $O_\prec(\Psi^2)$.

Subtracting the average $N^{-1}\sum_i$ from both sides of \eqref{exp2}
 and estimating $m^2$ by a constant (see \eqref{m is bounded}) yields
\begin{equation} \label{Lambdad estimate 2}
\absb{v_i - [v]} \;\leq\; C \absbb{\sum_k s_{ik} \pb{v_k - [v]} - \pb{\Upsilon_i - [\Upsilon]}} + O_\prec(\Psi^2) \;\prec\; \wt \Gamma \Psi^2 + r(\Psi)\,,
\end{equation}
where in the last step we used the fluctuation averaging estimate  \eqref{averaging without Q avg}  and $\abs{\Upsilon_i} \prec r(\Psi)$ from \eqref{Lambdao estimate 3}.
Together with $\abs{[v]} = \Theta\prec\Phi$ ,
this gives the estimate $\Lambda_d \prec \Gamma \Psi^2 + \Phi + r(\Psi)$.
Combining it with the bound \eqref{Lambdao estimate 3},   we conclude that 
\begin{equation}\label{quadr}
\Lambda \prec \wt \Gamma \Psi^2 + \Phi + r(\Psi).
\end{equation}
Now fix $\epsilon \in (0, \gamma/4)$. Using 
 the assumption $\wt \Gamma \Psi \leq C M^{-\gamma/4} \leq M^{-\epsilon}$, we conclude: if $\Psi$ and $\Phi$ satisfy \eqref{condition on Psi 2} then
\begin{equation} \label{iteration 2}
\Lambda \;\prec\; \Psi \qquad \Longrightarrow \qquad \Lambda \;\prec\; F(\Psi, \Phi)\,,
\end{equation}
where we defined
\begin{equation*}
F(\Psi, \Phi) \;\deq\; M^{-\epsilon} \Psi + \Phi + \sqrt{\frac{\im m}{M \eta}} + \frac{M^\epsilon}{M \eta}\,,
\end{equation*}
which plays a role similar to $F(\Psi)$ in Proposition~\ref{prop: optimal simple}.
(Here we estimated $\sqrt{\Psi (M \eta)^{-1}}$ in $r(\Psi)$
 by $M^{-\e} \Psi+M^\e (M\eta)^{-1}$.) 
From \eqref{lower bound on im msc} and the definition of $\wt {\f S}$ it easily follows that if $(\Psi, \Phi)$ satisfy \eqref{condition on Psi 2} then so do $(F(\Psi, \Phi), \Phi)$. Therefore iterating \eqref{iteration 2} $\ceil{\epsilon^{-1}}$ times and using the fact that $\epsilon \in (0,\gamma/4)$ was arbitrary yields
\begin{equation} \label{new bound for Lambda}
\Lambda \;\prec\; \sqrt{\frac{\im m}{M \eta}} + \frac{1}{M \eta} + \Phi\,.
\end{equation}
This implies the claimed bound \eqref{LL} on $\Lambda$.
Calling the right-hand side of \eqref{new bound for Lambda} $\Psi$, we find
\begin{equation} \label{r leq p}
r(\Psi) \;\leq\; C p(\Phi)\,.
\end{equation}
Hence the claimed bound  \eqref{LL}
on $\Lambda_o$ and $Z_i$ follows from \eqref{Lambdao estimate 3}.
\end{proof}

\subsection{Iteration for $\Theta$ and conclusion of the proof of Theorem \ref{thm: with gap}}\label{sec:itertheta}

Next, we prove the following version of \eqref{vself for exp}, which is the key tool for estimating $\Theta$.
\begin{lemma}\label{lem:psquare}
Let $\Phi$ be some deterministic control parameter satisfying $\Theta \prec \Phi$ in $\wt {\f S}$. Then
\begin{equation} \label{vself for optimal bound}
(1 - m^2) [v] - m^{-1}[v]^2 \;=\; O_\prec \pb{p(\Phi)^2 + M^{-\gamma/4} \Phi^2}\,.
\end{equation}
\end{lemma}
Notice that this bound is stronger than the previous
formula \eqref{vself for avg v}, as the power of $p(\Phi)$ is two instead of one.
The improvement is due to using fluctuation averaging in $[\Upsilon]$. 
Otherwise the proof is very similar to that of \eqref{vself for avg v}.

\begin{proof}
By Proposition \ref{prop: rough bound with gap}, we may assume that 
\begin{equation} \label{condition on Phi}
\Phi \;\leq\; M^{-\gamma/4}\wt \Gamma^{-1}
\end{equation}
since $\Theta \leq \Lambda \prec M^{-\gamma/4}\wt \Gamma^{-1}$. From Lemma \ref{lem: estimate of Lambdao 2} we get $\Lambda_o + \abs{Z_i} \prec p(\Phi)$ and $\Lambda \prec \Psi$, where
\begin{equation} \label{optimal Psi}
\Psi \;\deq\; p(\Phi) + \Phi\,.
\end{equation}
By definition of $\wt {\f S}$ and \eqref{condition on Phi},
we find that $\Psi \leq 2 M^{-\gamma/4} \wt \Gamma^{-1}$.

Now we expand the right-hand side of \eqref{vself for exp} exactly as in \eqref{selfc expanded deg 2} to get
\begin{equation} \label{self cons expanded 4}
-m^2 \sum_k s_{ik} v_k + m^2 \Upsilon_i \;=\; -v_i  + m^{-1} v_i^2 + O_\prec(\Psi^3)\,.
\end{equation}
Using Theorem \ref{thm: averaging with Lambdao} and the bound $\Lambda_o \prec p(\Phi)$ from Lemma \ref{lem: estimate of Lambdao 2}, we may prove, exactly as in Lemma \ref{lem: avg Upsilon}, that $\abs{[\Upsilon]} \prec p(\Phi)^2$.
Taking the average over $i$ in \eqref{self cons expanded 4} therefore yields
\begin{equation} \label{averaged vself expanded}
(1 - m^2) [v] - m^{-1} \frac{1}{N} \sum_i v_i^2 \;=\; -m^2 [\Upsilon] + O_\prec(\Psi^3) \;=\; O_\prec \pb{p(\Phi)^2 + \Psi^3}\,.
\end{equation}
Using the estimates \eqref{Lambdad estimate 2} and \eqref{r leq p}, we write the quadratic term on the left-hand side as
\begin{equation*}
\frac{1}{N} \sum_i v_i^2 \;=\; [v]^2 + \frac{1}{N} \sum_i \pb{v_i - [v]}^2 \;=\; [v]^2 + O_\prec \pB{\pb{\wt \Gamma \Psi^2 + p(\Phi)}^2} \;=\; [v]^2 + O_\prec \pb{M^{-\gamma/2} \Psi^2 + p(\Phi)^2}\,
\end{equation*}
where we also used $\wt\Gamma \Psi\leq 2M^{-\gamma/4}$, as observed after \eqref{optimal Psi}.
From \eqref{averaged vself expanded} we therefore get
\begin{equation*}
(1 - m^2) [v] - m^{-1}[v]^2 \;=\; O_\prec \pb{p(\Phi)^2 + M^{-\gamma/4} \Psi^2}\,.
\end{equation*}
The claim follows from \eqref{optimal Psi}.
\end{proof}

The bound on $\Theta$ will follow by iterating the following estimate.
\begin{lemma} \label{lem: iteration for optimal bound}
Fix $\epsilon \in (0,\gamma/12)$ and suppose that $\Theta \prec \Phi$ in $\wt {\f S}$ for some deterministic control parameter $\Phi$.
\begin{itemize}
\item[(i)] If $\Phi \geq M^{3 \epsilon} (M \eta)^{-1}$ then
\begin{equation} \label{iteration for optimal bound}
\Theta \;\prec\; M^{-\epsilon} \Phi\,. 
\end{equation}
\item[(ii)] If $\abs{E} \geq 2$,
$\frac{M^{3\epsilon}}{M(\kappa+\eta)} \leq \Phi \leq M^\epsilon \sqrt{\kappa+\eta}$,
and $M\eta\sqrt{\kappa+\eta} \geq M^{2\e}$, then
\begin{equation} \label{iteration for optimal bound2}
\Theta \;\prec\; \frac{1}{(M\eta)^2\sqrt{\kappa+\eta}}+M^{-\epsilon} \Phi\,.
\end{equation}
\end{itemize}
\end{lemma}
\begin{proof}
We begin by partitioning $\wt {\f S} = \wt {\f S}\,\!^> \cup \wt {\f S}\,\!^\leq$. This partition is analogous to the partition $\wt {\f S} = \wt {\f S}_> \cup \wt {\f S}_\leq$ from Section \ref{sec: rough bound with gap}, and will determine which of the two terms in the left-hand side of \eqref{vself for optimal bound} is estimated in terms of the others. Here
\begin{equation*}
\wt {\f S}\,\!^> \;\deq\; \hb{z \in \wt {\f S} \col \sqrt{\kappa + \eta} > M^{-\epsilon} \Phi}\,,
\qquad \wt {\f S}\,\!^\leq \;\deq\; \hb{z \in \wt {\f S} \col \sqrt{\kappa + \eta} 
\leq M^{-\epsilon} \Phi}.
\end{equation*}

We begin with the domain $\wt {\f S}\,\!^>$. Let $K > 0$ be a constant large enough that
\begin{equation*}
\sqrt{\kappa + \eta} \;\leq\; \frac{K}{2} \absb{1 - m^2} \abs{m}\,;
\end{equation*}
such constant exists by \eqref{m is bounded} and \eqref{1-msquare}. Define the indicator function 
\begin{equation}\label{psidef}
\psi \;\deq\; \indb{\Theta \leq \sqrt{\kappa + \eta} / K}\,.
\end{equation}
Hence on the event $\{\psi = 1\}$ we may absorb the quadratic term on the left-hand side of \eqref{vself for optimal bound} into the linear term, to get the bound
\begin{equation}\label{psiTheta}
\psi \Theta \;\prec\; (\kappa + \eta)^{-1/2} \pbb{\frac{\im m + \Phi}{M \eta} + \frac{1}{(M \eta)^2} + M^{-\gamma / 4} \Phi^2} \;\leq\; C \frac{M^\epsilon}{M \eta} + M^{\epsilon -\gamma/4} \Phi \;\leq\; C M^{-2 \epsilon} \Phi\,,
\end{equation}
where in the second step we used \eqref{lower bound on im msc}, the assumption $(M \eta)^{-1} \leq
M^{-3\e}\Phi \leq \Phi$,  and the definition of $\wt {\f S}\,\!^>$.
We conclude that in $\wt {\f S}\,\!^>$ we have
\begin{equation} \label{estimate for iteration 2}
\psi \Theta \;\prec\; M^{- 2 \epsilon} \Phi \;\leq\; M^{-\epsilon} \sqrt{\kappa + \eta}\,,
\end{equation}
where in the last step we used the definition of
 $\wt {\f S}\,\!^>$. This means that there is a gap of order $\sqrt{\kappa + \eta}$ 
between the bound in the definition of $\psi$ 
in \eqref{psidef} and the right-hand side of \eqref{estimate for iteration 2}. Moreover, by Proposition \ref{prop: rough bound with gap} we have $\Theta \prec M^{-\epsilon} \sqrt{\kappa + \eta}$ for $\eta = 2$. Hence a continuity argument on $\Theta$, similar to the proof of Proposition \ref{prop: rough bound 1}, yields \eqref{iteration for optimal bound} in $\wt {\f S}\,\!^>$.

Let us now consider the domain $\wt {\f S}\,\!^\leq$. We write the left-hand side of \eqref{vself for optimal bound} as $(1 - m^2 - m^{-1}[v]) [v]$. 
Solving the resulting equation for $[v]$, as in the proof of \eqref{phith}, yields the bound
\begin{equation}\label{Theta2}
\Theta \;\prec\; \abs{1 - m^2} + p(\Phi) + M^{-\gamma/8} \Phi \;\leq\; C M^{-\epsilon} \Phi
 + \sqrt{\frac{\im m + \Phi}{M \eta}} + \frac{1}{M \eta} \;\leq\; C M^{-\epsilon} \Phi + \frac{M^\epsilon}{M \eta} \;\leq\; C M^{-\epsilon} \Phi\,,
\end{equation}
where we used the definition of $\wt {\f S}\,\!^\leq$ and the bounds \eqref{1-msquare}
 and \eqref{lower bound on im msc}. This proves \eqref{iteration for optimal bound} in $\wt {\f S}\,\!^\leq$,
and hence completes the proof of part (i) of Lemma~\ref{lem: iteration for optimal bound}.

The proof of part (ii) is analogous. In this case we are in the domain $\wt {\f S}\,\!^>$, and use the estimate  $\im m\leq C\eta(\kappa+\eta)^{-1/2}$ from \eqref{lower bound on im msc}
instead of $\im m\leq C\sqrt{\kappa+\eta}$ in \eqref{psiTheta}. Using the other assumptions in part (ii),
we have
\begin{equation}\label{psiTheta1}
\psi \Theta \;\prec\;
  \frac{1}{(M\eta)^2\sqrt{\kappa+\eta}} +  C M^{-2 \epsilon} \Phi\, \;\leq\; M^{-\e} \sqrt{\kappa + \eta}\,,
\end{equation}
which replaces \eqref{psiTheta} and  \eqref{estimate for iteration 2}. The rest of the argument is unchanged.
\end{proof}

Armed with Lemma \ref{lem: iteration for optimal bound}, we may now complete the proof of Theorem \ref{thm: with gap}. Fix $\epsilon \in (0, \gamma/12)$. From Proposition \ref{prop: rough bound with gap} we get that $\Theta \prec \Phi_0$ for $\Phi_0 \deq (M \eta)^{-1/3} + M^{3 \epsilon} (M \eta)^{-1}$. Iteration of Lemma \ref{lem: iteration for optimal bound} therefore implies that, for all $k \in \N$, we have $\Theta \prec \Phi_k$ where
\begin{equation*}
\Phi_{k+1} \;\deq\; \frac{M^{3 \epsilon}}{M \eta} + M^{-\epsilon} \Phi_k \;\leq\; C_k \pbb{\frac{M^{3 \epsilon}}{M \eta} + M^{-\epsilon k } \Phi_0}\,.
\end{equation*}
Choosing $k = \ceil{\epsilon^{-1}}$ yields $\Theta \prec M^{3 \epsilon} (M \eta)^{-1}$. Since $\epsilon$ can be made as small as desired, we therefore obtain $\Theta \prec (M \eta)^{-1}$. This is \eqref{m-mest 2}.

In the regime $\abs{E} \geq 2$, the 
same argument with the better iteration bound \eqref{iteration for optimal bound2} yields \eqref{m-mestout}.
The iteration can be started with $\Phi_0= M^{3\e} (M\eta)^{-1}$ from  \eqref{m-mest 2}.

Finally, the bound $\Lambda \prec \Pi$ in \eqref{Gijest 2} follows from \eqref{m-mest 2} and Lemma \ref{lem: estimate of Lambdao 2}. This concludes the proof of Theorem \ref{thm: with gap}.

\section{Density of states and eigenvalue locations} \label{sec: dos}

\newcommand{\hp}[2]{with $(#1,#2)$-high probability\xspace}
\newcommand{\fn}{{\fra n}}
\newcommand{\al}{\alpha}

In this section we apply the local semicircle law to obtain information
on the density of states and on the location of eigenvalues.
The techniques used here have been developed in a series
of papers \cite{ESY2, ESYY, EYYrigi, EKYY1}.

The first result is to translate the local semicircle law,
Theorem~\ref{thm: with gap}, into a statement on the
counting function of the eigenvalues. Let $\lambda_1 \leq \lambda_2 \leq \cdots \leq \lambda_N$ denote the ordered eigenvalues of $H$, and
recall the semicircle density $\varrho$ defined in \eqref{definition of msc}.
We define the distribution functions
\begin{equation}\label{def:fn}
n(E) \;\deq\; \int_{-\infty}^E \varrho(x)  \, \dd x\,, \qquad \fn_N(E)  \;\deq\; \frac{1}{N} \absb{\h{\alpha \col \lambda_\alpha \leq E}}
\end{equation}
for the semicircle law and the empirical eigenvalue density of $H$.
 Recall also the definition \eqref{def: kappa} of $\kappa_x$ for $x \in \R$
and the definition \eqref{def wt eta E} of $\wt \eta_x$ for $|x|\leq 10$.
The following result is proved in Section~\ref{sec:count} below.

\begin{lemma} \label{lemma: counting function estimate}
Suppose that \eqref{m-mest 2} holds uniformly in $z \in \wt {\f S}$, i.e.\ for $\abs{E} \leq 10$ and $\wt \eta_E \leq \eta \leq 10$ we have
\begin{equation} \label{lsc hypothesis}
\abs{m_N(z)-m(z)} \;\prec\; \frac{1}{M\eta}\,.
\end{equation}
For given $E_1 < E_2$ in $[-10, 10]$ we abbreviate
\begin{equation} \label{def of cal E}
\wt\eta \;\deq\; \max \hb{ \wt\eta_{E} \col E\in [E_1, E_2]}.
\end{equation}
Then, for $-10 \leq E_1 < E_2 \leq 10$, we have
\begin{equation} \label{main estimate on n - nsc}
\absB{\pb{{\fra n}_N(E_2) -  {\fra n}_N(E_1)} - \pb{n(E_2) - n(E_1)}} \;\prec\; \wt \eta\,.
\end{equation}
\end{lemma}

The accuracy of the estimate \eqref{main estimate on n - nsc} depends 
on $\wt\Gamma$ (see \eqref{bound for Gamma tilde} for explicit bounds on $\wt \Gamma$),
since $\wt\Gamma$ determines $\wt\eta_E$, the smallest scale
on which the local semicircle law (Theorem~\ref{thm: with gap}) holds around the energy $E$.
In the regime away from the spectral edges $E=\pm 2$ and away from $E=0$, the 
parameter $\wt\Gamma$ is essentially bounded (see the example (i) from Section \ref{sec: example});
in this case $\wt \eta_E \asymp M^{-1}$ (up to an irrelevant logarithmic factor).
For $E$ near $0$, the parameter  $\wt \Gamma$ blows up as $E^{-2}$, so that $\wt\eta_E\sim E^{-12} M^{-1}$; however, 
if $S$ has a positive gap $\delta_-$ at the bottom of its spectrum, $\wt \Gamma$ remains bounded in the vicinity 
of $E = 0$ (see \eqref{bound for Gamma tilde}). See Definition \ref{def: gap} in Appendix \ref{app: bounds on varrho}
 for the definition of the spectral gaps $\delta_\pm$.

A typical example of $S$ without a positive gap $\delta_-$ is a $2\times 2$ block matrix
with zero diagonal blocks, i.e.  $s_{ij} = 0$ if $i,j\leq L$ or $L+1\leq i,j \leq N$. 
In this case, the vector ${\f v} = (1,1,\ldots 1, -1, -1, \ldots -1)$ consisting of $L$ ones and
$N-L$ minus ones satisfies $S{\f v} =-{\f v}$, so that $-1$ is in fact an eigenvalue of $S$. Since at energy $E=0$
we have $m^2(z)= m^2(i\eta) = -1 + O(\eta)$, the inverse matrix $(1-m^2S)^{-1}$, even 
after restricting it  to ${\f e}^\perp$, becomes singular as
$\eta\to 0$. Thus, $\wt \Gamma(i\eta)\sim \eta^{-1}$, and the estimates
leading to Theorem~\ref{thm: with gap} become unstable.
The corresponding random matrix
has the form
\begin{equation*}
H \;=\;
\begin{pmatrix} 0  & A \\ A^* & 0
\end{pmatrix}
\end{equation*}
where $A$ is an $L\times (N-L)$ rectangular matrix with independent centred entries.
The eigenvalues of $H$ are the square roots (with both signs) of the eigenvalues of
the random covariance matrices $AA^*$ and $A^*A$, whose spectral density is asymptotically given by the \emph{Marchenko-Pastur law} \cite{MP}.
The instability near $E = 0$ arises from the fact that $H$ has a macroscopically large kernel unless
$L/N\to 1/2$. In the latter case the support of the Marchenko-Pastur law extends to zero and 
in fact the density diverges as $E^{-1/2}$. We remark that a local version of the Marchenko-Pastur
 law was given in \cite{ESYY}
for the case when the limit of $L/N$ differs from 0, 1/2 and $\infty$; the ``hard edge'' case, $L/N\to 1/2$, in which the density near the lower spectral edge is singular, was treated in \cite{CMS}.

This example shows that the vanishing of $\delta_-$ may lead to a very different behaviour of the spectral statistics.
Although our technique is also applicable to random covariance matrices, for simplicity
in this section we assume that $\delta_-\geq c$ for some positive constant $c$. By  Proposition~\ref{prop: spectrum of S}, this holds for random band matrices, for full Wigner 
matrices (see Definition~\ref{def: Wigner}), and for their combinations; these examples are our main interest in this paper.

Under the condition $\delta_-\geq c$, the upper bound of \eqref{bound for Gamma tilde} yields
\begin{equation} \label{Gaim}
\wt\Gamma(E+i\eta) \;\leq\; \frac{C \log N}{\delta_++\theta}\,,
\end{equation}
where $\theta$ was defined in \eqref{def theta} and $\delta_+$ is the upper gap of the spectrum of $S$ given in Definition \ref{def: gap}.
Notice that $\theta$ vanishes near the spectral edge $E =\pm 2$ as $\eta\to 0$.
For the purpose of estimating $\wt \Gamma$, this deterioration is mitigated if the upper gap $\delta_+$
is non-vanishing. While full Wigner matrices satisfy $\delta_+ \geq c$,
the lower bound on $\delta_+$ for band matrices is weaker; see Proposition~\ref{prop: spectrum of S} for a precise statement.

We first give an estimate on $\wt\eta_x$ using the explicit bound \eqref{Gaim}. While not fully optimal, this estimate is sufficient for our purposes and in particular reproduces the correct behaviour when $\delta_+ \geq c$.

\begin{lemma}  Suppose that $\delta_-\geq c$ (so that \eqref{Gaim} holds). 
Then we have for any $\abs{x} \leq 2$
\begin{equation}\label{etab}
\wt \eta_x \;\leq\; \frac{CM^{3\gamma}}{M(\kappa_x+\delta_+ + M^{-1/5})^{7/2}}\,.
\end{equation}
In the regime $2 \leq \abs{x} \leq 10$ we have the improved bound
\begin{equation}\label{etabout}
\wt \eta_x \;\leq\;  \frac{CM^{3\gamma}}{M( \sqrt{\kappa_x}+\delta_+ + M^{-1/5})^3}\,.
\end{equation}
\end{lemma}

\begin{proof}
For any $\abs{x} \leq 2$ define
$\eta_x'$ as the solution of the equation
\begin{equation}\label{solx}
   \sqrt{ \frac{\sqrt{\kappa_x+\eta}}{M\eta}}  \frac{1}{ ( \kappa_x +\eta^{2/3}  +\delta_+)^{2}}
+ \frac{1}{M\eta} \frac{1}{ (  \kappa_x +\eta^{2/3} +\delta_+)^{3}} \;=\; M^{-\frac{3\gamma}{2}} .
\end{equation}
This solution is unique since the left-hand side is decreasing in $\eta$.
An elementary but tedious analysis of \eqref{solx} yields
\begin{equation}\label{etab1}
  \eta'_x \;\leq\; \frac{CM^{3\gamma}}{M(\kappa_x+\delta_+ + M^{-1/5})^{7/2}}\,.
\end{equation}
(The calculation is based on the observation that if $\eta(a+\eta^\al)\le b$ for some $a, b>0$ and $\al \ge0$, then $\eta \le 2b(b^{\frac{\al}{1+\al}}+a)^{-1}$.)
From  \eqref{Gaim},  $\im m(x+i\eta) \leq C\sqrt{\kappa_x+\eta}$ (see \eqref{lower bound on im msc})
and the  simple bound $\theta(x+i\eta) \geq c(\kappa_x +\eta^{2/3})$,
we get for $\eta \geq \eta'_x$
$$
   \sqrt{ \frac{\im m(x+i\eta)}{M\eta}} \wt \Gamma^2(x+i\eta)
+ \frac{1}{M\eta}\wt \Gamma^3(x+i\eta)  \;\le\; C(\log N)^3 M^{-\frac{3\gamma}{2}} .
$$
From the definition \eqref{def S varrho 2} of $\wt{\f S}$, we therefore get $\wt\eta_x \leq \eta_x'$, which proves  \eqref{etab}.

The proof of \eqref{etabout} is similar, but we use $\theta = \sqrt{\kappa+\eta}$ and 
the stronger bound  $\im m \le \eta/\sqrt{\kappa+\eta}$ available in the regime $\abs{x} \geq 2$.
 For $2\leq \abs{x} \leq 10$, define $\eta_x'$ to be the  solution of the equation
\be
   \sqrt{ \frac{1}{M\sqrt{\kappa_x+\eta}}} \frac{1}{ ( \sqrt{\kappa_x+\eta}  +\delta_+)^{2}}
 + \frac{1}{M\eta} \frac{1}{( \sqrt{\kappa_x+\eta}  +\delta_+)^{3}}  \;=\; M^{-\frac{3\gamma}{2}} .
\end{equation}
As for \eqref{etab1}, a tedious calculation yields
$$ 
   \eta_x' \;\leq\; \frac{CM^{3\gamma}}{M( \sqrt{\kappa_x}+\delta_+ + M^{-1/5})^3}\,.
$$
This concludes the proof.
\end{proof}

Next, we obtain an estimate on the extreme eigenvalues. 

\begin{theorem}[Extremal eigenvalues]\label{thm:extr}  Suppose that $\delta_-\geq c$ 
(so that \eqref{Gaim} holds) and that $N^{3/4} \leq M \leq N$. Then we have  
\begin{equation}\label{extreme2}
\norm{H} \;\le\; 2 +  O_\prec (X)\,,
\end{equation}
where we introduced the control parameter
\be\label{def:X}
X \;\deq\; \frac{N^2}{M^{8/3}} +  \pbb{\frac{N}{M^2}}^2 
\qbb{\delta_+ +  \Big( \frac{N}{M^2}\Big)^{1/7}}^{-12}\,.
\ee
In particular, if $\delta_+\geq c$ then
\begin{equation}\label{extreme}
\norm{H} \;\leq\; 2 + O_\prec \pbb{\frac{N^2}{M^{8/3}}}\,.
\end{equation}
\end{theorem}

Note that \eqref{extreme} yields the optimal error bound $O_\prec(N^{-2/3})$ in the case of a 
full and flat Wigner matrix (see Definition \ref{def: Wigner}). 
Under stronger assumptions on the law of the entries of $H$, Theorem \ref{thm:extr} can be improved as follows.

\begin{theorem}\label{thm:extr1} Suppose that the matrix elements $h_{ij}$ have a uniform subexponential decay, 
i.e.\ that there exist positive constants $C$ and $\vartheta$ such that
\be\label{subexp}
\P \pb{\abs{h_{ij}} \geq x^\vartheta \sqrt{s_{ij}}} \;\le\; C \, \me^{-x}\,.
\ee
Then \eqref{extreme2} holds with 
\be\label{def:X1}
  X \;\deq\; M^{-1/4}\,.
\ee
If in addition the law of each matrix entry is symmetric (i.e.\ $h_{ij}$ and $-h_{ij}$ have the same law), then \eqref{extreme2} holds with 
\be\label{def:Xsim}
  X \;\deq\; M^{-2/3}\,.
\ee
\end{theorem}

We remark that  \eqref{def:X1} can obtained via a relatively standard moment method argument combined with refined combinatorics. 
Obtaining the bound \eqref{def:Xsim} is fairly involved; it makes use of the Chebyshev polynomial
representation first used by Feldheim and Sodin \cite{FSo, So1} in this context 
for a special distribution of $h_{ij}$, and extended in \cite{EK2} to general symmetric entries.

\begin{proof}[Proof of Theorem \ref{thm:extr}]
We shall prove a lower bound on the smallest eigenvalue $\lambda_1$
of $H$; the largest eigenvalue $\lambda_N$ may be estimated similarly from above.
Fix a small $\gamma>0$ and set 
\begin{equation*}
 \ell \;\deq\; M^{6\gamma}  \frac{N^2}{M^{8/3}}\,.
\end{equation*}
We distinguish two regimes depending on the location of $\lambda_1$, i.e.\ we decompose
$$
   \ind{\lambda_1 \leq -2-\ell} \;=\;
  \phi_1+\phi_2\,,
$$
where
$$
  \phi_1 \;\deq\; \ind{-3\leq \lambda_1 \leq -2-\ell}\,,
  \qquad  \phi_2 \;\deq\;  \ind{\lambda_1 \leq -3}\,.
$$
In the first regime we further decompose the probability space 
by estimating
$$
   \phi_1 \;\leq\; \sum_{k=0}^{k_0} \phi_{1,k}\,,
  \qquad \phi_{1,k} \;\deq\; \indbb{   -2- \ell - \frac{k+1}{N}
\le\lambda_1\leq -2-\ell- \frac{k}{N}}\,.
$$
The upper bound $k_0$ is the smallest integer such that  $2+\ell +\frac{k_0+1}{N}\geq 3$;
clearly $k_0\leq  N$.
For any $k\leq k_0$ we set
$$
   z_k \;\deq\; E_k + \ii \eta_k\,, \qquad E_k \;\deq\; -2-\kappa_k\,, \qquad
 \kappa_k \;\deq\; \ell +\frac{k}{N}\,, \qquad \eta_k \;\deq\; M^{4\gamma}\frac{N}{M^2\sqrt{\kappa_k}}\,.
$$
Clearly, $\eta_k \leq \kappa_k$ since $M \leq N$.
On the support of $\phi_{1,k}$  we have $\abs{\lambda_1-E_k} \leq C/N \leq \eta_k$, so that we get the lower bound
\begin{equation} \label{low1}
\phi_{1,k} \im m_N(z_k) \;=\; \phi_{1,k} \frac{1}{N}\sum_{\al=1}^N \frac{\eta_k}{(\lambda_\al-E_k)^2+\eta_k^2}
  \;\geq\; \phi_{1,k} \frac{1}{N} \frac{\eta_k}{(\lambda_1-E_k)^2+\eta_k^2} \;\geq\;
  \frac{c}{N\eta_k}
\end{equation}
for some positive constant $c$. On the other hand, by \eqref{lower bound on im msc}, we have
$$
  \im m(z_k) \;\leq\; \frac{C\eta_k}{\sqrt{\kappa_k}}\,.
$$
Therefore we get
\begin{equation} \label{gh}
\phi_{1,k} \absb{\im m_N(z_k) - \im m(z_k)} \;\geq\; \frac{c}{N\eta_k} - \frac{C\eta_k}{\sqrt{\kappa_k}}
 \;\geq\; \frac{c' }{N\eta_k}
\end{equation}
for some positive constant $c'$. Here in the second step we used that $\eta_k/\sqrt{\kappa_k}\leq M^{-\gamma} (N\eta_k)^{-1}$. 

Suppose for now that $\delta_+\geq c$. Then by \eqref{etab} we have the upper bound $\wt \eta_x \leq CM^{3\gamma-1}$, uniformly for
$\abs{x} \leq 10$. Since $\eta_k\geq CM^{4\gamma-1}$ we find that $z_k\in \wt {\f S}$ with $\abs{\re z_k} \geq 2$. Hence \eqref{m-mestout} applies
for $z=z_k$ and we get
\begin{equation}
\absb{ \im m_N(z_k) - \im m(z_k)} \;\prec\; \frac{1}{M\kappa_k} + \frac{1}{(M\eta_k)^2\sqrt{\kappa_k}}
 \;\leq\; C M^{-\gamma}\frac{1}{N\eta_k}\,.
\end{equation}
Comparing this bound with \eqref{gh} we conclude that $\phi_{1,k} \prec 0$ (i.e.\ the event $\{\phi_{1,k} = 1\}$ has very small probability).
Summing over $k$ yields $\phi_1\prec 0$.
Note that in this proof the stronger bound \eqref{m-mestout} outside of the spectrum was essential; the general bound 
of order $(M\eta_k)^{-1}$ from \eqref{m-mest 2} is not smaller than the right-hand side of \eqref{gh}.

The preceding proof of $\phi_1 \prec 0$ assumed the existence of a spectral gap $\delta_+ \geq c$.
The above argument easily carries over to the case without a gap of constant size, in which case we choose
\begin{gather*}
  \ell \;\deq\; M^{6\gamma} \pBB{ \frac{N^2}{M^{8/3}} + \pbb{\frac{N}{M^2}}^2 
\qbb{\delta_+ +  \pbb{\frac{N}{M^2}}^{1/7}}^{-12}}\,,
\\
E_k \;\deq\; -2 - \kappa_k\,,
 \qquad  \kappa_k \;\deq\; \ell + \frac{k}{N}\,,
\qquad
 \eta_k \;\deq\; M^{4\gamma} \pbb{ \frac{N}{M^2\sqrt{\kappa_k}} + \frac{1}{M(\sqrt{\kappa_k}+\delta_+)^3}}\,.
\end{gather*}
The last term in $\eta_k$  guarantees that $z_k\in \wt {\f S}$, by \eqref{etabout}. Then we may repeat the above proof to get $\phi_1 \prec 0$ for the new function $\phi_1$.

All that remains to complete the proof of \eqref{extreme2} and \eqref{extreme} is the estimate $\phi_2 \prec 0$.
Clearly
$$
   \P (\lambda_1\leq -3) \;\leq\; \E \absb{\h{j \col \lambda_j \leq -3}}\,.
$$
In part (2) of Lemma 7.2 in \cite{EYY} it was shown, using the
moment method, that the right-hand side is bounded by $CN^{-c\log \log N}$ provided 
the matrix entries $h_{ij}$ have subexponential decay, i.e.
$$
  \P ( \abs{\zeta_{ij}} \geq x^\al ) \;\leq\; \beta \me^{-x} \qquad (x>0)\,,
$$
for some constants $\al,\beta$ (recall the notation \eqref{def:zeta}). In this paper we only assume polynomial decay,
\eqref{finite moments}. However, the subexponential decay assumption of \cite{EYY} was only used in
the first truncation step, Equations (7.28)--(7.29) in \cite{EYY}, where a 
new set of independent random variables $\wh h_{ij}$ was constructed with the
properties that
\begin{equation}\label{trunc}
\P \pb{\zeta_{ij} = \wh \zeta_{ij}} \;\geq\; 1-\me^{-n}\,, \qquad  \absb{\wh \zeta_{ij}} \;\leq\; n\,, \qquad \E \zeta_{ij} =0\,, \qquad \E \absb{\wh \zeta_{ij}}^2 \;\leq\; \E \abs{\zeta_{ij}}^2 + \me^{-n}
\end{equation}
for $n = (\log N)(\log \log N)$. Under the condition \eqref{finite moments} the same
truncation can be performed, but the estimates in \eqref{trunc} will be somewhat weaker; instead of the exponent $n=(\log N)(\log \log N)$
we get $n = D\log N$ for any fixed $D > 0$. The conclusion of the same proof is that, assuming only \eqref{finite moments}, we have
\begin{equation}\label{conc}
\E \absb{\h{j \col \lambda_j \leq -3}} \;\leq\; N^{-D}
\end{equation}
for any positive number $D$ and for any $N\geq N_0(D)$.  
This guarantees that $\phi_2 \norm{H} \prec 0$.
Together with the estimate $\phi_1  \norm{H} \leq 3 \phi_1 \prec 0$ established above,
this completes the proof of Theorem~\ref{thm:extr}. 
\end{proof}

\begin{proof}[Proof of Theorem \ref{thm:extr1}]
The estimate of $\norm{H}$ with $X=M^{-1/6}$ follows from the proof of part (2) of Lemma 7.2 in \cite{EYY},
by choosing $k=M^{-1/6-\e}$ with any small $\e>0$ in (7.32) of \cite{EYY}. This argument can be
improved to  $X=M^{-1/4}$ by the remark after (7.18) in \cite{EYY}. Finally, the bound with
$X=M^{-2/3}$ under the symmetry condition on the entries of $H$ is proved in Theorem 3.4 of \cite {EK2}.
\end{proof}

Next, we establish an estimate on the normalized counting function $\fn_N$ defined in \eqref{def:fn}.
As above, the exponents are not expected to be optimal, but the estimate is in general sharp if $\delta_+\geq c$.

\begin{theorem}[Eigenvalue counting function]\label{thm:count}
Suppose that $\delta_- \geq c$ (so that \eqref{Gaim} holds). Then
\begin{equation}\label{fn-n}
\sup_{E\in \R} \; \abs{\fn_N(E) - n(E)}  = \ O_\prec (Y),
\end{equation}
where we introduced the control parameter
\be\label{def:Y}
Y \;\deq\; \frac{1}{M} \pbb{\frac{1}{\delta_+ + M^{-1/5}}}^{7/2} \,.
\ee
\end{theorem}

\begin{proof}
First we prove the bound \eqref{fn-n} for any fixed $E \in [-10,10]$.
Define the dyadic energies $E_k \deq -2 - 2^k (\delta_++M^{-1/5})$. By \eqref{etab} we have for all $k \geq 0$
\begin{equation*}
  \max \hb{\wt\eta_E \col E\in [E_{k+1}, E_k]}
  \;\leq\; \frac{CM^{-1+4\gamma}}{\big[2^k (\delta_++ M^{-1/5})\big]^{7/2}}.
\end{equation*}
A similar bound holds for $E_k' \deq -2+2^k(\delta_++M^{-1/5})$. For any $E\in [-10,0]$, we express $\fn_N(E) - n(E)$ as a telescopic sum
and use \eqref{main estimate on n - nsc} to get
\begin{align}\label{fnn}
  |\fn_N(E) - n(E)| &\;\leq \;\; |\fn_N(-10) - n(-10)| + 
\sum_{k \geq 0} \absB{\pb{{\fra n}_N(E_{k+1}) -  {\fra n}_N(E_k)} - \pb{n(E_{k+1}) - n(E_k)}} \nonumber \\
  &\qquad +\sum_{k \geq 0} \absB{\pb{{\fra n}_N(E_{k+1}') -  {\fra n}_N(E_k')} - \pb{n(E_{k+1}') - n(E_k')}} \nonumber \\
  &\;\prec\; M^{-1+4\gamma} (\delta_++M^{-1/5})^{-7/2}. 
\end{align}
Here we used that $n(-10)=0$ and $\fn_N(-10)\leq \fn_N(-3) \prec 0$ by \eqref{conc}.
In fact, \eqref{fnn} easily extends to any $E<-10$. By an
 analogous dyadic analysis near the upper spectral edge, we also get \eqref{conc} for any $E\geq 0$. 
Since this holds for any $\gamma>0$, we thus proved
\begin{equation} \label{nN - n fixed E}
\abs{\fn_N(E) - n(E)} \;\prec\; Y\,
\end{equation}
for any fixed $E \in [-10,10]$.

To prove the statement uniformly in $E$, we 
define the \emph{classical location of the $\alpha$-th eigenvalue} $\gamma_\alpha$ through
\begin{equation}\label{def:gam}
\int_{-\infty}^{\gamma_\alpha} \varrho(x) \, \dd x \;=\; \frac{\alpha}{N}\,.
\end{equation}
Applying \eqref{nN - n fixed E} for the $N$ energies $E=\gamma_1, \dots, \gamma_N$, we get
\be
\absbb{\fn_N(\gamma_\al) - \frac{\al}{N}} \;\prec\; Y
\label{fal}
\end{equation}
uniformly in $\alpha = 1, \dots, N$. Since $\fn_N(E)$ and $n(E)$ are nondecreasing and $Y \geq 1/N$, we find
\begin{equation*}
\sup \hb{\fn_N(E) - n(E) \col \gamma_{\al-1} \leq E \leq \gamma_{\al}} \;\leq\; \fn_N(\gamma_\al) - n(\gamma_{\al - 1}) \;=\; \fn_N(\gamma_\al) - n(\gamma_{\al}) + \frac{1}{N} \;=\; O_\prec(Y)
\end{equation*}
uniformly in $\al=2,3,\dots$. Below $\gamma_1$ we use \eqref{fal} to get
$$
   \sup_{E\leq \gamma_1} \pb{\fn_N(E) - n(E)} \;\leq\; \fn_N(\gamma_1) \;=\; O_\prec(Y)\,.
$$
Finally, for any $E\ge\gamma_N$, we have $\fn_N(E)- n(E) =\fn_N(E)-1 \leq 0$
deterministically. Thus we have proved
\begin{equation*}
  \sup_{E\in \R} \pb{\fn_N(E) - n(E)} \;=\; O_\prec(Y)\,.
\end{equation*}
A similar argument yields $\inf_{E \in \R} \pb{\fn_N(E) - n(E)} = O_\prec(Y)$. This concludes the proof of
Theorem~\ref{thm:count}.
\end{proof}

Next, we derive rigidity bounds on the locations of the eigenvalues. 
Recall the definition 
of $\gamma_\al$ from \eqref{def:gam}.

\begin{theorem}[Eigenvalue locations] \label{thm:evl} 
Suppose that $\delta_-\geq c$ (so that \eqref{Gaim} holds) and
that \eqref{extreme2} and \eqref{fn-n} hold with some positive control parameters $X,Y \leq C$.
Define  $\wh \alpha \deq \min \h{\alpha, N + 1 - \alpha}$ and let $\e>0$ be arbitrary.
Then
\begin{equation}\label{evl}
\abs{\lambda_\alpha - \gamma_\alpha} \;\prec\; Y \pbb{\frac{N}{\wh\al}}^{1/3} \qquad \text{for} \quad \wh\al\geq M^\e NY\,,
\end{equation}
and
\begin{equation}\label{evl2}
\abs{\lambda_\alpha - \gamma_\alpha} \;\prec\; X+(M^\e Y)^{2/3} \qquad \text{for} \quad \wh\al\leq  M^\e NY\,.
\end{equation}
\end{theorem}

\begin{proof}
To simplify notation, we assume that $\al\leq N/2$ so that $\wh \alpha = \alpha$; the other eigenvalues are handled analogously. Without loss of generality we assume that $\lambda_{N/2} \leq 1$. Indeed, the condition $\lambda_{N/2} \leq 1$ is equivalent to $ {\fra n}(1) \geq 1/2$, which holds with very high probability by Theorem~\ref{thm:count} and the fact that $n_{sc}(1) > 1/2$.

The key relation is
\begin{equation}\label{key} 
\frac{\alpha}{N} \;=\; n(\gamma_\alpha) \;=\;  {\fra n}_N (\lambda_\alpha) \;=\; n(\lambda_\alpha)  
+O_\prec(Y),
\end{equation}
where in the last step we used  Theorem~\ref{thm:count}.
By definition of $n(x)$ we have for $-2 \leq x \leq 1$ that
\begin{equation}\label{nnx}
n(x) \;\asymp\; (2+x)^{3/2} \;\asymp\; \kappa_x^{3/2}\,, \qquad n'(x) \;\asymp\; n(x)^{1/3}\,.
\end{equation}
Hence for $\al\leq N/2$ we have
\begin{equation}\label{nprop}
\gamma_\al +2 \;\asymp\; \pbb{\frac{\al}{N}}^{2/3}\,, \qquad n(\gamma_\al) = \frac{\al}{N}\,, \qquad
  n'(\gamma_\al) \;\asymp\; \pbb{\frac{\al}{N}}^{1/3}\,.
\end{equation}

Suppose first that $\al\ge\al_0 \deq M^\e NY$. Then $n(\gamma_\al)\geq M^\e Y$,
so that the relation \eqref{key} implies
$$
      \absb{n(\gamma_\alpha) -  n(\lambda_\alpha)} \;\prec\;  Y \leq M^{-\e} n(\gamma_\alpha) \,,
$$
which yields $n(\gamma_\alpha) \asymp n(\lambda_\alpha)$. By \eqref{nnx}, we we therefore get that $n'(\gamma_\alpha) \asymp n'(\lambda_\alpha)$ as well. Since $n'$ is nondecreasing, we get $n'(x) \asymp n'(\gamma_\alpha) \asymp n'(\lambda_\alpha)$ for any $x$ between $\gamma_\al$ and $\lambda_\al$.
Therefore, by the mean value theorem, we have
$$
  \abs{\gamma_\al-\lambda_\al} \;\leq\;  \frac{C \abs{n(\gamma_\alpha) - n(\lambda_\alpha)} }{n'(\gamma_\al)} 
 \;\prec\; Y \pbb{\frac{N}{\al}}^{1/3}\,,
$$
where in the last step we used \eqref{key} and \eqref{nprop}.
This proves \eqref{evl} for $\al\geq M^\e NY$.

For the remaining indices, $\al< \al_0$, we get from \eqref{key} the upper bound
$$
  2+\lambda_\al \;\leq\; 2+\lambda_{\al_0} \;=\; 2+\gamma_{\al_0} + O_\prec(Y^{2/3}) \;\prec\; (M^{\e}Y)^{2/3}\,,
$$
where in the second step we used \eqref{evl} and in the last step \eqref{nprop}.
In order to obtain a lower bound, we use Theorem~\ref{thm:extr} to get
$$
 -(2+\lambda_\al) \;\leq\; -(2+\lambda_1) \;\prec\; X\,.
$$
Similar bounds hold for $\gamma_\al$ as well:
$$
 0 \;\leq\;   2+ \gamma_\al \;\leq\;  2+\gamma_{\al_0} \;\leq\;  (M^{\e}Y)^{2/3}\,.
$$
Combining these bounds, we obtain
$$
    \abs{\lambda_\al-\gamma_\al} \;\prec\; X +  (M^{\e}Y)^{2/3}\,.
$$
This concludes the proof.
\end{proof}

Finally, we state a trivial corollary of  Theorem~\ref{thm:evl}.
\begin{corollary}\label{cor:evl}
Suppose that $\delta_-\geq c$ and
that \eqref{extreme2} and \eqref{fn-n} hold with some positive control parameters $X,Y \leq C$. 
Then
$$
  \sum_{\al=1}^N \abs{\lambda_\al-\gamma_\al}^2 \;\prec\;  NY(Y+ X^2)\,.
$$
\end{corollary}

\subsection{Local density of states: proof of Lemma~\ref{lemma: counting function estimate}}\label{sec:count}

In this section we prove Lemma \ref{lemma: counting function estimate}. Define the empirical eigenvalue distribution
\begin{equation*}
\varrho_N(x) \;=\; \frac{1}{N} \sum_{\alpha = 1}^N \delta(x - \lambda_\alpha)\,,
\end{equation*}
so that we may write
\begin{equation*}
\qquad {\fra n}_N(E) \;=\; \frac{1}{N} \, 
\abs{\h{\alpha \col \lambda_\alpha \leq E}} \;=\; 
\int_{-\infty}^E \varrho_N(x) \, \dd x\,, \qquad
m_N(z) \;=\; \frac{1}{N} \tr G(z) \;=\; \int \frac{\varrho_N(x) \, \dd x}{x - z}\,.
\end{equation*}
We introduce the differences
\begin{equation*}
\varrho^\Delta \;\deq\; \varrho_N - \varrho \,, \qquad m^\Delta \;\deq\; m_N - m\,.
\end{equation*}

Following \cite{ERSY}, we use the Helffer-Sj\"ostrand functional calculus \cite{Davies, HS}.
Introduce $\cal E \;\deq\; \max\hb{E_2 - E_1, \wt\eta}\,.$
Let $\chi$ be a smooth cutoff function equal to $1$ on $[-\cal E, \cal E]$ and vanishing 
on $[-2 \cal E, 2 \cal E]^c$, such that $\abs{\chi'(y)} \leq C \cal E^{-1}$. 
Let $f$ be a 
characteristic function of the interval $[E_1, E_2]$ smoothed on the scale
 $\wt\eta$: $f(x) = 1$ on $[E_1+\wt\eta, E_2-\wt\eta]$, $f(x) = 
0$ on $[E_1, E_2 ]^c$, $\abs{f'(x)} \leq C \wt\eta^{-1}$, and $\abs{f''(x)} \leq C \wt\eta^{-2}$. Note that the 
supports of $f'$ and $f''$ have measure $O(\wt\eta)$.

Then we have the estimate (see Equation (B.13) in \cite{ERSY})
\begin{multline} \label{HS split}
\absbb{\int f(\lambda) \, \varrho^\Delta(\lambda) \, \dd \lambda} \;\leq\; C \absbb{\int \dd x \int_0^\infty \dd y \, 
(f(x) + y f'(x)) \, \chi'(y) \, m^\Delta(x + \ii y)}
\\
+
C \absbb{\int \dd x \int_0^{\wt\eta} \dd y \, f''(x) \chi(y) \, y \im m^\Delta(x + \ii y)}
+
C \absbb{\int \dd x \int_{\wt\eta}^\infty \dd y \, f''(x) \chi(y) \, y \im m^\Delta(x + \ii y)}\,.
\end{multline}
Since $\chi'$ vanishes away from $[\cal E, 2 \cal E]$ and $f$ vanishes away from $[E_1, E_2]$, we may apply \eqref{lsc hypothesis} to get
\begin{equation}\label{largeeta}
\abs{m_N(x+iy)-m(x+iy)} \;\prec\; \frac{1}{My}
\end{equation}
uniformly for $x\in [E_1, E_2]$ and $y\geq \wt\eta$.
Thus the first term on the right-hand side of \eqref{HS split} is 
bounded by
\begin{equation}\label{1st}
\frac{C}{M \cal E}\int \dd x \int_{\cal E}^{2 \cal E} \dd y \, \abs{f(x) + y f'(x)} 
\;\prec\; \frac{1}{M}\,.
\end{equation}
In order to estimate the two remaining terms of \eqref{HS split}, we estimate $\im m^\Delta(x + \ii y)$.
If $y \geq \wt \eta$ we may use \eqref{largeeta}.  Consider therefore the case $0 < y \leq \wt \eta$.  From Lemma \ref{lemma: 
msc} we find
\begin{equation} \label{im msc bound}
\abs{\im m(x + \ii y)} \;\leq\; C \sqrt{\kappa_x + y}\,.
\end{equation}
By spectral decomposition of $H$, it is easy to see that the function
 $y \mapsto y \im m_N(x + \ii y)$ is monotone 
increasing.  Thus we get, using \eqref{im msc bound}, $x + \ii \wt \eta \in \wt{\f S}$,
 and \eqref{lsc hypothesis}, that
\begin{equation} \label{im m for small y 0}
y \im m_N(x + \ii y) \;\leq\; \wt \eta \im m_N(x + \ii \wt \eta) \;\prec\;
 \wt \eta 
\pbb{\sqrt{\kappa_x + \wt \eta} + \frac{1}{M\wt \eta}} 
\;\prec\;  \wt \eta  \sqrt{\kappa_x + \wt \eta}+ \frac{1}{M}\,,
\end{equation}
for $y \leq \wt \eta$ and $x\in [E_1,E_2]$.
Using $m^\Delta = m_N - m$ and recalling \eqref{im msc bound}, we therefore get
\begin{equation} \label{im m for small y}
\abs{y \im m^\Delta(x + \ii y)} \;\prec\; \wt \eta \sqrt{\kappa_x + \wt \eta}+ \frac{1}{M}\,,
\end{equation}
for $y \leq \wt \eta$ and $x\in [E_1,E_2]$.
The second term of \eqref{HS split} is therefore bounded by
\begin{equation*}
\pbb{\wt \eta \sqrt{\kappa_x + \wt \eta}+ \frac{1}{M}}
 \int \dd x \, \abs{f''(x)} \int_0^{\wt\eta} \dd y \, \chi(y) \;\leq\;
 \wt \eta \sqrt{\kappa_x + \wt \eta}+ \frac{1}{M}\,.
\end{equation*}

In order to estimate the third term on the right-hand side of \eqref{HS split}, we integrate by parts, first in $x$ and 
then in $y$, to obtain the bound
\begin{multline} \label{HS after IBP}
C \absbb{\int \dd x \, f'(x) \, \wt\eta \re m^\Delta(x + \ii \wt\eta)}
+
C \absbb{\int \dd x \int_{\wt\eta}^\infty \dd y \, f'(x) \chi'(y) y
 \re m^\Delta(x + \ii y)}
\\
+
C \absbb{\int \dd x \int_{\wt\eta}^\infty \dd y \, f'(x) \chi(y) 
\re m^\Delta(x + \ii y)}\,.
\end{multline}
The second term of \eqref{HS after IBP} is similar to the first term on the right-hand side of \eqref{HS split}, and is 
easily seen to be bounded by $1/M$ as in \eqref{1st}.

In order to bound the first and third terms of \eqref{HS after IBP}, we estimate, for any $y \leq \wt \eta$,
\begin{equation}\label{derint}
\absb{m^\Delta(x + \ii y)} \;\leq\; \absb{m^\Delta(x + \ii \wt \eta)} + \int_y^{\wt \eta} \dd u \, \pB{\absb{\partial_u  m_N(x + \ii u)} + \absb{\partial_u m(x + \ii u)}}\,.
\end{equation}
Moreover, using the monotonicity  of  $y \mapsto y \im m_N(x + \ii y)$
 and the identity $\sum_j |G_{ij}|^2 =\eta^{-1}
\im G_{ii}$ , we find for any $u \leq \wt \eta$ that
\begin{equation*}
\absb{\partial_u m_N(x + \ii u)} \;=\; \absbb{\frac{1}{N} \tr  G^2(x + \ii u)} \;\leq\;
 \frac{1}{N} \sum_{i,j} \absb{G_{ij}(x + \ii u)}^2 \;=\; \frac{1}{u} \im  m_N(x + \ii u) \;\leq\; \frac{1}{u^2} \, \wt \eta \im  m_N(x + \ii \wt \eta)\,.
\end{equation*}
Similarly, we find from \eqref{definition of msc} that
\begin{equation*}
\absb{\partial_u m(x + \ii u)} \;\leq\; \frac{1}{u^2} \, \wt \eta \im m(x + \ii \wt \eta) \leq \frac{C\wt \eta}{u^2} \qquad (u\leq \wt\eta).
\end{equation*}
Thus \eqref{derint} and \eqref{largeeta} yield
\begin{equation} \label{bound for full m delta}
\absb{m^\Delta(x + \ii y)} \;\prec\; \frac{1}{M\wt\eta}
+ \int_y^{\wt \eta} \dd u \, \frac{\wt \eta}{u^2}
\Big( 1+ \frac{1}{M\wt\eta}\Big)
 \;\prec\;  \frac{\wt \eta}{y} \qquad (y \leq \wt \eta)\,,
\end{equation}
where we also used that $\wt\eta\geq M^{-1}$.
Using \eqref{bound for full m delta} for $y=\wt\eta$,
 we may now estimate the first term of \eqref{HS after IBP} by $\wt\eta$.

What remains is the third term of \eqref{HS after IBP},
which can be estimated, using \eqref{largeeta}, by
\begin{equation*}
 \int \dd x \int_{\wt \eta}^{2 \cal E} \dd y \, \abs{f'(x)} \frac{1}{M y} 
  \;\leq\; C M^{-1}(1+\abs{\log \wt\eta}) \;\leq\; CM^{-1}\log M\,.
\end{equation*}
Summarizing, we have proved that
\begin{equation} \label{smoothed counting function}
\absbb{\int f(\lambda) \, \varrho^\Delta(\lambda) \, \dd \lambda} 
\;\prec\;  \frac{1}{M}+\wt \eta \sqrt{\kappa_x + \wt \eta}+ \wt\eta+ 
\frac{\log M}{M}
 \;\prec\; \wt\eta + \frac{1}{M}.
\end{equation}

Since $ \im m_N(x + \ii \wt\eta)$ controls the local density
on scale $\wt\eta$, we may estimate $\abs{{\fra n}_N(E) - n(E)}$ using \eqref{im m for small y 0} according to
\begin{equation*}
\abs{{\fra n}_N(x + \wt\eta) - {\fra n}_N(x - \wt\eta)}
 \;\leq\; C\wt \eta \im  m_N(x + \ii \wt\eta)
\prec   \wt \eta  \sqrt{\kappa_x + \wt \eta}+ \frac{1}{M}.
\end{equation*}
Thus we get
\begin{equation*}
\absbb{ {\fra n}_N(E_1) -  {\fra n}_N(E_2) - \int f(\lambda) \, \varrho_N(\lambda) \, \dd \lambda} \;\leq\; C \sum_{i = 
1,2} \pb{ {\fra n}(E_i + \wt \eta) -  {\fra n}(E_i - \wt \eta)}
 \;\prec\;   \wt \eta  \sqrt{\kappa_x + \wt \eta}+ \frac{1}{M}\,.
\end{equation*}
Similarly, since $\varrho$ has a bounded density, we find
\begin{equation*}
\absbb{n(E_1) - n(E_2) - \int f(\lambda) \, \varrho(\lambda) \, \dd \lambda} \;\leq\; C \wt\eta\,.
\end{equation*}
Together with \eqref{smoothed counting function} and recalling $\wt\eta\geq M^{-1}$, we therefore get 
\eqref{main estimate on n - nsc}. This concludes the proof of Lemma \ref{lemma: counting function estimate}.

\section{Bulk universality}\label{sec:bulk}

Local eigenvalue statistics are described by correlation functions on the scale $1/N$.
Fix an integer $n \ge 2$ and an energy $E \in (-2,2)$. Abbreviating $\f x = (x_1, x_2,
\ldots x_n)$, we define the \emph{local correlation function}
\be
   f_N^{(n)}(E, \f x) \;\deq\; \frac{1}{\varrho(E)^n} p_N^{(n)} \pbb{ E+ \frac{x_1}{N\varrho(E)},
  E+ \frac{x_2}{N\varrho(E)}, \ldots ,  E+ \frac{x_n}{N\varrho(E)}}\,,
\label{loccor}
\ee
where $p_N^{(n)}$ is the $n$-point correlation function of the $N$ eigenvalues and $\varrho(E)$ is the density of the semicircle law defined in \eqref{definition of msc}. Universality of the local eigenvalue statistics means that, for any fixed $n$, the limit as $N \to \infty$ of the local correlation function $f_N^{(n)}$ only depends on the symmetry class of the matrix entries, and is otherwise independent of their distribution. In particular, the limit of $f^{(n)}_N$ coincides with that of a GOE or GUE matrix, which is explicitly known. In this paper, we consider
local correlation functions averaged over a small energy interval of size $\ell=N^{-\e}$,
\be\label{ftilde}
   \wt f_N^{(n)}(E, {\bf x})  \;\deq\; \frac{1}{2\ell} \int_{E-\ell}^{E+\ell}  f_N^{(n)}(E', {\bf x}) \, \dd E'\,.
\ee
Universality is understood in the sense of the weak limit, as $N\to\infty$ for fixed $|E|<2$, of $\wt f_N^{(n)}(E, {\bf x})$ in the variables ${\bf x}$.

The general approach developed in \cite{ESY4, ESYY, EYY} 
to prove the universality of the local eigenvalue statistics in
the bulk spectrum of a general Wigner-type matrix consists of three steps.
\begin{itemize}
\item[(i)] A rigidity estimate on the locations of the eigenvalues, in the sense of a quadratic mean.

\item[(ii)] The spectral universality for matrices with a small Gaussian component,
via local ergodicity of the Dyson Brownian motion (DBM).

\item[(iii)] A perturbation argument that removes the small Gaussian component by
comparing Green functions.
\end{itemize}

In this paper we do not give the details of steps (ii) and (iii), since they have been concisely presented elsewhere, e.g.\ in \cite{EYBull}. Here we only summarize the results and the key arguments of steps (ii) and (iii) for the general class of matrices we consider. 
In this section we assume that $H$ is either real symmetric
or complex Hermitian. The former case means that the entries of $H$ are real. The latter means, loosely, that its off-diagonal entries have a nontrivial imaginary part. More precisely, in the complex Hermitian case we shall replace the lower bound on the variances $s_{ij}$ from Definition \ref{def: Wigner} with the following, stronger, condition.
\begin{definition} \label{def: complex full}
We call the Hermitian matrix $H$ a \emph{complex $a$-full Wigner matrix} 
if for each $i,j$ the $2\times 2$ covariance matrix 
$$
\sigma_{ij} \;=\; \begin{pmatrix} \E (\re h_{ij})^2 & \E (\re h_{ij})(\im h_{ij}) \\
  \E (\re h_{ij})(\im h_{ij}) & \E ( \im h_{ij})^2 
\end{pmatrix}
$$
satisfies
$$
     \sigma \;\geq\; \frac{a}{N}
$$
as a symmetric matrix. Note that this condition implies that $H$ is $a$-full, but the converse is not true.
\end{definition}

\newcommand{\bla}{\mbox{\boldmath $\lambda$}}
\newcommand{\cH}{{\cal H}}

We consider a stochastic flow of Wigner-type matrices generated by the Ornstein-Uhlenbeck equation
$$
   \dd H_t \;=\; \frac{1}{\sqrt{N}} \, \dd B_t -\frac{1}{2} H_t\dd t
$$
with some given initial matrix $H_0$.
Here $B$ is an $N \times N$ matrix-valued standard Brownian motion with 
the same symmetry type as $H$. The resulting dynamics on the level of the eigenvalues is Dyson Brownian motion (DBM).
It is well known that $H_t$ has the same distribution as the matrix
\begin{equation} \label{distribution of Ht}
\me^{-t/2}H_0+ (1-\me^{-t})^{1/2}U\,,
\end{equation}
where $U$ is an independent standard Gaussian Wigner matrix of the same symmetry class as $H$.
In particular, $H_t$ converges to $U$ as $t\to\infty$.
The eigenvalue distribution converges to the Gaussian equilibrium measure, whose density is explicitly given by
$$
\mu(\bla) \;=\; \frac{1}{Z} \, \me^{-\beta N \cH(\bla)}\dd\bla\,, \qquad \cH(\bla) \;\deq\; \sum_{i=1}^N \frac{\lambda_i^2}{4}
  - \frac{1}{N}\sum_{i<j} \log \abs{\lambda_i-\lambda_j}\,;
$$
here $\beta=1$ for the real symmetric case (GOE) and $\beta=2$ for the complex Hermitian case (GUE).

The matrix $S^{(t)}$ of variances of $H_t$ is given by
$$
   S^{(t)} \;=\; \me^{-t} S^{(0)} + (1-\me^{-t}){\f e}{\f e}^*,
$$
where $S^{(0)}$ is the matrix of variances of $H_0$. It is easy to see that the gaps $\delta_\pm(t)$ of $S^{(t)}$ satisfy
$\delta_\pm(t) \geq \delta_\pm(0)$; therefore the corresponding parameters \eqref{def:rhohat} satisfy $\wt\Gamma_t(z)\leq \wt \Gamma_0(z)$. Since all estimates behind our main theorems in Sections~\ref{sec: setup} and \ref{sec: dos} improve
if $\delta_\pm$ increase, it is immediate that all results in these sections
hold for $H_t$ provided they hold for $H_0$.

The key quantity to be controlled when establishing bulk universality is the mean quadratic distance
of the eigenvalues from their classical locations,
\begin{equation}\label{Qdef}
  Q \;\deq\; \max_{t \geq 0} \, \E^{(t)} \frac{1}{N}\sum_i (\lambda_i-\gamma_i)^2\,,
\end{equation}
where $\E^{(t)}$ denotes the expectation with respect to the ensemble $H_t$.
By Corollary~\ref{cor:evl} we have
$$
    Q \;\leq\; N^\e Y(Y+X^2)
$$
for any $\e>0$ and $N\geq N_0(\e)$. Here we used that the estimate from Corollary~\ref{cor:evl} is uniform in $t$, by the remark in the previous paragraph. 

We modify the original DBM by adding a local relaxation term of the form $\frac{1}{2\tau}\sum_i (\lambda_i-\gamma_i)^2$
to the original Hamiltonian $\cH$, which has the effect of artificially speeding up the relaxation of the dynamics. Here $\tau \ll 1$ is a small parameter, the relaxation time of the modified dynamics. We choose $\tau \deq N^{1+4\e}Q$ for some $\e>0$.  As Theorem 4.1 of \cite{ESYY} 
(see also Theorem 2.2 of \cite{EYBull}) shows, the
local statistics of the eigenvalue gaps of $H_t$ and GUE/GOE coincide
if  $t\geq N^\e\tau= N^{1+4\e} Q$,  i.e.\ if
\begin{equation}\label{lowt}
   t \;\geq\;  N^{1+5\e} Y(Y+X^2)\,.
\end{equation}
The local statistics are averaged over $N^{1-\e}$ consecutive eigenvalues or, alternatively,
in the energy parameter $E$ over an interval of length $N^{-\e}$.

To complete the programme (i)--(iii), we need to compare the local statistics of the original 
ensemble $H$ and $H_t$, i.e.\ perform step (iii). We first recall the Green function comparison
theorem from \cite{EYY} for the case $M \asymp N$ (generalized Wigner). The result states, roughly, that expectations
of Green functions with spectral parameter $z$ satisfying $\im z\ge N^{-1-\e}$ are 
determined by the first four moments of the single-entry distributions. Therefore the local
eigenvalue statistics on a very small scale, $\eta=N^{-1-\e}$, of two Wigner ensembles are indistinguishable if the first four moments of their matrix entries match.
More precisely, for the local $n$-point correlation functions \eqref{loccor}
to match, one needs to compare
expectations of $n$-th order monomials of the form 
\be \label{mmm}
\prod_{k=1}^n m_N(E_k + i \eta)\,,
\ee
where the energies $E_k$ are chosen in the bulk spectrum
with $E_k-E_{k'} = O(1/N)$. (Recall that $m_N(z) = \frac{1}{N}\tr G(z)$.)

The proof uses a Lindeberg-type replacement strategy to change the distribution of
each matrix entry $h_{ij}$ one by one in a telescopic sum.  
The idea of applying Lindeberg's method in random matrices 
was recently used  by Chatterjee \cite{Ch}  for comparing  the traces of the Green functions; 
the idea was also used by Tao and Vu \cite{TV}
 in the context of comparing individual eigenvalue 
distributions.  
 The error resulting from each replacement is estimated using
a fourth order resolvent expansion, where all resolvents $G(z)=(H-z)^{-1}$ 
with $z=E_k +i\eta$ appearing in \eqref{mmm} are expanded
with respect to the single matrix entry $h_{ij}$ (and its conjugate $h_{ji}= \bar h_{ij}$).
If the first four moments of the two distributions match, then the terms of at most fourth order
in this expansion remain unchanged by each replacement.
The error term is of order $\E \abs{h_{ij}}^5 \asymp N^{-5/2}$, which is negligible even
after summing up all $N^2$ pairs of indices $(i,j)$. This estimate assumes
that the resolvent entries in the expansion (and hence all factors $m_N(z)$ in \eqref{mmm}) are essentially bounded.

The Green function comparison method therefore has two main ingredients. First, a high probability apriori estimate is needed on the
resolvent entries at any spectral parameter $z$ with imaginary part $\eta$ slightly below $1/N$:
\be\label{GFTGij}
   \max_{i,j} \abs{G_{ij}(E+i\eta)} \;\prec\; N^{2\e} \qquad  (\eta \geq N^{-1-\e})
\ee
for any small $\e>0$.  Clearly, the same estimate also holds for $m_N(E+i\eta)$.
 The bound \eqref{GFTGij} is typically obtained from the local semicircle law
for the resolvent entries, \eqref{Gijest 2}. Although the local semicircle law is effective only
for $\im z \gg 1/N$, it still gives an almost optimal bound for a somewhat smaller $\eta$ by using
the trivial estimate 
\be\label{trivgg}
   \max_{i,j} \abs{G_{ij}(E+i\eta)} \;\leq\; \log N \,  \pbb{\frac{\eta'}{\eta}}
\sup_{\eta''\ge \eta'}\max_i \im G_{ii}(E+i\eta'') \qquad (\eta \leq \eta')
\ee
with the choice of $\eta' = N^{-1+\e}$. The proof of \eqref{trivgg} follows
from a simple dyadic decomposition; see the proof of Theorem 2.3 in Section 8
of \cite{EYY} for details.

The second ingredient is  the construction of an initial ensemble $H_0$ whose time evolution $H_t$ for some $t\le 1$ satisfying \eqref{lowt} is close to $H$; here closeness is measured by the \emph{matching of moments} of the matrix entries between the ensembles $H$ and $H_t$. We shall choose $H_0$, with variance matrix $S^{(0)}$, 
so that the second moments of $H$ and $H_t$  match,
\begin{equation}\label{2match}
   S \;=\; \me^{-t} S^{(0)} +  (1-\me^{-t}){\f e} {\f e}^*\,,
\end{equation}
and the third and fourth moments are close.   We remark that the  matching of higher moments was introduced  
in the work of \cite{TV}, while the idea of approximating a general matrix ensemble
 by an appropriate Gussian 
one  appeared earlier in \cite{EPRSY}.  
They have to be so close that
even after multiplication with at most five resolvent entries and summing up for all $i,j$ indices,
their difference is still small. (Five resolvent entries appear in the fourth order
of the resolvent expansion of $G$.) Thus, given \eqref{GFTGij}, we require that
\begin{equation}\label{s34}
\max_{i,j} \absb{\E h_{ij}^s -\E^{(t)} h_{ij}^s}
   \;\leq\; N^{-2 -  (2n+9)  \e} \qquad (s=3,4)\,
\end{equation}
 to ensure that the expectations of the $n$-fold product
in \eqref{mmm} are close.  
This formulation holds for the real symmetric case; in the complex Hermitian case 
all moments of order $s=3,4$ involving the real and imaginary parts of $h_{ij}$ have to
be approximated. To simplify notation, we work with the real symmetric case in the sequel.

The matching can be done in two steps. In the first we construct a
matrix of variances $S^{(0)}$ such that \eqref{2match} holds. 
This first step is possible if, given $S$ associated with $H$, \eqref{2match} can be satisfied for a doubly stochastic $S^{(0)}$, i.e.\ if
$H$ is an $a$-full Wigner matrix and 
\begin{equation}\label{upt}
      a \;\geq\; Ct
\end{equation}
with some large constant $C$. For the complex Hermitian case, the condition \eqref{upt}
is the same but $H$ has to be complex $a$-full Wigner matrix (see Definition \ref{def: complex full}).

In the second step of moment matching, we use Lemma 3.4 of \cite{EYY2} to construct an ensemble $H_0$ with variances $S^{(0)}$, such that the entries of $H$ and $H_t$ satisfy
\begin{equation*}
\E h_{ij} \;=\; \E^{(t)} h_{ij} \;=\; 0\,, \qquad \E h_{ij}^2 \;=\; \E^{(t)} h_{ij}^2 \;=\; s_{ij}\,, \qquad \E h_{ij}^3 \;=\; \E^{(t)} h_{ij}^3  \,, \qquad \absb{\E h_{ij}^4 -\E^{(t)} h_{ij}^4} \;\leq\; C t s_{ij}^2\,.
\end{equation*}
This means that \eqref{s34} holds if 
$$
    C ts_{ij}^2 \;\leq\; N^{-2-   (2n+9)  \e}\,.
$$
Suppose that $H$ is $b$-flat, i.e.\ that $s_{ij}\leq b/N$. Then this condition holds provided
\begin{equation}
   C  tb^2 \;\leq\; N^{-  (2n+9)  \e}\,.
\label{upt3}
\end{equation}
The argument so far assumed that $M \asymp N$ ($H$ is a generalized Wigner matrix), in which case $G_{ij}(E+i\eta')$ remains
essentially bounded down to the scale $\eta' \approx 1/N$. If $M\ll N$, then \eqref{Gijest 2}
provides control only  down to scale $\eta'\gg 1/M$ and \eqref{trivgg} gives only the weaker bound
\be\label{gijm}
     \abs{G_{ij}(E+\ii \eta)} \;\prec\; \frac{1}{M\eta}\,,
\ee
for  any $\eta \le 1/M$,
which replaces \eqref{GFTGij}. 
Using this weaker bound, the condition \eqref{upt3} is replaced with
\begin{equation}
   C  tb^2 \;\prec\;  (M\eta)^{ n+4}\,,
\label{upt2}
\end{equation}
which is needed for $n$-fold products of the form \eqref{mmm}  to be close.
(For convenience, here we use the notation $A_N\prec B_N$ even
for deterministic quantities to indicate that $A_N\le N^{\e} B_N$
for any $\e>0$ and $N\ge N_0(\e)$.) 
The bound \eqref{upt2} thus guarantees that,
for any fixed $n$, the expectations of the $n$-fold products
of the form \eqref{mmm} with respect to the ensembles $H$
and $H_t$ are close. Following the argument in 
the proof of Theorem 6.4 of \cite{EYY},
this means that for any smooth,
compactly supported function $O \col \R^n\to \R$, the expectations 
of observables
\be
   \sum_{i_1\ne i_2\ne \ldots \ne i_n} O_\eta \pB{ N(\lambda_{i_1}-E), N(\lambda_{i_2}-E) ,\dots,
N(\lambda_{i_n}-E)}
\label{Oeta}
\ee
are close, where the smeared out observable $O_\eta$ on scale $\eta$ is defined through
$$
   O_\eta (\beta_1, \ldots, \beta_n) \;\deq\;  \frac{1}{(\pi N)^n}
  \int_{\R^n} \dd \alpha_1 \cdots \dd \alpha_n \, O(\al_1, \ldots, \al_n) \prod_{j=1}^n \theta_\eta \pbb{\frac{\beta_j-\al_j}{N} }\,,
  \qquad
   \theta_\eta(x) \;\deq\; \im \frac{1}{x-i\eta}\,.
$$

To conclude the result for observables with $O$ instead of $O_\eta$ in \eqref{Oeta},
we need to estimate, for both ensembles, the difference
\be
\E \sum_{i_1\ne i_2\ne \ldots \ne i_n} (O-O_\eta)
\pB{N(\lambda_{i_1}-E), N(\lambda_{i_2}-E), \dots,
N(\lambda_{i_n}-E)}\,.
\label{OO}
\ee
Due to the smoothness of $O$,
we can decompose
$O-O_\eta = Q_1+ Q_2$,
where
$$
|Q_1(\beta_1, \ldots, \beta_n)| \;\le\; CN\eta \prod_{j=1}^n{\bf 1}( |\beta_j|\le K)
$$
and
$$
|Q_2(\beta_1, \ldots, \beta_n)| \;\le\; C\sum_{j=1}^n {\bf 1}( |\beta_j|\ge K)
  \prod_{j=1}^n \frac{1}{1+\beta_j^2}\,,
$$
with an arbitratry parameter $K\gg N/M$. Here the constants depend on $O$.
The contribution from $Q_1$ to \eqref{OO} can thus
be estimated by
$$
   \E  \sum_{i_1\ne i_2\ne \ldots \ne i_n} Q_1\Big( \ldots \Big) \;\prec\; CN\eta K^n\,,
$$
where we used that the expected 
number of eigenvalues in the interval $[E- K/N, E+ K/N]$
is $O_\prec(K)$, since \eqref{gijm}  guarantees that
the density is bounded on scales larger than $1/M$.
The contribution from $Q_2$ to \eqref{OO} is estimated by
\be\label{Q1}
   \E  \sum_{i_1\ne i_2\ne \ldots \ne i_n} Q_2\Big( \ldots \Big) \;\prec\; CK^{-1} \pbb{\frac{N}{M}}^n\,.
\ee
In the last step we used \eqref{gijm} to estimate
\be\label{Q2}
  \sum_{k=1}^N \frac{1}{1+ N^2(\lambda_k-E)^2} \;=\; \frac{1}{N} \im \tr G \pbb{E+ \frac{\ii}{N}}
 \;\prec\; \frac{N}{M}\,.
\ee
Optimizing the choice of $K$ and $\eta$,  \eqref{upt2} becomes
\be\label{upto}
 C  tb^2 \;\prec\; \pbb{\frac{M}{N}}^{(n^2+1)( n+4)}\,.
\ee
Summarizing the conditions \eqref{lowt}, \eqref{upt}, and \eqref{upto}, we require that
$$
  N^{1+5\e}Y(Y+X^2) \;\prec\; 
\min \hbb{ a, b^{-2} \pbb{\frac{M}{N}}^{ (n^2+1)(n+4)}}
$$
in order to have bulk universality. We have therefore proved the following result.
\begin{theorem} \label{thm: bulk univ}
Suppose that $H$ is $N/M$-flat and $a$-full (in the real symmetric case) or complex $a$-full 
(in the complex Hermitian case). 
Suppose moreover that \eqref{extreme2} and \eqref{fn-n} hold with some positive control parameters $X,Y \leq C$.
Fix an arbitrary positive parameter $\e>0$.
Then the   local $n$-point correlation functions  of $H$, averaged over the energy parameter in an
interval of size $N^{-\e}$  around  $|E|<2$ (see \eqref{ftilde}), 
coincide with those of GOE or GUE provided that 
\be
N^{1+6\e} Y(Y+X^2) \;\leq\; 
 \min \hbb{ a,  \pbb{\frac{M}{N}}^{ (n^2+1)(n+4)+2}  }\,.
\label{fincond}
\end{equation}
In particular, if  $N^{3/4}\leq M\leq N$ then 
 \eqref{extreme2} and \eqref{fn-n} hold with
 $X$ and $Y$ defined in \eqref{def:X} and \eqref{def:Y}.
\end{theorem}
We conclude with a few examples illustrating Theorem~\ref{thm: bulk univ}.  

\begin{corollary}\label{cor:bulk}  Fix an integer $n\ge 2$.  There exists
a positive number  $p(n) \geq c n^{-3}$ with the following property. 
Suppose that $H$ satisfies {\fontseries{bx}\selectfont any} of the following conditions
for some sufficiently small $\xi>0$.
\begin{enumerate}
\item
$cN^{-1-\xi}\leq s_{ij} \leq CN^{- 1+ p(n)  -\xi}$.
\item
$cN^{-\frac{9}{8}+\xi}\leq s_{ij} \leq CN^{-1}$.
\item
$H$ is a one-dimensional band matrix with band width $W$ with a mean-field component
of size $\nu$ (see Definition~\ref{def: bandwigner}) such that  $W\ge N^{ 1-  p(n) +\xi}$
and $\nu \ge N^{15+\xi}W^{-16}$. 
\end{enumerate}
Then there exists an $\e>0$ (depending on $\xi$ and  $n$)  such that
the  local $n$-point correlation functions  of $H$, averaged over the energy parameter in an
interval of size $N^{-\e}$  around  $|E|<2$, 
coincide with those of GOE or GUE (depending on the symmetry class of $H$).
\end{corollary}
 We remark that the conditions for the
 upper bound on $s_{ij}$ in parts (i)  and (iii) are similar.
But the band structure in (iii)
allows one to choose a much smaller mean-field component than in (i).

\begin{proof} 
In Case (i),  we have $a=cN^{-\xi}$ and $b=N/M$ in Definition \ref{def: Wigner}; hence $\delta_\pm\geq cN^{-\xi}$
 by Proposition \ref{prop: spectrum of S}. Therefore
$Y= M^{-1} N^{-7\xi/2}$ and $X=N^2M^{-8/3}$ from \eqref{def:X} and \eqref{def:Y}, so that \eqref{fincond} reads
$$
  \frac{N}{M}\Big( \frac{1}{M} + \frac{N^4}{M^{16/3}}\Big) \;\leq\; N^{- (1+6\e) }
 N^{7\xi} \min \hbb{ N^{-\xi},  \Big( \frac{M}{N}\Big)^{ (n^2+1)(n+4)+2 }}\,.
$$
By Theorem \ref{thm: bulk univ} bulk universality therefore holds provided that $M \ge N^{1- p(n)  +\xi}$
with any sufficiently small positive $\xi>0$ (and $\e$ chosen appropriately, depending on $\xi$
 and $n$).  The function $p(n)$ can be easily computed. 

We remark that if we additionally assume that $h_{ij}$ 
has a symmetric law with subexponential decay \eqref{subexp},
 then by Theorem~\ref{thm:extr1} we can use the improved control parameter $X= M^{-2/3}$.
This yields a better threshold $p(n)$. For example, for $n=2$ we obtain $p(n)=\frac{1}{34}$.

In Case (ii) we take  $M=N$, i.e.\ $b=c$ 
and  $\delta_+ \geq  a =  N^{-1/8+\xi}$. Then with the choice \eqref{def:X} and \eqref{def:Y} we have
 $Y \leq C N^{-1} \delta_+^{-7/2}$, 
$X \leq  C N^{-2/3} + C N^{-2} (\delta_++ N^{-1/7})^{-12}$, so that \eqref{fincond} reads
$$
   \delta_+^{-7/2} \pB{N^{-1} \delta_+^{-7/2} + N^{-4/3} + N^{-4} (\delta_+ + N^{-1/7})^{-24}} \;\ll\; a\,,
$$
which holds since $\delta_+ \geq a\ge N^{-1/8}$.

Finally, in Case (iii) we have $W\asymp M$,  $b=N/M$,  $a=\nu$,   $\delta_+ \geq c \nu  + c (M/N)^2$,
and $\delta_- \geq c$.  Since $M\ge N^{22/23}$ we have  $\delta_+\geq cM^{-1/5}$,
Thus, with the choice \eqref{def:X} and \eqref{def:Y}, we have
$$
  Y \asymp \frac{1}{M\delta_+^{7/2}} 
 \;\leq\; C \frac{N^{7}}{M^{8}}\,, \qquad X \;\leq\; C \frac{N^2}{M^{8/3}} + C \frac{N^{26}}{M^{28}} \;\asymp\;  \frac{N^{26}}{M^{28}}\,,
$$
and \eqref{fincond} reads
$$
   \frac{N^{8}}{M^{8}}\Big(  \frac{N^{7}}{M^{8}} + \frac{N^{52}}{M^{56}}\Big) \;\ll\; \min \hbb{\nu \,,\, 
\Big( \frac{M}{N} \Big)^{  (n^2+1)(n+4)+2 } }\,.
$$
This leads to the conditions
\begin{equation}\label{2con}
\nu \;\gg\; \frac{N^{15}}{M^{16}},  \qquad  M\gg N^{1-p(n)}\,, 
\end{equation} 
with some positive $p(n)$, 
which concludes the proof. 
\end{proof}

\appendix

\section{Behaviour of $\Gamma$ and $\wt \Gamma$} \label{app: bounds on varrho}

In this section we give basic bounds on the parameters $\Gamma$ and $\wt \Gamma$. As it turns out, their behaviour is intimately linked with the spectrum of $S$, more precisely with its spectral gaps. Recall that the spectrum of $S$ lies in $[-1,1]$, with $1$ being a simple eigenvalue.

\begin{definition} \label{def: gap}
Let $\delta_-$ be the distance from $-1$ to the spectrum of $S$, and $\delta_+$ the distance from $1$ to the spectrum of $S$ restricted to $\f e^\perp$. In other words, $\delta_\pm$ are the largest numbers satisfying
\begin{equation*}
S \;\geq\;  -1+\delta_-, \qquad  S\big|_{\f e^\perp} \;\leq\; 1-\delta_+ \,.
\end{equation*}
\end{definition}

The following proposition gives explicit bounds on $\Gamma$ and $\wt \Gamma$ depending on the spectral gaps $\delta_\pm$. We recall the notations $z = E + \ii \eta$, $\kappa \deq \absb{\abs{E} - 2}$ and the definition of $\theta$ from \eqref{def theta}.

\begin{proposition} \label{prop: bounds on Gamma}
There is a universal constant $C$ such that the following holds uniformly in the domain $\hb{ z=E+i\eta \col |E|\leq 10, \, M^{-1}\leq \eta\leq 10}$, and in particular in any spectral domain $\f D$.
\begin{enumerate}
\item
We have the estimate
\begin{equation} \label{bound for Gamma}
\frac{1}{C \sqrt{\kappa+\eta}} \;\leq\; \Gamma(z)\;\leq\; \frac{C \log N} {1- \max_\pm \absb{ \frac{1\pm m^2}{2}}} \;\leq\; \frac{C \log N}{\min\{\eta + E^2, \theta\}}\,.
\end{equation}
\item
In the presence of a gap $\delta_-$ we may improve the upper bound  to
\begin{equation} \label{bound for Gamma with gap}
\Gamma(z) \;\leq\; \frac{ C\log N}{\min \h{\delta_- + \eta + E^2, \theta}}.
\end{equation}
\item
For $\wt\Gamma$ we have the bounds
\begin{equation} \label{bound for Gamma tilde}
C^{-1} \;\leq\; \wt \Gamma (z)\;\leq\; \frac{C\log N}{\min \h{\delta_- + \eta + E^2, \delta_+ + \theta}}.
\end{equation}
\end{enumerate}
\end{proposition}

\begin{proof}
The first bound of \eqref{bound for Gamma} follows from $(1 - m^2 S)^{-1} \f e = (1 - m^2)^{-1} \f e$ combined with \eqref{1-msquare}.
In order to prove the second bound of \eqref{bound for Gamma}, we write
$$
    \frac{1}{1-m^2 S} \;=\; \frac{1}{2} \frac{1}{1 - \frac{1+m^2S}{2}}
$$
and observe that
\begin{equation} \label{q def}
\normbb{ \frac{1+m^2S}{2}}_{\ell^2 \to \ell^2} \;\leq\; \max_{\pm} \absbb{\frac{1\pm m^2}{2}} \;\eqd\; q\,.
\end{equation}
Therefore
\begin{align*}
\normbb{\frac{1}{1-m^2 S}}_{\ell^\infty\to\ell^\infty} &\;\leq\; \sum_{n=0}^{n_0 - 1} \normbb{\frac{1+m^2S}{2}}^n
_{\ell^\infty\to\ell^\infty}
+ \sqrt{N}  \sum_{n=n_0}^\infty  \normbb{\frac{1+m^2S}{2}}^n
_{\ell^2\to\ell^2}
\\
&\;\leq\; n_0 + \sqrt{N} \frac{q^{n_0}}{1 - q}
\\
&\;\leq\; \frac{C \log N}{1 - q}\,,
\end{align*}
where in the last step we chose $n_0 = \frac{C_0 \log N}{1 - q}$ for large enough $C_0$.
Here we used that  $\norm{S}_{\ell^\infty\to\ell^\infty}\leq 1$ and \eqref{m is bounded}
to estimate the summands in the first sum. This concludes the proof of the second bound of \eqref{bound for Gamma}. The third bound of \eqref{bound for Gamma} follows from the elementary estimates
\begin{equation} \label{elementary bounds for q}
\absbb{\frac{1 - m^2}{2}} \;\leq\; 1 - c (\eta + E^2)\,, \qquad \absbb{\frac{1 + m^2}{2}} \;\leq\; 1 - c \pbb{(\im m)^2 + \frac{\eta}{\im m + \eta}} \;\leq\; 1 - c \theta
\end{equation}
for some universal constant $c > 0$, where in the last step we used Lemma \ref{lemma: msc}.

The estimate \eqref{bound for Gamma with gap} follows similarly. Due to the gap $\delta_-$ in the spectrum of $S$, we may replace the estimate \eqref{q def} with
\begin{equation} \label{modified q}
\normbb{ \frac{1+m^2S}{2}}_{\ell^2 \to \ell^2} \;\leq\; \max \hbb{1 - \delta_- - \eta - E^2 \,,\, \absbb{\frac{1 + m^2}{2}}}\,.
\end{equation}
Hence  \eqref{bound for Gamma with gap} follows using \eqref{elementary bounds for q}.

The lower bound of \eqref{bound for Gamma tilde} was proved in \eqref{varrho geq c}. The upper bound is proved similarly to \eqref{bound for Gamma with gap}, except that \eqref{modified q} is replaced with
\begin{equation*}
\normbb{ \frac{1+m^2S}{2} \biggr|_{\f e^\perp}}_{\ell^2 \to \ell^2} \;\leq\; \max \hBB{1 - \delta_- - \eta - E^2\,,\, \min \hbb{1 - \delta_+ \,,\, \absbb{\frac{1 + m^2}{2}}}}\,.
\end{equation*}
This concludes the proof of \eqref{bound for Gamma tilde}.
\end{proof}

The following proposition gives the behaviour of the spectral gaps $\delta_\pm$ for the example matrices from Section~\ref{sec: example}.

\begin{proposition}[Spectrum of $S$ for example matrices] \label{prop: spectrum of S}
\begin{itemize}
\item[(i)] If $H$ is an $a$-full Wigner matrix then $\delta_- \geq a$ and $\delta_+ \geq a$. 

\item[(ii)]
 If $H$ is a band matrix there is a  positive 
constant $c$, depending on the dimension $d$ and the profile function $f$, such that
 $\delta_- \geq c$ and $\delta_+ \geq c (W/L)^2$.

\item [(iii)]
If $H = \sqrt{1-\nu}H_B + \sqrt{\nu}
 H_W$, where $H_B$ is a band matrix, $H_W$ is an $a$-full Wigner matrix independent of $H_B$,
and $\nu\in [0,1]$ (see Definition~\ref{def: bandwigner}), then there is a constant $c$ depending only on the dimension $d$ and the profile function $f$ of $H_B$, such that
 $\delta_- \geq c$ and $\delta_+ \geq c (W/L)^2+\nu a$.
\end{itemize}
\end{proposition}

\begin{proof}
For the case where $H$ is an $a$-full Wigner matrix, the claim easily follows by splitting
\begin{equation*}
S \;=\; (S - a \f e \f e^*) + a \f e \f e^*\,.
\end{equation*}
By assumption, the first term is $(1 - a)$ times a doubly stochastic matrix. Hence its
 spectrum lies in $[-1 + a, 1 - a]$. The claims on $\delta_\pm$ now follow easily.

The claims about band matrices were proved in Lemma A.1 of \cite{EYY} and Equation (5.16) of \cite{EKYY3}, respectively.
Finally, (iii) easily follows from (i) and (ii).
\end{proof}

\section{Proof of Theorems \ref{thm: averaging} 
and \ref{thm: averaging with Lambdao}} \label{app: fluct averaging}

Theorems \ref{thm: averaging} and \ref{thm: averaging with Lambdao} are essentially simple special cases of the much more involved, and general, fluctuation averaging estimate from \cite{EKYfluc}. Nevertheless, here we give the details of the proofs because (a) they do not strictly follow from the formulation of the result in \cite{EKYfluc}, and (b) their proof is much easier than that of \cite{EKYfluc}, so that the reader only interested in the applications of fluctuation averaging to the local semicircle law need not read the lengthy proof of \cite{EKYfluc}.
We start with a simple lemma which summarizes the key properties of $\prec$ when combined with expectation.

\begin{lemma} \label{lem: exp prec}
Suppose that the deterministic control parameter $\Psi$ satisfies $\Psi \geq N^{-C}$, and that for all $p$ there is a constant $C_p$ such that the nonnegative 
random variable $X$ satisfies $\E X^p \leq N^{C_p}$. Suppose moreover that that $X \prec \Psi$. Then for any
fixed $n\in \N$ we have
\begin{equation}\label{EX}
\qquad \E X^n \;\prec\; \Psi^n\,.
\end{equation}
(Note that this estimate involves deterministic quantities only, i.e.\
it means that $\E X^n\leq N^\e \Psi^n$ for any $\e>0$ if $N\geq N_0(n,\e)$.) Moreover, we have
\begin{equation} \label{Qi X}
P_i X^n \;\prec\; \Psi^n\,, \qquad Q_i X^n \;\prec\; \Psi^n
\end{equation}
uniformly in $i$.
If $X = X(u)$ and $\Psi = \Psi(u)$ depend on some parameter $u$ and the 
above assumptions are uniform in $u$, then so are the conclusions.
\end{lemma}

\begin{proof}[Proof of Lemma \ref{lem: exp prec}]
It is enough to consider the case $n=1$; the case of larger $n$ follows immediately from the case $n = 1$, using the basic properties of $\prec$
from Lemma \ref{lemma: basic properties of prec}.

For the first claim, pick $\epsilon > 0$. Then
\begin{equation*}
\E X \;=\; \E X \ind{X \leq N^{\epsilon} \Psi} + \E X \ind{X > N^{\epsilon} \Psi} \;\leq\; N^{\epsilon} \Psi + \sqrt{\E X^2} \sqrt{\P(X > N^\epsilon \Psi)} \;\leq\; N^\epsilon \Psi + N^{C_2/2 - D/2}\,,
\end{equation*}
for arbitrary $D > 0$. The first claim therefore follows by choosing $D$ large enough.

The second claim follows from Chebyshev's inequality, using a high-moment estimate combined with Jensen's inequality for partial expectation. We omit the details, which are similar to those of the first claim.
\end{proof}

We shall apply Lemma \ref{lem: exp prec} to resolvent entries of $G$. In order to verify its assumptions, we record the following bounds.

\begin{lemma}\label{lemma:B2}
Suppose that $\Lambda \prec \Psi$ and $\Lambda_o \prec \Psi_o$ for some deterministic control parameters $\Psi$ and $\Psi_o$ both satisfying \eqref{admissible Psi}. Fix $p \in \N$. Then for any $i \neq j$ and  $\bb T \subset \{1, \dots, N\}$ satisfying $\abs{\bb T} \leq p$ and $i,j \notin \bb T$ we have
\begin{equation} \label{basic bounds on GT}
G_{ij}^{(\bb T)} \;=\; O_\prec(\Psi_o) \,, \qquad \frac{1}{G_{ii}^{(\bb T)}} \;=\; O_\prec(1)\,.
\end{equation}
Moreover, we have the rough bounds $\absb{G_{ij}^{(\bb T)}} \leq M$ and
\begin{equation} \label{rough bound on 1/G}
\E \absBB{\frac{1}{G_{ii}^{(\bb T)}}}^n \;\leq\; N^\epsilon
\end{equation}
for any $\epsilon > 0$ and $N \geq N_0(n, \epsilon)$.
\end{lemma}

\begin{proof}
The bounds \eqref{basic bounds on GT} follow easily by a repeated application of \eqref{resolvent expansion type 1}, the assumption $\Lambda \prec M^{-c}$, and the lower bound in \eqref{m is bounded}. The deterministic bound
 $\absb{G_{ij}^{(\bb T)}} \leq M$ follows immediately from $\eta \geq M^{-1}$ by definition of a spectral domain.

In order to prove \eqref{rough bound on 1/G}, we use Schur's complement formula \eqref{schur} applied to $1/ G_{ii}^{(\bb T)}$, where the expectation is estimated using \eqref{finite moments} and $\absb{G_{ij}^{(\bb T)}} \leq M$. (Recall \eqref{lower bound on W}.) This gives
\begin{equation*}
\E \absBB{\frac{1}{G_{ii}^{(\bb T)}}}^p \;\prec\; N^{C_p}
\end{equation*}
for all $p \in \N$. Since $1 / G_{ii}^{(\bb T)} \prec 1$, \eqref{rough bound on 1/G} therefore follows from \eqref{EX}.
\end{proof}

\begin{proof}[Proof of Theorem \ref{thm: averaging with Lambdao}]
First we claim that, for any fixed $p \in \N$, we have
\begin{equation} \label{Xi estimate}
\absbb{Q_k \frac{1}{G^{(\bb T)}_{kk}}} \;\prec\; \Psi_o
\end{equation}
uniformly for $\bb T \subset \{1, \dots, \N\}$, $\abs{\bb T} \leq p$, and $k \notin \bb T$. To simplify notation, for the proof we set $\bb T = \emptyset$; the proof for nonempty $\bb T$ is the same.
From Schur's complement formula \eqref{schur} we get $\abs{Q_k  (G_{kk})^{-1}} \leq \abs{h_{kk}} + \abs{Z_k}$.
The first term is estimated by $\abs{h_{kk}} \prec M^{-1/2} \leq \Psi_o$. The second term is estimated exactly as in \eqref{Zi split} and \eqref{estimate for sum Gkl}:
\begin{equation*}
\abs{Z_k} \;\prec\; \pBB{\sum_{x \neq y}^{(k)} s_{kx} \absb{G_{xy}^{(k)}}^2 s_{yk}}^{1/2} \;\prec\; \Psi_o\,,
\end{equation*}
where in the last step we used that $\absb{G_{xy}^{(k)}} \prec \Psi_o$ as follows from \eqref{basic bounds on GT},
 and the bound $1 / \abs{G_{kk}} \prec 1$ (recall that $\Lambda \prec \Psi \leq M^{-c}$). This concludes the proof of \eqref{Xi estimate}.

Abbreviate $X_k \deq Q_k (G_{kk})^{-1}$. We shall estimate $\sum_{k} t_{ik} X_k$ in probability by estimating its $p$-th moment by $\Psi_o^{2p}$, from which the claim will easily follow using Chebyshev's inequality. Before embarking on the estimate for arbitrary $p$, we illustrate its idea by estimating the variance 
\begin{equation} \label{variance of average of X}
\E \absbb{\sum_k t_{ik} X_k}^2 \;=\; \sum_{k,l} t_{ik} \ol t_{il} \,  \E X_{k} \ol X_{l} \;=\; \sum_{k} \abs{t_{ik}}^2 \, \E X_{k} \ol X_{k} + \sum_{k \neq l} t_{ik} \ol t_{il} \, \E X_{k} \ol X_{l}\,.
\end{equation}
Using Lemma \ref{lem: exp prec} and the bounds \eqref{condition on weights} on $t_{ik}$, we find that the first term on the right-hand side of \eqref{variance of average of X} is $O_\prec(M^{-1} \Psi_o^2) = O_\prec(\Psi_o^4)$, where we used the estimate \eqref{admissible Psi}. Let us therefore focus on the second term of \eqref{variance of average of X}. Using the fact that $k \neq l$, we apply \eqref{resolvent expansion type 1} to $X_{k}$ and $X_{l}$ to get
\begin{equation} \label{variance calculation}
\E X_{k} \ol X_{l} \;=\; \E Q_{k} \pbb{\frac{1}{G_{k k}}} Q_{l} \ol{\pbb{\frac{1}{G_{l l}}}} \;=\;
\E Q_{k} \pbb{\frac{1}{G_{kk}^{(l)}} - \frac{G_{kl} G_{lk}}{G_{kk} G_{kk}^{(l)} G_{ll}}} Q_{l} \ol{\pbb{\frac{1}{G_{ll}^{(k)}} - \frac{G_{lk} G_{kl}}{G_{ll} G_{ll}^{(k)} G_{kk}}}}\,.
\end{equation}
We multiply out the parentheses on the right-hand side. The crucial observation is that if the random variable $Y$ is independent of $i$ (see Definition \ref{definition: P Q}) then $\E Q_i(X) Y = \E Q_i(XY) = 0$. Hence out of the four terms obtained from the right-hand side of \eqref{variance calculation}, the only nonvanishing one is
\begin{equation*}
\E Q_{k} \pbb{\frac{G_{kl} G_{lk}}{G_{kk} G_{kk}^{(l)} G_{ll}}} Q_{l} \ol{\pbb{\frac{G_{lk} G_{kl}}{G_{ll} G_{ll}^{(k)} G_{kk}}}} \;\prec\; \Psi_o^4\,.
\end{equation*}
Together with \eqref{condition on weights}, this concludes the proof of $\E \absb{\sum_{k}t_{ik} X_k}^2 \prec \Psi_o^4$.

After this pedagogical interlude we move on to the full proof. Fix some even integer $p$ and write
\begin{equation*}
\E \absbb{\sum_k t_{ik} X_k}^p \;=\; \sum_{k_1, \dots, k_p} t_{ik_1} \cdots t_{i k_{p/2}} \ol t_{i k_{p/2 + 1}} \cdots \ol t_{ik_p} \, \E X_{k_1} \cdots X_{k_{p/2}} \ol X_{k_{p/2 + 1}} \cdots \ol X_{k_p}\,.
\end{equation*}
Next, we regroup the terms in the sum over $\f k \deq (k_1, \dots, k_p)$ according to the partition of $\{1, \dots, p\}$ generated by the indices $\f k$. To that end, let $\fra P_p$ denote the set of partitions of $\{1, \dots, p\}$, and $\cal P(\f k)$ the element of $\fra P_p$ defined by the equivalence relation $r \sim s$ if and only if $k_r = k_s$.
In short, we reorganize the summation according to coincidences among the indices $\f k$. Then we write
\begin{equation} \label{Xp expanded}
\E \absbb{\sum_k t_{ik} X_k}^p \;=\; \sum_{P \in \fra P_p} \sum_{\f k} t_{ik_1} \cdots t_{i k_{p/2}} \ol t_{i k_{p/2 + 1}} \cdots \ol t_{ik_p}\, \ind{\cal P(\f k) = P} V(\f k) \,,
\end{equation}
where we defined
\begin{equation*}
V(\f k) \;\deq\; \E X_{k_1} \cdots X_{k_{p/2}} \ol X_{k_{p/2 + 1}} \cdots \ol X_{k_p}\,.
\end{equation*}
Fix $\f k$ and set $P \deq \cal P(\f k)$ to be partition induced by the coincidences in $\f k$.
For any $r\in  \{1, \dots, p\}$, we denote by  $[r]$ the block of $r$ in $P$. 
Let $L \equiv L(P) \deq \h{r \col [r] = \{r\}} \subset \{1, \dots, p\}$ be the set of ``lone'' labels.
We denote by $\f k_L \deq (k_r)_{r \in L}$ the summation indices associated with lone labels.

The resolvent entry $G_{kk}$ depends strongly on the randomness in the $k$-column of $H$, but only weakly on the randomness in the other columns.
We conclude that if $r$ is a lone label then all factors $X_{k_s}$ with $s \neq r$
in $V(\f k)$ depend weakly on the randomness in the $k_r$-th column of $H$. Thus, the idea is to make all resolvent entries inside the expectation of $V(\f k)$ as independent of the indices $\f k_L$ as possible (see Definition \ref{definition: P Q}),
using the identity \eqref{resolvent expansion type 1}. To that end, we say that a resolvent entry $G_{xy}^{(\bb T)}$ with $x,y \notin \bb T$ is \emph{maximally expanded} if $\f k_L \subset \bb T \cup \{x,y\}$. 
The motivation behind this definition is that using \eqref{resolvent expansion type 1} we cannot add upper indices from the set $\f k_L$ to a maximally expanded resolvent entry. We shall apply
 \eqref{resolvent expansion type 1} to all resolvent entries in $V(\f k)$.
In this manner we generate a sum of monomials consisting of off-diagonal resolvent entries
and inverses of diagonal resolvent entries. We can now repeatedly apply
 \eqref{resolvent expansion type 1} to each factor
 until either they are all maximally expanded or a sufficiently large 
 number of off-diagonal resolvent entries has been generated. 
 The cap on the number of off-diagonal entries is introduced to 
ensure that this procedure terminates after a finite number of steps.

In order to define the precise algorithm, let $\cal A$ denote the set of monomials in the off-diagonal entries $G_{xy}^{(\bb T)}$, with $\bb T \subset \f k_L$, $x \neq y$, and $x,y \in \f k \setminus \bb T$, as well as the inverse diagonal entries $1 / G_{xx}^{(\bb T)}$, with $\bb T \subset \f k_L$ and $x \in \f k \setminus \bb T$. Starting from $V(\f k)$, the algorithm will recursively generate sums of monomials in $\cal A$.
Let $d(A)$ denote the number of off-diagonal entries in $A \in \cal A$. For $A \in \cal A$ we shall define $w_0(A), w_1(A) \in \cal A$ satisfying
\begin{equation} \label{properties of w}
A \;=\; w_0(A) + w_1(A) \,, \qquad d(w_0(A)) \;=\; d(A) \,, \qquad d(w_1(A)) \;\geq\; \max \hb{2, d(A) + 1}\,.
\end{equation}
The idea behind this splitting is to use \eqref{resolvent expansion type 1} on one entry of $A$; the first term on the right-hand side of \eqref{resolvent expansion type 1} gives rise to $w_0(A)$ and the second to $w_1(A)$. The precise definition of the algorithm applied to $A \in \cal A$ is as follows.

\begin{itemize}
\item[(1)]
If all factors of $A$ are maximally expanded or $d(A) \geq p + 1$ then 
stop the expansion of $A$. In other words, the algorithm cannot be applied to $A$ in the future.
\item[(2)]
Otherwise choose some (arbitrary) factor of $A$ that is not maximally expanded. If this entry is off-diagonal, $G^{(\bb T)}_{xy}$, write
\begin{equation} \label{splitting off-diag}
G^{(\bb T)}_{xy} \;=\; G_{xy}^{(\bb T u)} + \frac{G_{xu}^{(\bb T)} G_{uy}^{(\bb T)}}{G_{uu}^{(\bb T)}}
\end{equation}
for the smallest $u \in \f k_L \setminus (\bb T \cup \{x,y\})$. If the chosen entry is diagonal, $1/ G^{(\bb T)}_{xx}$, write
\begin{equation} \label{splitting diag}
\frac{1}{G_{xx}^{(\bb T)}} \;=\; \frac{1}{G_{xx}^{(\bb Tu)}} - \frac{G_{xu}^{(\bb T)} G_{ux}^{(\bb T)}}{G_{xx}^{(\bb T)} G_{xx}^{(\bb Tu)} G_{uu}^{(\bb T)}}
\end{equation}
for the smallest $u \in \f k_L \setminus (\bb T \cup \{x\})$. Then the splitting $A = w_0(A) + w_1(A)$ is defined by the splitting induced by \eqref{splitting off-diag} or \eqref{splitting diag}, in the sense that we replace the factor $G^{(\bb T)}_{xy}$ or $1/G_{xx}^{(\bb T)}$ in the monomial $A$ by the right-hand sides of  \eqref{splitting off-diag} or \eqref{splitting diag}.
\end{itemize}
(This algorithm contains some arbitrariness in the choice of the factor of $A$ to be expanded. It may be removed for instance by first fixing some ordering  of all resolvent entries $G_{ij}^{(\bb T)}$. Then in (2) we choose the first factor of $A$ that is not maximally expanded.) Note that \eqref{splitting off-diag} and \eqref{splitting diag} follow from \eqref{resolvent expansion type 1}. It is clear that \eqref{properties of w} holds with the algorithm just defined.

We now apply this algorithm recursively to each entry $A^{r} \deq 1 / G_{k_r k_r}$ 
in the definition of $V(\f k)$. More precisely, we start with $A^r$ and define $A_{0}^r \deq w_0(A^r)$ and $A_{1}^r \deq w_1(A^r)$. In the second step of the algorithm we define four monomials
$$
 A_{00}^r \;\deq\; w_0(A_0^r)\,, \qquad A_{01}^r \;\deq\; w_0(A_1^r)\,, \qquad  
A_{10}^r \;\deq\; w_1(A_0^r)\,, \qquad A_{11}^r \;\deq\; w_1(A_1^r)\,,
$$
and so on, at each iteration performing the steps (1) and (2) on each new monomial independently of the others. 
Note that the lower indices are binary sequences that describe the recursive application of the
operations $w_0$ and $w_1$.
In this manner we generate a binary tree whose vertices are given by finite binary strings $\sigma$. 
The associated monomials satisfy $A_{\sigma i}^r \deq w_i(A_\sigma^r)$ for $i = 0,1$,
where $\sigma i$ denotes the binary string obtained by appending $i$ to the right end of $\sigma$. See Figure \ref{fig: tree} for an illustration of the tree.

\begin{figure}[ht!]
\begin{center}
\includegraphics{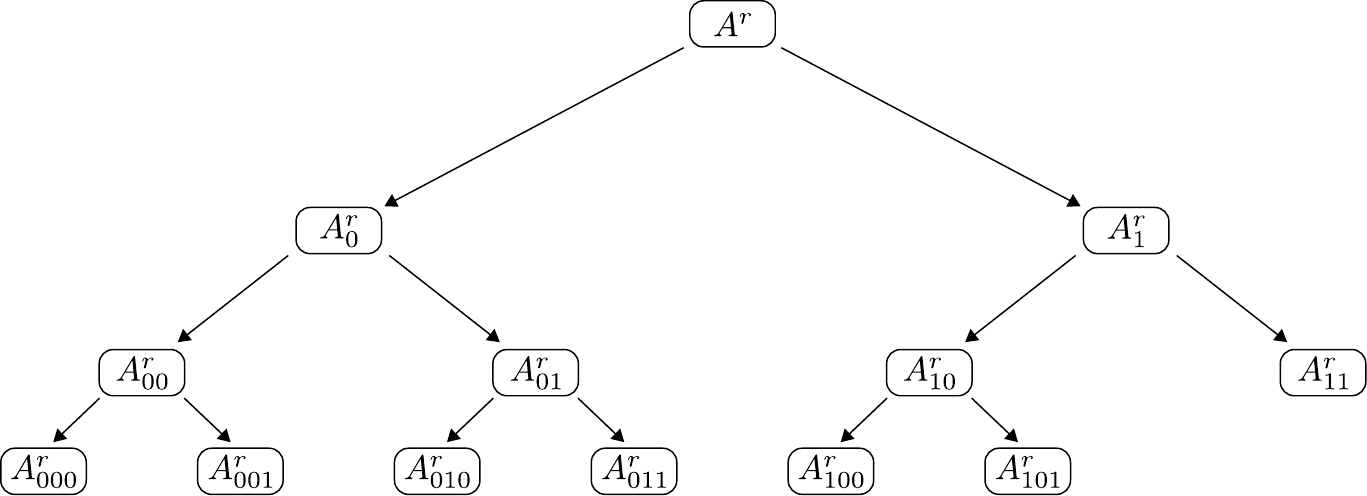}
\end{center}
\caption{The binary tree generated by applying the algorithm (1)--(2) to a monomial $A^r$. Each vertex of the tree is indexed by a binary string $\sigma$, and encodes a monomial $A^r_\sigma$. An arrow towards the left represents the action of $w_0$ and an arrow towards the right the action of $w_1$. The monomial $A_{11}^r$ satisfies the assumptions of step (1), and hence its expansion is stopped, so that the tree vertex $11$ has no children. \label{fig: tree}}
\end{figure}

We stop the recursion of a tree vertex whenever the associated monomial satisfies the stopping rule of step (1). In other words, the set of leaves of the tree is the set of binary strings $\sigma$ such that either all factors of $A^r_\sigma$ are maximally expanded or $d(A^r_\sigma) \geq p + 1$. We claim that the resulting binary tree is finite, i.e.\ that the algorithm always reaches step (1) after a finite number of iterations. Indeed, by the stopping rule in (1), we have $d(A^r_\sigma) \leq p + 1$ for any vertex $\sigma$ of the tree. Since each application of $w_1$ increases $d(\cdot)$ by at least one, and in the first step (i.e.\ when applied to $A^r$) by two,
we conclude that the number of ones in any $\sigma$ is at most $p$. Since each application of $w_1$ increases the number of resolvent entries by at most four, and the application of $w_0$ does not change this number,
we find that the number of resolvent entries in $A^r_\sigma$ is bounded by $4 p + 1$. Hence the maximal number of
 upper indices in $A^r_\sigma$ for any tree vertex $\sigma$ is $(4p+1)p$.
Since each application of $w_0$ increases the total number of upper indices by one, we find that $\sigma$ contains at most $(4p+1)p$  zeros. We conclude that the maximal length of the string $\sigma$ (i.e.\ the depth of the tree) is at most $(4p +1)p+ p = 4p^2+2p$. 
A string $\sigma$ encoding a tree vertex contains at most $p$ ones. Denoting by $k$ the number of ones in a string encoding a leaf of the tree, we find that the number of leaves is bounded by $\sum_{k = 0}^p \binom{4p^2 + 2p}{k} 
\leq (Cp^2)^{p}$.
Therefore, denoting by $\cal L_r$ the set of leaves of the binary tree generated from $A^r$, 
we have $\abs{\cal L_r} \leq (Cp^2)^p$.

By definition of the tree and $w_0$ and $w_1$, we have the decomposition
\begin{equation} \label{X split}
X_{k_r} \;=\; Q_{k_r} \sum_{\sigma \in \cal L_r} A_\sigma^r\,.
\end{equation}
Moreover, each monomial $A_\sigma^r$ for $\sigma \in \cal L_r$ either consists entirely of maximally expanded resolvent entries or satisfies $d(A_\sigma^r) = p + 1$. (This is an immediate consequence of the stopping rule in (1)).

Next, we observe that for any string $\sigma$ we have
\begin{equation} \label{A sigma size}
A_\sigma^k \;=\; O_\prec \pb{\Psi_o^{b(\sigma) + 1}}\,,
\end{equation}
where $b(\sigma)$ is the number ones in the string $\sigma$. Indeed, if $b(\sigma) = 0$ then this follows from \eqref{Xi estimate}; if $b(\sigma) \geq 1$ this follows from the last statement in \eqref{properties of w} and \eqref{basic bounds on GT}.

Using \eqref{Xp expanded} and \eqref{X split} we have the representation
\begin{equation} \label{V in terms of leaves}
V(\f k) \;=\; \sum_{\sigma_1 \in \cal L_1} \cdots \sum_{\sigma_p \in \cal L_p} \E \pb{Q_{k_1} A_{\sigma_1}^1} \cdots \pb{Q_{k_p} \ol{A_{\sigma_p}^p}}\,.
\end{equation}

We now claim that any nonzero term on the right-hand side of \eqref{V in terms of leaves} satisfies
\begin{equation} \label{estimate of integrand}
\pb{Q_{k_1} A_{\sigma_1}^1} \cdots \pb{Q_{k_p} \ol{A_{\sigma_p}^p}} \;=\; O_\prec \pb{\Psi_o^{p + \abs{L}}}\,.
\end{equation}

\begin{proof}[Proof of \eqref{estimate of integrand}]
Before embarking on the proof, we explain its idea.
By \eqref{A sigma size}, the naive size of the left-hand side of \eqref{estimate of integrand} is $\Psi_o^p$. The key observation is that each
lone label $s\in L$ yields one extra factor $\Psi_o$ to the estimate. This is because the expectation 
in \eqref{V in terms of leaves} would vanish if all other factors  $\pb{Q_{k_r} A_{\sigma_r}^r}$, $r\ne s$,
were independent of $k_s$. The expansion of the binary tree makes this dependence explicit
by exhibiting $k_s$ as a lower index. But this requires performing an operation $w_1$
with the choice $u=k_s$ in \eqref{splitting off-diag} or \eqref{splitting diag}.
However, $w_1$ increases the number of off-diagonal element by at least one. 
In other words, every index associated with a lone label must have a ``partner'' index in a different resolvent entry which arose by application of $w_1$.
Such a partner index may only be obtained through the creation of at least one off-diagonal resolvent entry. The actual proof
below shows that this effect applies \emph{cumulatively} for all lone labels.

In order to prove \eqref{estimate of integrand}, we consider two cases. Consider first the case where for some $r = 1, \dots, p$ the monomial $A_{\sigma_r}^r$ on the left-hand side of \eqref{estimate of integrand} is not maximally expanded. Then $d(A_{\sigma_r}^r) = p + 1$, so that \eqref{basic bounds on GT} yields $A_{\sigma_r}^r \prec \Psi_o^{p + 1}$. Therefore the observation that $A_{\sigma_s}^s \prec \Psi_o$ for all $s \neq r$, together with \eqref{Qi X} implies that the left-hand side of \eqref{estimate of integrand} is $O_\prec\pb{\Psi_o^{2p}}$. Since $\abs{L} \leq p$, \eqref{estimate of integrand} follows.

Consider now the case where $A_{\sigma_r}^r$ on the left-hand side of \eqref{estimate of integrand} is maximally expanded for all $r = 1, \dots, p$. The key observation is the following claim about the left-hand side of \eqref{estimate of integrand} with a nonzero expectation.
\begin{itemize}
\item[$(*)$]
For each $s \in L$ there exists $r = \tau(s) \in \{1, \dots, p\} \setminus \{s\}$ such that the monomial $A_{\sigma_r}^r$ contains a resolvent entry with lower index $k_s$.
\end{itemize}
In other words,  after expansion, the lone label $s$ has a ``partner'' label $r = \tau(s)$, such that the index
$k_s$ appears also in the expansion of $A^r$ (note that there may be several such partner labels $r$). 
To prove $(*)$, suppose by contradiction that there exists an $s \in L$ such that for all $r \in \{1, \dots, p\} \setminus \{s\}$ the lower index $k_s$ does not appear in the monomial $A_{\sigma_r}^r$. To simplify notation, we assume that $s = 1$. Then, for all $r = 2, \dots, p$, since $A_{\sigma_r}^r$ is maximally expanded, we find that $A_{\sigma_r}^r$ is independent of $k_1$ (see Definition \ref{definition: P Q}).
Therefore we have
\begin{equation*}
\E \pb{Q_{k_1} A_{\sigma_1}^1} \pb{Q_{k_2} A_{\sigma_2}^2} \cdots \pb{Q_{k_p} \ol{A_{\sigma_p}^p}} \;=\; \E Q_{k_1} \pB{A_{\sigma_1}^1 \pb{Q_{k_2} A_{\sigma_2}^2} \cdots \pb{Q_{k_p} \ol{A_{\sigma_p}^p}}} \;=\; 0\,,
\end{equation*}
where in the last step we used that $\E Q_i(X) Y = \E Q_i(XY) = 0$ if $Y$ is independent of $i$.
This concludes the proof of $(*)$.

For $r \in \{1, \dots, p\}$ we define $\ell(r) \deq \sum_{s \in L} \ind{\tau(s) = r}$, the number of times that the label $r$ was chosen as a partner
to some lone label $s$. We now claim that 
\begin{equation} \label{key linking estimate}
A_{\sigma_r}^r \;=\; O_\prec \pb{\Psi_o^{1 + \ell(r)}}\,.
\end{equation}
To prove \eqref{key linking estimate}, fix $r \in \{1, \dots, p\}$. By definition, for each $s \in \tau^{-1}(\{r\})$ the index $k_s$ appears as a lower index in the monomial $A_{\sigma_r}^r$. Since $s \in L$ is by definition a lone label
and $s \neq r$, we know that $k_s$ does not appear as an index in $A^r$. By definition of the monomials associated with the tree vertex $\sigma_r$, it follows that $b(\sigma_r)$, the number of ones in $\sigma_r$, is at least $\absb{\tau^{-1}(\{r\})} = \ell(r)$
since each application of $w_1$ adds precisely one new (lower) index.
Note that in this step it is crucial that $s \in \tau^{-1}(\{r\})$ was a lone label. Recalling \eqref{A sigma size}, we therefore get \eqref{key linking estimate}.

Using \eqref{key linking estimate} and Lemma \ref{lem: exp prec} we find
\begin{equation*}
\absB{\pb{Q_{k_1} A_{\sigma_1}^1} \cdots \pb{Q_{k_p} \ol{A_{\sigma_p}^p}}} \;\prec\; \prod_{r = 1}^p \Psi_o^{1 + \ell(r)} \;=\; \Psi_o^{p + \abs{L}}\,.
\end{equation*}
This concludes the proof of \eqref{estimate of integrand}.
\end{proof}

Summing over the binary trees in \eqref{V in terms of leaves} and using Lemma \ref{lem: exp prec}, we get from \eqref{estimate of integrand}
\begin{equation} \label{bound for Vk}
V(\f k) \;=\; O_\prec\pb{\Psi_o^{p + \abs{L}}}\,.
\end{equation}
We now return to the sum \eqref{Xp expanded}. We perform the summation by first fixing $P \in \fra P_p$, with associated lone labels $L = L(P)$. We find
\begin{equation*}
\absbb{\sum_{\f k} \ind{\cal P(\f k) = P} \, t_{ik_1} \cdots t_{i k_{p/2}} \ol t_{i k_{p/2 + 1}} \cdots \ol t_{ik_p}} \;\leq\;   (M^{-1})^{p-|P|} \;\leq\; (M^{-1/2})^{p - \abs{L}}\,;
\end{equation*}
in the first step we used \eqref{condition on weights} and the fact that the summation is performed over $\abs{P}$ free indices, the remaining $p - \abs{P}$ being estimated by $M^{-1}$; in the second step we used that  
each block of $P$ that is not contained in $L$ consists of at least two labels, so that $p - \abs{P} \geq (p-\abs{L})/2$.
From \eqref{Xp expanded} and \eqref{bound for Vk} we get
\begin{equation*}
\E \absbb{\sum_k t_{ik} X_k}^p \;\prec\; \sum_{P \in \fra P_p} (M^{-1/2})^{p - \abs{L(P)}} \, \Psi_o^{p + \abs{L(P)}}  \;\leq\; C_p \Psi_o^{2p}\,,
\end{equation*}
where in the last step we used the lower bound from \eqref{admissible Psi} and estimated the summation over $\fra P_p$ with a constant $C_p$ (which is bounded by $(Cp^2)^p$). Summarizing, we have proved that
\begin{equation} \label{final moment estimate}
\E \absbb{\sum_k t_{ik} X_k}^p \;\prec\; \Psi_o^{2p}
\end{equation}
for any $p \in 2 \N$.

We conclude the proof of Theorem \ref{thm: averaging with Lambdao} with a simple application of 
Chebyshev's inequality. Fix $\epsilon > 0$ and $D > 0$. Using \eqref{final moment estimate} and Chebyshev's inequality we find
\begin{equation*}
\P \pbb{\absbb{\sum_k t_{ik} X_k} > N^\epsilon \Psi_o^2} \;\leq\; N \, N^{-\epsilon p}
\end{equation*}
for large enough $N \geq N_0(\epsilon, p)$. Choosing $p \geq \epsilon^{-1} (1 + D)$ concludes
 the proof of Theorem \ref{thm: averaging with Lambdao}.
\end{proof}

\begin{remark}
The identity \eqref{resolvent expansion type 1} is the only identity about the entries of $G$ that is needed in the proof of Theorem \ref{thm: averaging with Lambdao}. In particular, \eqref{resolvent expansion type 2} is never used, and the actual entries of $H$ never appear in the argument.
\end{remark}

\begin{proof}[Proof of Theorem \ref{thm: averaging}]
The first estimate of \eqref{averaging with Q} follows from Theorem \ref{thm: averaging with Lambdao} and the simple bound $\Lambda_o\leq \Lambda\prec \Psi$.
The second estimate of \eqref{averaging with Q} may be proved by following the proof of 
Theorem \ref{thm: averaging with Lambdao} verbatim; 
the only modification is the bound
\begin{equation*}
\absb{Q_k G_{kk}^{(\bb T)}} \;=\; \absb{Q_k \pb{G_{kk}^{(\bb T)} - m}} \;\prec\; \Psi\,,
\end{equation*}
which replaces \eqref{Xi estimate}. Here we again use the same upper bound $\Psi_o = \Psi$ 
for $\Lambda$ and $\Lambda_o$.

In order to  prove \eqref{averaging without Q},  we write Schur's complement formula \eqref{schur} using \eqref{identity for msc} as
\begin{equation} \label{inverse Schur}
\frac{1}{G_{ii}} \;=\; \frac{1}{m} + h_{ii} - \pbb{\sum_{k,l}^{(i)} h_{ik} G_{kl}^{(i)} h_{li} - m}\,.
\end{equation}
Since $\abs{h_{ii}} \prec M^{-1/2} \leq \Psi$ and $\abs{1/G_{ii} - 1/m} \prec \Psi$, we find that the term in parentheses is stochastically dominated by $\Psi$. Therefore we get, inverting \eqref{inverse Schur} and expanding the right-hand side, that
\begin{equation*}
v_i \;=\; G_{ii} - m \;=\; m^2 \pbb{-h_{ii} + \sum_{k,l}^{(i)} h_{ik} G_{kl}^{(i)} h_{li} - m} + O_\prec(\Psi^2)\,.
\end{equation*}
Taking the partial expectation $P_i$ yields
\begin{equation*}
P_i v_i \;=\; m^2 \pbb{\sum_k^{(i)} s_{ik} G_{kk}^{(i)} - m} + O_\prec(\Psi^2) \;=\; m^2 \sum_k s_{ik} v_k + O_\prec(\Psi^2)\,,
\end{equation*}
where in the second step we used \eqref{resolvent expansion type 1}, \eqref{s leq W}, and \eqref{basic bounds on GT}.
 Therefore we get, using \eqref{averaging with Q} and $Q_i G_{ii} = Q_i(G_{ii}-m) = Q_i v_i$,
\begin{equation*}
w_a \;\deq\; \sum_i t_{ai} v_i \;=\; \sum_i t_{ai} P_i v_i + \sum_i t_{ai} Q_i v_i \;=\; m^2 \sum_{i,k} t_{ai} s_{ik} v_k + O_\prec(\Psi^2) \;=\; m^2 \sum_{i,k} s_{ai} t_{ik} v_k + O_\prec(\Psi^2)\,,
\end{equation*}
where in the last step we used that the matrices $T$ and $S$ commute by assumption.
Introducing the vector $\f w = (w_a)_{a = 1}^N$ we therefore have the equation 
\begin{equation}\label{selfcons}
\f w = m^2 S \f w + O_\prec(\Psi^2)\,,
\end{equation}
where the error term is in the sense of the $\ell^\infty$-norm (uniform in the components of the vector $\f w$).
Inverting the matrix $1 - m^2 S$ and recalling the definition \eqref{def:rho} yields  \eqref{averaging without Q}. 

The proof of \eqref{averaging without Q avg}  is similar, except that we have to treat the subspace $\f e^\perp$ separately.  Using \eqref{sum_t_1}  we write
\begin{equation*}
\sum_i t_{ai} (v_i - [v]) \;=\; \sum_i t_{ai} v_i - \sum_i \frac{1}{N} v_i\,,
\end{equation*}
and apply the above argument to each term separately.
This yields
\begin{equation*}
\sum_i t_{ai} (v_i - [v]) \;=\; m^2 \sum_{i} t_{ai} \sum_k s_{ik} v_k  - m^2 \sum_i \frac{1}{N} \sum_k t_{ik} v_k + O_\prec(\Psi^2) \;=\; m^2 \sum_{i,k} s_{ai} t_{ik} (v_k - [v]) + O_\prec(\Psi^2)\,,
\end{equation*}
where we used \eqref{S is stochastic} in the second step.
Note that the error term on the right-hand side is perpendicular to $\f e$ when regarded as a vector indexed by $a$, since all other terms in the equation are. Hence we may invert the matrix $(1 - m^2 S)$ on the subspace $\f e^\perp$, as above, to get \eqref{averaging without Q avg}. 
\end{proof}

We conclude this section with an alternative proof of Theorem \ref{thm: averaging with Lambdao}. While the underlying argument remains similar, the following proof makes use of an additional decomposition of the space of random variables, which avoids the use of the stopping rule from Step (1) in the above proof of Theorem \ref{thm: averaging with Lambdao}. This decomposition may be regarded as an abstract reformulation of the stopping rule.

\begin{proof}[Alternative proof of Theorem \ref{thm: averaging with Lambdao}]
As before, we set $X_k \deq Q_k (G_{kk})^{-1}$. For simplicity of presentation, we set $t_{ik} = N^{-1}$. The decomposition is defined using the operations $P_i$ and $Q_i$, introduced in Definition \ref{definition: P Q}. It is immediate that $P_i$ and $Q_i$ are projections, that $P_i + Q_i = 1$, and that all of these projections commute with each other. For a set $A \subset \{1, \dots, N\}$ we use the notations $P_A \deq \prod_{i \in A} P_i$ and $Q_A \deq \prod_{i \in A} Q_i$.

Let $p$ be even and introduce the shorthand $\wt X_{k_s} \deq X_{k_s}$ for $s \leq p/2$ and $\wt X_{k_s} \deq \ol X_{k_s}$ for $s > p/2$. Then we get
\begin{equation*}
\E \absbb{\frac{1}{N} \sum_k X_k}^{p}
\;=\; \frac1{N^p}  \sum_{k_1, \dots, k_p}  \E \prod_{s = 1}^p \wt X_{k_s} \;=\;
\frac1{N^p}  \sum_{k_1, \dots, k_p}  \E \prod_{s = 1}^p \pBB{\prod_{r = 1}^p (P_{k_r} + Q_{k_r})\wt X_{k_s}}\,.
\end{equation*}
Introducing the notations $\f k = (k_1, \dots, k_p)$ and $[\f k] = \{k_1, \dots, k_p\}$, we therefore get by multiplying out the parentheses
\begin{equation} \label{Zp expanded}
\E \absbb{\frac{1}{N} \sum_k X_k}^{p} \;=\;
\frac1{N^p}  \sum_{\f k} \sum_{A_1, \dots, A_p \subset [\f k]}  \E \prod_{s = 1}^p \pb{P_{A_s^c} Q_{A_s} \wt X_{k_s}}\,.
\end{equation}

Next, by definition of $\wt X_{k_s}$, we have that $\wt X_{k_s} = Q_{k_s} \wt X_{k_s}$, which implies that $P_{A^c_s} \wt X_{k_s} = 0$ if $k_s \notin A_s$. Hence may restrict the summation to $A_s$ satisfying
\begin{equation} \label{C cond 1}
k_s \;\in\; A_s
\end{equation}
for all $s$. Moreover, we claim that the right-hand side of \eqref{Zp expanded} vanishes unless
\begin{equation} \label{C cond 2}
k_s \;\in\; \bigcup_{q \neq s} A_{q}
\end{equation}
for all $s$. Indeed, suppose that $k_s \in \bigcap_{q \neq s} A_{q}^c$ for some $s$, say $s = 1$. In this case, for each $s = 2, \dots, p$, the factor $P_{A_s^c} Q_{A_s} \wt X_{k_s}$ is independent of $k_1$ (see Definition \ref{definition: P Q}). Thus we get
\begin{multline*}
\E \prod_{s = 1}^p \pb{P_{A_s^c} Q_{A_s} \wt X_{k_s}} \;=\; \E \pb{P_{A_1^c} Q_{A_1} Q_{k_1} \wt X_{k_1}} \prod_{s = 2}^p \pb{P_{A_s^c} Q_{A_s} \wt X_{k_s}}
\\
=\; \E Q_{k_1} \pbb{\pb{P_{A_1^c} Q_{A_1} \wt X_{k_1}} \prod_{s = 2}^p \pb{P_{A_s^c} Q_{A_s} \wt X_{k_s}}} \;=\; 0\,,
\end{multline*}
where in the last step we used that $\E Q_i(X) = 0$ for any $i$ and random variable $X$.

We conclude that the summation on the right-hand side of \eqref{Zp expanded} is restricted to indices satisfying \eqref{C cond 1} and \eqref{C cond 2}. Under these two conditions we have
\begin{equation} \label{sum size A}
\sum_{s = 1}^p \abs{A_s} \;\geq\; 2 \, \abs{[\f k]}\,,
\end{equation}
since each index $k_s$ must belong to at least two different sets $A_q$: to $A_s$ (by \eqref{C cond 1}) as well as to some $A_q$ with $q \neq s$ (by \eqref{C cond 2}).

Next, we claim that for $k \in A$ we have
\begin{equation} \label{claim on size of QA}
\abs{Q_A X_k } \;\prec\; \Psi_o^{\abs{A}}\,.
\end{equation}
(Note that if we were doing the case $X_k = Q_k G_{kk}$ instead of $X_k = Q_k (G_{kk})^{-1}$, then \eqref{claim on size of QA} would have to be weakened to $\abs{Q_A X_k } \prec \Psi^{\abs{A}}$, in accordance with \eqref{averaging with Q}. Indeed, in that case and for $A = \{k\}$, we only have the bound $\abs{Q_k G_{kk}} \prec \Psi$ and not $\abs{Q_k G_{kk}} \prec \Psi_o$.)

Before proving \eqref{claim on size of QA}, we show it may be used to complete the proof. Using \eqref{Zp expanded}, \eqref{claim on size of QA}, and Lemma \ref{lem: exp prec}, we find
\begin{multline*}
\E \absbb{\frac{1}{N} \sum_k X_k}^{p} \;\prec\; C_p \frac{1}{N^p} \sum_{\f k} \Psi_o^{2 \abs{[k]}} \;=\; C_p \sum_{u = 1}^p \Psi_o^{2u} \frac{1}{N^p} \sum_{\f k} \ind{\abs{[\f k]} = u}
\\
\leq\; C_p \sum_{u = 1}^p \Psi_o^{2u} N^{u - p} \;\leq\; C_p (\Psi_o + N^{-1/2})^{2p} \;\leq\; C_p \Psi_o^{2p}\,,
\end{multline*}
where in the first step we estimated the summation over the sets $A_1, \dots, A_p$ by a combinatorial factor $C_p$ depending on $p$, in the forth step we used the elementary inequality $a^n b^m \leq (a + b)^{n + m}$ for positive $a,b$, and in the last step we used \eqref{admissible Psi} and the bound $M \leq N$. Thus we have proved \eqref{final moment estimate}, from which the claim follows exactly as in the first proof of Theorem \ref{thm: averaging with Lambdao}.

What remains is the proof of \eqref{claim on size of QA}. The case $\abs{A} = 1$ (corresponding to $A = \{k\}$) follows from \eqref{Xi estimate}, exactly as in the first proof of Theorem \ref{thm: averaging with Lambdao}. To simplify notation, for the case $\abs{A} \geq 2$ we assume that $k = 1$ and $A = \{1, 2, \dots, t\}$ with $t \geq 2$. It suffices to prove that
\begin{equation} \label{claim for QA 1/G}
\absbb{Q_t \cdots Q_2 \frac{1}{G_{11}}} \;\prec\; \Psi_o^{t}\,.
\end{equation}
We start by writing, using \eqref{resolvent expansion type 1},
\begin{equation*}
Q_2 \frac{1}{G_{11}} \;=\; Q_2 \frac{1}{G_{11}^{(2)}} + Q_2 \frac{G_{12} G_{21}}{G_{11} G_{11}^{(2)} G_{22}} \;=\; Q_2 \frac{G_{12} G_{21}}{G_{11} G_{11}^{(2)} G_{22}}\,,
\end{equation*}
where the first term vanishes since $G_{11}^{(2)}$ is independent of $2$ (see Definition \ref{definition: P Q}). We now consider
\begin{equation*}
Q_3 Q_2 \frac{1}{G_{11}} \;=\; Q_2 Q_3 \frac{G_{12} G_{21}}{G_{11} G_{11}^{(2)} G_{22}}\,,
\end{equation*}
and apply \eqref{resolvent expansion type 1} with $k = 3$ to each resolvent entry on the right-hand side, and multiply everything out. The result is a sum of fractions of entries of $G$, whereby all entries in the numerator are diagonal and all entries in the denominator are diagonal. The leading order term vanishes,
\begin{equation*}
Q_2 Q_3 \frac{G_{12}^{(3)} G_{21}^{(3)}}{G_{11}^{(3)} G_{11}^{(23)} G_{22}^{(3)}} \;=\; 0\,,
\end{equation*}
so that the surviving terms have at least three (off-diagonal) resolvent entries in the numerator.
We may now continue in this manner; at each step the number of (off-diagonal) resolvent entries in the numerator increases by at least one.

More formally, we obtain a sequence $A_2, A_3, \dots, A_t$, where $A_2 \deq Q_2 \frac{G_{12} G_{21}}{G_{11} G_{11}^{(2)} G_{22}}$ and $A_{i}$ is obtained by applying \eqref{resolvent expansion type 1} with $k = i$ to each entry of $Q_i A_{i - 1}$, and keeping only the nonvanishing terms. The following properties are easy to check by induction.
\begin{enumerate}
\item
$A_i = Q_i A_{i - 1}$.
\item
$A_i$ consists of the projection $Q_2 \cdots Q_i$ applied to a sum of fractions such that all entries in the numerator are diagonal and all entries in the denominator are diagonal.
\item
The number of (off-diagonal) entries in the numerator of each term of $A_i$ is at least $i$.
\end{enumerate}
By Lemma \ref{lem: exp prec} combined with (ii) and (iii) we conclude that $\abs{A_i} \prec \Psi_o^i$. From (i) we therefore get
\begin{equation*}
Q_t \cdots Q_2 \frac{1}{G_{11}} \;=\; A_t \;=\; O_\prec(\Psi_o^t)\,.
\end{equation*}
This is \eqref{claim for QA 1/G}. Hence the proof is complete.
\end{proof}

\section{Large deviation bounds}

We consider random variables $X$ satisfying
\begin{equation} \label{cond on X}
\E X \;=\; 0\,, \qquad \E \abs{X}^2 \;=\; 1\,, \qquad \p{\E \abs{X}^p}^{1/p} \;\leq\; \mu_p
\end{equation}
for all $p \in \N$ and some constants $\mu_p$.

\begin{theorem}[Large deviation bounds] \label{thm: LDE}

Let $\pb{X_i^{(N)}}$ and $\pb{Y_i^{(N)}}$ be independent families of random variables and 
$\pb{a_{ij}^{(N)}}$ and $\pb{b_i^{(N)}}$ be deterministic; here $N \in \N$ and $i,j = 1, \dots, N$. Suppose that all entries $X_i^{(N)}$ and $Y_i^{(N)}$ are independent and satisfy \eqref{cond on X}. Then we have the bounds
\begin{align} \label{lde 1}
\sum_i b_i X_i &\;\prec\; \pbb{\sum_i \abs{b_i}^2}^{1/2}\,,
\\ \label{lde 2}
\sum_{i,j} a_{ij} X_i Y_j &\;\prec\; \pbb{\sum_{i,j} \abs{a_{ij}}^2}^{1/2}\,,
\\ \label{lde 3}
\sum_{i \neq j} a_{ij} X_i X_j &\;\prec\; \pbb{\sum_{i \neq j} \abs{a_{ij}}^2}^{1/2}\,.
\end{align}
 If the coefficients $a_{ij}^{(N)}$ and $b_i^{(N)}$
depend on an additional parameter $u$, then all of these estimates are uniform in $u$ (see Definition \ref{def:stocdom}), i.e.\ the threshold $N_0= N_0(\e, D)$ in the definition of $\prec$ depends only on the family $\mu_p$ from \eqref{cond on X} and $\delta$ from \eqref{lower bound on W}; in particular, $N_0$ does not depend on $u$. 
\end{theorem}

\begin{proof}
The estimates \eqref{lde 1}, \eqref{lde 2}, and \eqref{lde 3} follow from Lemmas B.2, B.3, and B.4 of \cite{EKYY3}, combined with Chebyshev's inequality.
\end{proof}

\end{document}